\documentclass[reqno,11pt]{amsart}

\usepackage{amsmath} 
\usepackage{amssymb} 
\usepackage{amsfonts}
\usepackage{mathrsfs}
\usepackage{amsthm} 
\usepackage[all]{xy}
\usepackage{vmargin}       
\usepackage{lineno,hyperref} 

\renewcommand\labelenumi{(\roman{enumi})}
\renewcommand\theenumi\labelenumi

\setmarginsrb{2cm}{1.5cm}{2cm}{3cm}{3cm}{1cm}{5cm}{1cm}     

\newcommand{\Z}{\mathbb{Z}}

\newcommand{\R}{\mathbb{R}}
\newcommand{\C}{\mathbb{C}}

\newcommand{\rad}{\operatorname{rad}}

\newcommand{\m}{\operatorname{m}}

\newcommand{\rank}{\operatorname{rank}}

\newcommand{\Hom}{\operatorname{Hom}}

\newcommand{\Lie}{\operatorname{Lie}}

\def\bigquotient#1#2{%
    \left.\raise1ex\hbox{$#1$}\middle/\lower1ex\hbox{$#2$}\right.%
}
\def\quotient#1#2{%
    \left.\raise0.3ex\hbox{$#1$} \middle/\lower0.3ex\hbox{$#2$}\right.%
}

\newtheoremstyle{thme}
{2}						
{2}						
{\itshape}				
{}						
{\bfseries}				
{}						
{\newline}				
{}						

\newtheoremstyle{rem}	
{2}						
{2}						
{\upshape}				
{}						
{\bfseries}				
{}						
{\newline}				
{}						

\theoremstyle{thme}
\newtheorem{thmintro}{Theorem}
 
\newtheorem{corolintro}[thmintro]{Corollary}

\newtheorem{thm}{Theorem}[section]
\newtheorem*{thm*}{Theorem}

\newtheorem{prop}[thm]{Proposition}

\newtheorem{lem}[thm]{Lemma}

\theoremstyle{rem}
\newtheorem*{remarkintro}{Remark}

\newtheorem{remark}[thm]{Remark}

\newtheorem{defin}[thm]{Definition}

\begin{document}    


\title{Restriction of irreducible modules for $\mbox{Spin}_{2n+1}(K)$ to $\mbox{Spin}_{2n}(K)$}
\author{Mika\"el Cavallin}
\address{Fachbereich Mathematik, Postfach 3049, 67653 Kaiserslautern, Germany.} 
\email{cavallin.mikael@gmail.com} 
\thanks{The author would like to acknowledge the support of the Swiss National Science Foundation through grant no. 20020-135144 as well as the ERC Advanced Grant through grant no. 291512.}

\begin{abstract}     

Let \(K\) be an algebraically closed field of characteristic \(p\geqslant 0\) and let $Y=\mbox{Spin}_{2n+1}(K)$ $(n\geqslant 3)$ be a simply connected simple algebraic group of type $B_n$  over \(K.\) Also let \(X\) be the subgroup  of type \(D_n,\) embedded in \(Y\) in the usual way, as the derived subgroup of the stabilizer of a non-singular one-dimensional subspace of the natural module for \(Y.\) In this paper, we give a complete set of isomorphism classes of finite-dimensional, irreducible, rational \(KY\)-modules on which  \(X\) acts with exactly two composition factors, completing the work of Ford in \cite{Fo1}.
\end{abstract}

\maketitle


\section{Introduction}     

In the $1950$s, Dynkin  \cite{Dynkin} determined the maximal closed connected subgroups of the classical algebraic groups over $\C.$ The most difficult part of the investigation concerned irreducible closed simple subgroups $X$ of $\mbox{SL}(V).$ In the course of his analysis, Dynkin observed that if $X$ is a simple algebraic group over $\C$ and if $\phi : X \to \mbox{SL}(V)$ is an irreducible rational representation, then with specified exceptions the image of $X$ is maximal among closed connected subgroups in one of the classical groups $\mbox{SL}(V),$ $\mbox{Sp}(V)$ or $\mbox{SO}(V).$  Here Dynkin determined the triples $(Y,X,V)$ where $Y$ is a closed connected subgroup of $\mbox{SL}(V),$ $V$ is an irreducible $KY$-module different from the natural module for $Y$ or its dual, and $X$ is a closed connected subgroup of $Y$ such that the restriction of $V$ to $X,$ written $V|_X,$ is also irreducible. Such triples (as well as similar triples with $X$ a closed \emph{disconnected} subgroup of $Y$) shall be referred to as \emph{irreducible triples} in the remainder of the paper.
\vspace{5mm}

In the $1980$s, Seitz \cite{Se} extended the analysis to the situation of fields of arbitrary characteristic. By introducing new techniques, he determined all irreducible triples $(Y,X,V)$ where $Y$ is a simply connected simple algebraic group of classical type over an algebraically closed field $K$ having characteristic $p\geqslant 0,$ $X$ is a closed connected proper subgroup of $Y,$ and $V$ is an irreducible, tensor indecomposable $KY$-module. In particular, in the situation where $Y=\mbox{Spin}_{2n+1}(K)$ $(n\geqslant 3)$ and $X$ is a subgroup of type \(D_n\), embedded in the usual way (as the derived subgroup of the stabilizer of a non-singular one-dimensional subspace of the natural module for $Y$), Seitz showed that $V|_X$ can be irreducible only if $p=2,$ thus drastically reducing the number of possibilities (see \cite[Table 1 $(\mbox{MR}_4)$]{Se}). His investigation was then extended by Testerman \cite{Te} to exceptional algebraic groups $Y,$ again for $X$ a closed connected subgroup.
\newpage

The work of Dynkin, Seitz and Testerman provides a complete classification of irreducible triples $(Y,X,V)$ where $Y$ is a simple algebraic group, $X$ is a closed connected proper subgroup of $Y$ and $V$ is an irreducible, tensor indecomposable $KY$-module. Now focusing on disconnected subgroups, Ford in \cite{Fo1} and \cite{Fo2}, investigated irreducible triples $(Y,X,V)$ in the special case where $Y$ is of classical type, $X$ is disconnected with simple connected component $X^{\circ},$ and the restriction $V|_{X^{\circ}}$ has $p$-restricted composition factors. In particular, in \cite[Section 3]{Fo1}, Ford considered the embedding of $X=X^{\circ} \langle \theta \rangle$ in   $Y=\mbox{Spin}_{2n+1}(K)$ $(n\geqslant 3),$ where $X^{\circ}$ is a subgroup of type \(D_n\), embedded in the usual way, and $\theta$ denotes the standard graph automorphism of order $2$ of $X^{\circ}.$ Contrary to the connected case, a whole family of irreducible triples arises in this situation (see \cite[Theorem 1, Table II, No. U$_2$]{Fo1} for a complete list).
\vspace{5mm}

For completeness, observe that more recently, Ghandour \cite{Gh} gave a complete classification of irreducible triples $(Y,X,V)$ in the case where $Y$ is a simple algebraic group of exceptional type, $X$ is a closed, disconnected, positive-dimensional subgroup of $Y$ and $V$ is an irreducible, $p$-restricted, rational $KY$-module. Finally, in \cite{BGMT}, \cite{BGT} and \cite{BMT}, the authors essentially treat the case of triples $(Y,X,V)$ where $Y$ is of classical type, $X$ is a closed, positive-dimensional subgroups of $Y,$ and $V$ is an irreducible, tensor indecomposable $KY$-module, removing the previously mentioned assumption of Ford.
\vspace{5mm}

In this paper, part of the author's doctoral thesis, we initiate the investigation of  triples $(Y,X,V),$ where $Y$ is a simply connected simple algebraic group of classical type over $K,$ $X$ is a maximal, closed, connected subgroup of $Y,$ and $V$ is a finite-dimensional, irreducible, rational $KY$-module on which $X$ acts with \emph{exactly two} composition factors. Not only is this problem a natural follow-up to the study  of irreducible triples initiated by Dynkin, but also identifying such triples can  sometimes provide partial information on the structure of certain  Weyl modules for $X,$ as well as insight on their corresponding irreducible quotients (see \cite[Theorem 5]{Thesis} and \cite[Corollary 6]{Thesis}, for instance).

\subsection{Statements of results}     

Let $Y=\mbox{Spin}_{2n+1}(K)$ $(n\geqslant 3)$ be a simply connected simple algebraic group of type $B_n$  over \(K\) and let \(X\) be a subgroup of type \(D_n,\) embedded in \(Y\) in the usual way, as the derived subgroup of the stabilizer of a non-singular one-dimensional subspace of the natural module for \(Y.\) Fix $B_Y$ a Borel subgroup of $Y$ containing a maximal torus $T_Y$ of $Y,$ and let $B_X\subset B_Y$ be a Borel subgroup of $X$ containing the fixed maximal torus $T_Y$ of $X.$ Finally, denote by   $\{\lambda_1,\ldots,\lambda_n\}$ the set of fundamental weights associated to $B_Y$ and by $\{\omega_1,\ldots,\omega_n\}$ the set of fundamental weights associated to $B_X,$ where we adopt the usual labelling of the Dynkin diagram of \(Y\) and \(X\) (see Bourbaki \cite{Bourbaki}). We refer the reader to Section \ref{[Section] Preliminaries}  for more notation.
\vspace{5mm}

As seen above, there are no irreducible triples $(Y,X,V)$ if \(p\neq 2,\) so it is  natural to consider the next best possible scenario, that is, the situation where $V$ is a finite-dimensional, irreducible, rational  $KY$-module  such that $V|_X$ has exactly two composition factors. Now \cite[Theorem 3.3]{Fo1} already provides us with a whole family of examples of  such irreducible modules for $Y,$ whose two $KX$-composition factors are  interchanged by the aforementioned graph automorphism $\theta$ of $X.$ The aim of this paper is to  give a proof of the following generalization of Ford's result, which consists in  a complete  classification of triples $(Y,X,V)$ for which $X$ has exactly two composition factors on $V.$

\begin{thmintro}\label{Main result}
Let $Y$ be a simply connected simple algebraic group of type $B_n$ over $K$ $(n\geqslant 3),$ and let $X$ be a simple subgroup of type $D_n,$ embedded in $Y$ in the usual way. Let $V=L_Y(\lambda)$ be an irreducible non-trivial $KY$-module having $p$-restricted highest weight $\lambda=\sum_{r=1}^n{a_r\lambda_r}\in X^+(T_Y),$ and if $\lambda\notin \Z  \lambda_n,$ let $1\leqslant k<n$ be maximal such that $a_k  \neq 0.$ Then $X$ has exactly two composition factors on $V$ if and only if one of the following holds.
\begin{enumerate}
\item \label{[Dn<Bn] Theorem: the case lambda=lambda_i for 0<i<n} 	$\lambda=\lambda_k$ and $p\neq 2.$
\item \label{[Dn<Bn] Theorem: the case lambda=lambda_n} $\lambda=\lambda_n.$
\item \label{[Dn<Bn] Theorem: Remaining cases}											$\lambda$ is neither as in \textnormal{\ref {[Dn<Bn] Theorem: the case lambda=lambda_i for 0<i<n}} nor \textnormal{\ref{[Dn<Bn] Theorem: the case lambda=lambda_n}}, \(a_n\leqslant 1,\) and the following divisibility conditions are satisfied.
					\begin{enumerate}
					\item \label{[Dn<Bn] Theorem: Remaining cases (1)}				$p\mid a_i+a_j + j - i$ for every $1\leqslant i<j<n$ such that $a_i a_j \neq 0$ and $a_r=0$ for $i<r<j.$
					\item \label{[Dn<Bn] Theorem: Remaining cases (2)}				$p\mid 2(a_n + a_k + n - k )-1.$
					\end{enumerate}
\end{enumerate}
Furthermore, if $(\lambda,p)$ is as in \textnormal{\ref{[Dn<Bn] Theorem: the case lambda=lambda_i for 0<i<n}}, \textnormal{\ref{[Dn<Bn] Theorem: the case lambda=lambda_n}}  or \textnormal{\ref{[Dn<Bn] Theorem: Remaining cases}}, then $V|_{X}$ is completely reducible.
\end{thmintro}

\begin{remarkintro}
Let $(\lambda,p)$ be as in case \ref{[Dn<Bn] Theorem: the case lambda=lambda_n}  or \ref{[Dn<Bn] Theorem: Remaining cases} of Theorem \ref{Main result}, with $a_n=1$ in the latter case. Then as seen in \cite{Fo1}, the $KX$-composition factors of $V$ have respective highest weights $\omega=\sum_{r=1}^{n-1}{a_r\omega_r}+(a_{n-1}+1)\omega_n$ and $\omega' =\omega^\theta$ (that is, $\omega'=\omega+\omega_{n-1}-\omega_n).$  If on the other hand $(\lambda,p)$ is as in case \ref{[Dn<Bn] Theorem: the case lambda=lambda_i for 0<i<n} or \ref{[Dn<Bn] Theorem: Remaining cases} of Theorem \ref{Main result}, with $a_n=0$ in the latter case, then the $KX$-composition factors of $V$ have respective highest weights $\omega=\sum_{r=1}^{k}{a_r\omega_r}+\delta_{k,n-1}a_k\omega_n$ and $\omega'=\sum_{r=1}^{k-2}{a_r\omega_r}+(a_{k-1}+1)\omega_{k-1}+(a_k-1)(\omega_k+\delta_{k,n-1}\omega_n).$
\end{remarkintro}

The second result of this paper consists in a generalization of Theorem \ref{Main result} to the case of arbitrary irreducible, finite-dimensional, rational  \(KY\)-modules. Its proof essentially relies on an application of the well-known Steinberg Tensor Product Theorem (see Theorem \ref{[Preliminaries] Weights and multiplicities: Steinberg tensor product theorem}).

\begin{corolintro}\label{main_corol}
Let $Y,$ $X$ be as in the statement of Theorem \textnormal{\ref{Main result}} and consider an irreducible non-trivial $KY$-module $V=L_Y(\delta)$ having  highest weight $\delta\in X^+(T_Y).$ Then $X$ has exactly two composition factors on $V$ if and only if  one of the following holds.
\begin{enumerate}
\item \label{main_corol1} \(\delta=p^r\lambda\) for some  \(r\in \mathbb{Z}_{\geqslant 0}\) and some \(\lambda\in X^+(T_Y)\) as in case  \textnormal{\ref{[Dn<Bn] Theorem: the case lambda=lambda_i for 0<i<n}}, \textnormal{\ref{[Dn<Bn] Theorem: the case lambda=lambda_n}},  or \textnormal{\ref{[Dn<Bn] Theorem: Remaining cases}} of Theorem \textnormal{\ref{Main result}}. 
\item \label{main_corol2} \(p=2\) and \(\delta=\delta_0+2\delta_1+\cdots+2^r\delta_r\) for some \(r\in \mathbb{Z}_{> 0}\) and $p$-restricted \(\delta_0,\ldots,\delta_r\in X^+(T_Y)\) such that $\#\{1\leqslant j\leqslant r: \langle \delta_j,\alpha_n\rangle \neq 0\} =\#\{1\leqslant j\leqslant r: \delta_j=\lambda_n\} =1.$
\end{enumerate}
\end{corolintro}

Our proof of Theorem \ref{Main result} is essentially based on the method introduced by Ford in \cite[Section 3]{Fo1}: let $V$ be a finite-dimensional, irreducible, rational $KY$-module having highest weight  $\lambda=\sum_{r=1}^{n-1}{a_r\lambda_r},$ and denote by $v^+$ a maximal vector in $V$ for $B_Y.$ (Observe that the case where $a_n\neq 0$ is dealt with in \cite[Theorem 3.3]{Fo1}.) If $p\neq 2,$ then $X$ acts reducibly on $V$ by \cite[Table 1 $(\mbox{MR}_4)$]{Se}, yielding the existence of a second maximal vector $w^+\notin \langle v^+\rangle_K $ in $V$ for $B_X.$ The main idea behind  the proof consists in determining the condition(s) on $(V,p)$ under which 
\begin{equation*}
V=\Lie(X)v^+ \oplus \Lie(Y)w^+,
\end{equation*}
where $\Lie(X)$ and $\Lie(Y)$ denote the Lie algebras of $X$ and $Y,$ respectively. A result analogous to \cite[Lemma 3.4]{Fo1}, namely Proposition \ref{[Dn<Bn] Proof of the main result: Key Proposition}, then suggests taking a closer look at generating sets for certain weight spaces in $V,$ leading to the detailed  investigation  recorded in the appendices of the paper.

One of the main difficulties in adjusting Ford's arguments to fit our situation resides in the new weight spaces that need to be understood. We deal with  the latter as follows: first we start by computing their dimension  in the corresponding  Weyl module $V_Y(\lambda)$ of weight $\lambda$ for $Y,$ using the results in \cite{Cavallin}. An application of Proposition \ref{A simplification for V(lambda)mu} then yields bases for such weight spaces in $V_Y(\lambda).$ Finally, we deduce  generating sets in $V$ by explicitly working  with the action of the Lie algebra on the quotient of $V_Y(\lambda)$ by its unique maximal proper submodule. 
\vspace{5mm}

Results on generating sets for weight spaces are interesting in their own rights, since applicable in various areas of representation theory. For instance in \cite{Fo3}, Ford applied his methods (more precisely, those used in the proof of \cite[Proposition 3.1]{Fo1}) to the study of symmetric groups in order to prove a conjecture formulated by Jantzen and Seitz in \cite{JaSe}, equivalent to the Mullineux conjecture. We refer the reader to  \cite{Mullineux} for a proof of the latter. Furthermore, the fact that it was possible to generalize Ford's argument suggests a method of investigating branching problems  for $Y$ and $X$ in general, as well as a way of studying other embeddings, like the usual embedding of $\mbox{Spin}_{2n+1}(K)$ in $\mbox{Spin}_{2n+2}(K),$ for instance.

\subsection{Organization of the paper}     

Finally, let us make some comments on the overall organization of the paper. In Section \ref{[Section] Preliminaries}, we fix the notation that will be used for the rest of the paper and then we present several preliminary results which will be needed in the proof of Theorem \ref{Main result}. In particular, we recall some basic facts on representations of (semi)simple algebraic groups and Lie algebras, as well as standard results concerning weights and their multiplicities. (We refer the reader to Table \ref{Some_weight_multiplicities_in_Weyl_modules} for a survey of certain  weight multiplicities in characteristic zero.) In Section \ref{[Section] proof of the main result}, we give a proof of Theorem \ref{Main result} and Corollary \ref{main_corol}.  Finally, the study of the structure of various Weyl modules and their irreducible quotients can be found in Appendices \ref{Weight spaces for G of type Al} and \ref{Weight spaces for G of type Bn}. In particular, we give explicit generating sets for certain weight spaces in the latter modules. The description of such sets plays a major role in the proof of Theorem \ref{Main result}, but may also be of independent interest, so we refer the reader to Propositions \ref{Conditions_for_f(i,j)_to_belong_to_V(i,j)}, \ref{Conditions_for_f(k,n)_to_belong_to_V(k,n)}, \ref{basis_of_W_lambda-alpha_i-2alpha_k-...-2alpha_n}, \ref{condition_if_l=k-1}  \ref{condition_if_l<k-1}, and Theorem \ref{How_to_relate_divisibility_conditions_and_generators_for_V_i,j} for a summary of the corresponding investigations.

\section{Preliminaries}\label{[Section] Preliminaries}     

In this section, we recall some elementary properties of simple Lie algebras, representations of simple algebraic groups,  weights and their multiplicities, and finally conclude with a brief investigation of generating sets for weight spaces in Weyl modules. Unless specified otherwise, most results can be found in \cite{Bourbaki}, \cite{St1}, \cite{Humphreys2}, \cite{Humphreys1}, or \cite{Ca}.

\subsection{Notation}\label{Notation}     

We first fix some notation that will be used for the rest of the paper. Let $K$ be an algebraically closed field of characteristic $p\geqslant 0$  and let $G$ be a simply connected simple algebraic group over $K.$ Also fix a Borel subgroup $B=UT$ of $G,$ where $T$ is a maximal torus of $G$ and $U$ denotes the unipotent radical of $B.$ Let $n=\rank G = \dim T$ and let $\Pi=\{\alpha_1,\ldots,\alpha_n\}$ be a corresponding base of the root system $\Phi=\Phi^+\sqcup \Phi^-$ of $G,$ where $\Phi^+$ and $\Phi^-$ denote the sets of positive and negative roots of $G,$ respectively.  Let
\[
X(T)=\Hom (T,K^*) 
\]
denote the character group of $T$ and write $(-,-)$ for the usual inner product on $X(T)_{\R}=X(T) \otimes_{\Z} \R.$ Recall the existence of a partial ordering on \(X(T)_{\R},\) defined by \(\mu\preccurlyeq \lambda\) if and only if \(\lambda-\mu \in \Gamma, \) where \(\Gamma\) denotes the monoid of \(\mathbb{Z}_{\geqslant 0}\)-linear combinations of simple roots. (We also  write \(\mu\prec \lambda\) to indicate that \(\mu\preccurlyeq \lambda\) and  \(\mu \neq \lambda.\)) 

In addition, let $\{\lambda_1,\ldots,\lambda_n\}$ be the set of fundamental weights for $T$ corresponding to our choice of base $\Pi,$ that is $\langle \lambda_i,\alpha_j\rangle =\delta_{ij}$ for every $1\leqslant i,j\leqslant n,$ where \(\langle \lambda, \alpha\rangle = 2(\lambda,\alpha) (\alpha,\alpha)^{-1}\) for   $\lambda,\alpha$ with \(\alpha \neq 0.\) Set
\[
X^+(T)=\{\lambda\in X(T): \langle \lambda,\alpha_r\rangle \geqslant 0 \mbox{ for every } 1\leqslant r\leqslant n\}
\]
and call a character $\lambda \in X^+(T)$ a \emph{dominant character}. Every such character can be written as a $\Z_{\geqslant 0}$-linear combination  $\lambda=\sum_{r=1}^n{a_r\lambda_r}.$ In addition, for $\alpha\in \Phi,$ define the reflection $s_{\alpha}:X(T)_{\R} \to X(T)_{\R}$ relative to $\alpha$ by $s_{\alpha}(\lambda)=\lambda-\langle \lambda,\alpha\rangle\alpha,$ this for all $\lambda\in X(T)_{\R},$ and denote by $\mathscr{W}$ the finite group $\langle s_{\alpha_r}:1\leqslant r\leqslant n\rangle,$ called the \emph{Weyl group of $G.$} Finally, we use the notation 
\[
U_{\alpha}=\{x_{\alpha}(c):c\in K\}
\] 
to denote the root subgroup of \(G\) corresponding to the root \(\alpha \in \Phi,\) (i.e. \(x_{\alpha}:K\to G\) is an injective morphism of algebraic groups such that \(tx_{\alpha}(c)t^{-1}=x_{\alpha}(\alpha(t)c)\) for all \(t\in T\) and \(c\in K\)) and we write \(\Lie(G)\) for the Lie algebra of \(G.\)

\subsection{Chevalley bases and structure constants}     
\label{structure constants and Chevalley basis}          

In this section, we recall some elementary facts on Chevalley bases for simple Lie algebras and on the so-called structure constants. In addition, we fix signs for the latter in types $A$ and $B.$ Most of the results presented here can be found in \cite[Chapter VII]{Humphreys1} or \cite[Chapter 4]{Ca}. Let  $\mathfrak{g}$ be a finite-dimensional simple Lie algebra over $\mathbb{C}$ with Cartan subalgebra \(\mathfrak{h}.\) Set \(n=\rank \mathfrak{g}\) and fix an ordered base $\Pi=\{\alpha_1,\ldots,\alpha_n\}$  of the corresponding root system $\Phi=\Phi^+\sqcup \Phi^{-}$ of $\mathfrak{g},$ where \(\Phi^+\) and \(\Phi^-\)  denote the sets of positive and negative roots of \(\Phi,\) respectively. To each root \(\alpha\in \Phi\) corresponds a $1$-dimensional root space \(\mathfrak{g}_\alpha \) defined by 
\[
\mathfrak{g}_\alpha=\{x\in \mathfrak{g}:[h,x]=\alpha(h)x \mbox{ for all } h\in \mathfrak{h}\}.
\]

It is often easier to work with a basis \(\{f_{\alpha},h_{\alpha_r},e_\alpha:\alpha \in \Phi^+,1\leqslant r\leqslant n\}\) for \(\mathfrak{g},\) where \(e_{\alpha}\in \mathfrak{g}_\alpha, \)  \(f_{\alpha}=e_{-\alpha}\in \mathfrak{g}_{-\alpha}\)   are root vectors for \(\alpha \in \Phi^+\) and \([e_{\alpha_r},f_{\alpha_r}]=h_{\alpha_r}\) for   \(1\leqslant r\leqslant n.\) Such a basis can be chosen in many ways: throughout this paper, we fix a \emph{standard Chevalley basis}
\[
\mathscr{B}=\{f_{\alpha},h_{\alpha_r},e_\alpha: \alpha\in \Phi^+, 1\leqslant r\leqslant n\} 
\]
for any simple Lie algebra $\mathfrak{g},$ whose elements satisfy the usual relations (see \cite[Theorem 4.2.1]{Ca}, for example). In particular, for all $\alpha,\beta \in \Phi$ such that $\alpha+\beta\in \Phi,$ we have 
\begin{equation}
[e_{\alpha},e_{\beta}]  = N_{(\alpha,\beta)}e_{\alpha+\beta}=\pm (q+1)e_{\alpha+\beta},
\label{structure_constants_relation_0}
\end{equation}
where $q$ is the greatest integer for which $\alpha-q\beta \in \Phi.$ The $N_{(\alpha,\beta)}$ (\(\alpha,\beta\in \Phi\)) are called the \emph{structure constants.} One can easily check that for any pair of roots $(\alpha,\beta),$ we have
\begin{equation}
N_{(\beta,\alpha)} = -N_{(\alpha,\beta)}=N_{(-\alpha,-\beta)},
\label{structure_constants_relation_1}
\end{equation}
and using the Jacobi identity of $\mathfrak{g},$ one can prove (see  \cite[Theorem 4.1.2 (ii)]{Ca}, for example) that for every $\alpha,\beta,\gamma \in \Phi$ satisfying $\alpha+\beta+\gamma=0,$ we have
\begin{equation}
\frac{N_{(\alpha,\beta)}}{(\gamma,\gamma)}  =\frac{N_{(\beta,\gamma)}}{(\alpha,\alpha)} = \frac{N_{(\gamma,\alpha)}}{(\beta,\beta)}.
\label{structure_constants_relation_2}
\end{equation}
Finally, one can show (e.g. \cite[Theorem 4.1.2 (iv)]{Ca}) that for every $\alpha,\beta,\gamma,\delta \in \Phi$ such that $\alpha+\beta+\gamma+\delta=0$ and no pair are opposite, we have
\begin{equation}
\frac{N_{(\alpha,\beta)}N_{(\gamma,\delta)}}{(\alpha+\beta,\alpha+\beta)} + \frac{N_{(\gamma,\alpha)}N_{(\beta,\delta)}}{(\gamma+\alpha, \gamma+\alpha)} + \frac{N_{(\beta,\gamma)}N_{(\alpha,\delta)}}{(\beta +\gamma,\beta+\gamma)}=0.
\label{structure_constants_relation_3}
\end{equation}

It turns out that the choice of a sign  in \eqref{structure_constants_for_B_part1} for certain well-chosen ordered pairs    of positive roots (called \emph{extraspecial}) completely determines the signs of the remaining constant structures. 
As in \cite[Section 2.1]{Ca}, we thus fix an ordering \(\leqslant\) on $\Phi^+$ by saying that $\alpha \leqslant \beta$ if either $\alpha=\beta$ or if there exists $1\leqslant m\leqslant n$ and $c_1,\ldots,c_m\in \mathbb{Z}$ such that  $\beta-\alpha=\sum_{r=1}^m{c_r\alpha_r},$ with $c_m>0.$ We shall also write $\alpha < \beta$ if  $\alpha\leqslant \beta$ and $\alpha \neq \beta.$

\begin{defin}\label{Definition: Extraspecial pairs}
An ordered pair of positive roots $(\alpha,\beta)$ is \emph{special} if $\alpha+\beta\in \Phi^+$ and $ \alpha < \beta.$ Also, such a pair is \emph{extraspecial} if for all special pairs $(\gamma,\delta)$ satisfying $\gamma+\delta=\alpha+\beta,$ we have $\alpha \leqslant \gamma.$
\end{defin}

\begin{remark}\label{Uniqueness of extraspecial pairs}
In view of Definition \ref{Definition: Extraspecial pairs}, one notices that if $\gamma\in \Phi^+,$ then either $\gamma \in \Pi$ or there exist unique $\alpha,\beta\in \Phi^+$ such that $\alpha+\beta=\gamma$ and $(\alpha,\beta)$ is extraspecial.  In particular, the number of distinct extraspecial pairs equals the number of non-simple positive roots.
\end{remark}

By \cite[Proposition 4.2.2]{Ca}, the structure constants of $\mathfrak{g}$ are uniquely determined by their values on the set of extraspecial pairs, for which we can arbitrarily choose the sign in \eqref{structure_constants_relation_0}. Throughout this paper, we assume  $N_{(\alpha,\beta)} > 0$ for any extraspecial pair $(\alpha,\beta).$

\begin{lem}\label{structure_constants_for_A}
Let $\mathfrak{g}$ be a simple Lie algebra of type $A_n$ $(n\geqslant 2)$ over \(\C\) and adopt the notation above. Then the extraspecial pairs are $(\alpha_i,\alpha_{i+1}+\cdots+\alpha_j),$ where $1\leqslant i<j\leqslant n.$ Moreover $$N_{(\alpha_i+\cdots+\alpha_r, \alpha_{r+1}+\cdots +\alpha_j)}=1$$  for every $1 \leqslant i \leqslant r < j \leqslant n.$ 
\end{lem}

\begin{proof}  
Consider a pair \((\alpha,\beta)=(\alpha_i,\alpha_{i+1}+\cdots+\alpha_j)\) as in the statement of the lemma  and observe that \(\alpha+\beta\in \Phi^+\) and \(\alpha<\beta,\) so \((\alpha,\beta)\) is special. Also, if  \((\gamma,\delta)\) is another special pair satisfying  \(\alpha+\beta=\gamma+\delta,\) then there exists a unique \(i\leqslant r < j\) such that  \(\gamma=\alpha_i+\cdots +\alpha_r,\) \(\delta =\alpha_{r+1}+\cdots +\alpha_j.\) Therefore \(\gamma \geqslant \alpha,\) showing that \((\alpha,\beta)\) is extraspecial.  Now an application of Remark \ref{Uniqueness of extraspecial pairs}  shows that all extraspecial pairs are of the aforementioned form, so that the first assertion of the lemma holds. For the second assertion, we proceed by induction on \(0\leqslant r-i\leqslant n-2.\) If \(r-i=0\) (i.e. \(r=i\)), then \((\alpha_i,\alpha_{i+1}+\cdots+\alpha_j)\) is an extraspecial pair and thus the result follows from \eqref{structure_constants_relation_0} together with our assumption on the positivity of structure constants for such pairs.  We next suppose   that $1 \leqslant i<r<j\leqslant n,$ in which case applying \eqref{structure_constants_relation_3} to the roots $\alpha=\alpha_i,$ $\beta=-(\alpha_i+\cdots+\alpha_r),$ $\gamma=-(\alpha_{r+1}+\cdots+\alpha_j),$ and $\delta=\alpha_{i+1}+\cdots +\alpha_j$  yields
\[
0 = N_{(\alpha,\beta)}N_{(\gamma,\delta)} +  N_{(\beta,\gamma)}N_{(\alpha,\delta)},
\]
since all roots in \(\Phi\) have equal length and \(\gamma+\alpha \) is not a root. Now  $N_{(\beta,\gamma)} = - N_{(\alpha_i+\cdots+\alpha_r, \alpha_{r+1}+\cdots +\alpha_j)} $  by \eqref{structure_constants_relation_1},  while by the $r=i$ case, we know that $N_{(\alpha,\delta)}=1.$ Finally, by \eqref{structure_constants_relation_1}, \eqref{structure_constants_relation_2}, and the $r=i$ case again, we get $N_{(\alpha,\beta)}=-1$ and $N_{(\gamma,\delta)}= N_{({\alpha_{r+1}+\cdots+\alpha_j,\alpha_{i+1}+\cdots+\alpha_r)}},$ so that $$N_{(\alpha_i+\cdots+\alpha_r, \alpha_{r+1}+\cdots +\alpha_j)} = N_{(\alpha_{i+1} + \cdots + \alpha_r, \alpha_{r+1}+\cdots +\alpha_j)}.$$ The result then follows by induction.
\end{proof}

We conclude this section by considering a simple Lie algebra \(\mathfrak{g}\) of type $B_n$ over \(\C\) $(n\geqslant 3).$ Here we leave to the reader to check that the extraspecial pairs $(\alpha,\beta)$ are  as in Table \ref{Extraspecial pairs (type B)}. (In view of Remark \ref{Uniqueness of extraspecial pairs}, it suffices to show that each pair \((\alpha,\beta)\) appearing in the table is extraspecial,  as in the proof of Lemma \ref{structure_constants_for_A}.)
\begin{table}[h]
\center
\begin{tabular}{ccc}
 \hline \hline \\[-2ex] 
 $\alpha$				&				$\beta$															&	$\mbox{Conditions}$				\\  
 \hline \hline \\ [-1.5ex]
$\alpha_i$ 				&				$\alpha_{i+1}+\cdots+\alpha_j$ 									& 	$1\leqslant i<j\leqslant n$				\cr
$\alpha_i$				&				$\alpha_{i+1}+\cdots+\alpha_k+2\alpha_{k+1}+\cdots+2\alpha_n$ 	&	$1\leqslant i<k<n$			 				\cr
$\alpha_i$			&				$\alpha_{i-1}+\alpha_{i}+2\alpha_{i+1} +\cdots+2\alpha_n$			&	$1 < i <n$	\cr
$\alpha_n$			&				$\alpha_{n-1}+ \alpha_n$			&	$ $							\\[0.2cm]
\hline\cr
\end{tabular}
\caption{Extraspecial pairs $(\alpha,\beta)$ for $\Phi$ of type $B_n$ $(n\geqslant 3).$}
\label{Extraspecial pairs (type B)}
\end{table}

\begin{lem}\label{structure_constants_for_B}
Let $\mathfrak{g}$ be a simple Lie algebra of type $B_n$ over \(\C\) $(n\geqslant 3) $ and adopt the notation above. Then the following assertions hold.
\begin{enumerate} 
\item \label{structure_constants_for_B_part1} $N_{(\alpha_i+\cdots+\alpha_r, \alpha_{r+1}+\cdots +\alpha_j)}=1,$ $1\leqslant i\leqslant r<j\leqslant n.$
\item \label{structure_constants_for_B_part2} $N_{(\alpha_i+\cdots +\alpha_r, \alpha_{r+1}+\cdots+\alpha_j+2\alpha_{j+1}+\cdots+2\alpha_n)}=1,$ $1\leqslant i\leqslant r<j<n.$
\item \label{structure_constants_for_B_part3} $N_{(\alpha_{r}+\cdots +\alpha_j,\alpha_i+\cdots+\alpha_j+2\alpha_{j+1}+\cdots+2\alpha_n)}=1,$ $1\leqslant i<r\leqslant j <n.$
\item \label{structure_constants_for_B_part4} $N_{(\alpha_j+\cdots+\alpha_n,\alpha_r+\cdots +\alpha_n)}=2,$ $1\leqslant r<j\leqslant n.$
\item \label{structure_constants_for_B_part5} $N_{(\alpha_{i}+\cdots + \alpha_j,\alpha_{r+1}+\cdots+\alpha_j+2\alpha_{j+1}+\cdots+2\alpha_n)}=-1,$ $1\leqslant i \leqslant r <j <n.$
\end{enumerate}
\end{lem}

\begin{proof}
We proceed as in the proof of Lemma \ref{structure_constants_for_A}, applying \eqref{structure_constants_relation_3} to well-chosen quadruples of roots \((\alpha,\beta,\gamma,\delta).\) We refer the reader to the proof Lemma \ref{structure_constants_for_A} for \ref{structure_constants_for_B_part1}, which can be dealt with in the exact same fashion. Next  \ref{structure_constants_for_B_part2} obviously holds in the situation where \(i=r,\) since in this case the considered pair of roots is extraspecial. If on the other hand \(i<r,\) then arguing by induction on \(r-i\) and applying \eqref{structure_constants_relation_3} to the roots  \(\alpha=\alpha_i,\) \(\beta=-(\alpha_i+\cdots+\alpha_r),\) \(\gamma=-(\alpha_{r+1}+\cdots+\alpha_j+2\alpha_{j+1}+\cdots+2\alpha_n),\)  and \(\delta=\alpha_{i+1}+\cdots+\alpha_j+2\alpha_{j+1}+\cdots+2\alpha_n\) yields the desired result. Part \ref{structure_constants_for_B_part3} can be dealt with in two steps: firstly, one shows using induction on \(j-i\) (where again the considered pair is extraspecial in the base case \(i=j-1\)) that the desired assertion holds in the case where \(r=j.\) Here \eqref{structure_constants_relation_3} is applied to the roots \(\alpha=\alpha_i,\) \(\beta=-\alpha_j,\) \(\gamma=-(\alpha_i+\cdots+\alpha_j+2\alpha_{j+1}+\cdots+2\alpha_n),\) and \(\delta=\alpha_{i+1}+\cdots+\alpha_{j-1}+2\alpha_j+\cdots+2\alpha_n.\) Secondly, arguing by induction on \(j-r\) and  applying   \eqref{structure_constants_relation_3} to the roots \(\alpha=\alpha_r,\) \(\beta=-(\alpha_r+\cdots+\alpha_j),\) \(\gamma=-(\alpha_i+\cdots+\alpha_j+2\alpha_{j+1}+\cdots+2\alpha_n),\) \(\delta=\alpha_i+\cdots+\alpha_r+2\alpha_{r+1}+\cdots+2\alpha_n\)  yields the result. Similarly, one proceeds in two steps in order to show that \ref{structure_constants_for_B_part4} holds. Firstly, one shows that \(N_{(\alpha_n,\alpha_r+ \cdots + \alpha_n)}=2\) for \(1\leqslant r<n,\) arguing by induction on \(n-r\) (the case where \(r=n-1\) directly following from our assumption on structure constants for extraspecial pairs). In order to do so, \eqref{structure_constants_relation_3} is applied to the roots \(\alpha=\alpha_r,\) \(\beta=-\alpha_n,\) \(\gamma=-(\alpha_r+\cdots+\alpha_n),\) \(\delta=\alpha_{r+1}+\cdots+\alpha_{n-1}+2\alpha_n.\) One then easily concludes that the assertion holds in general by applying \eqref{structure_constants_relation_3} to \(\alpha=\alpha_j,\) \(\beta=-(\alpha_j+\cdots+\alpha_n),\) \(\gamma=-(\alpha_r+\cdots+\alpha_n),\) \(\delta= \alpha_r+\cdots+\alpha_j+2\alpha_{j+1}+\cdots+2\alpha_n.\) Finally, \ref{structure_constants_for_B_part5} is dealt with in the three following steps.

\begin{enumerate}
\item One first shows that the desired assertion holds in the case where \(r=i\) and \(j=i+1\) (that is, when the considered pair is \((\alpha_i+\alpha_{i+1},\alpha_{i+1}+2\alpha_{i+2}+\cdots+2\alpha_n)\) for some \(1\leqslant i<n-1).\) As usual, applying \eqref{structure_constants_relation_3} to the positive roots  \(\alpha=\alpha_{i+1},\) \(\beta=-\alpha_i-\alpha_{i+1},\) \(\gamma=-(\alpha_{i+1}+2\alpha_{i+2}+\cdots+2\alpha_n),\) \(\delta=\alpha_i+\alpha_{i+1}+2\alpha_{i+2}+\cdots+2\alpha_n\) allows one to conclude.
\item One next deals with the situation where \(r=i,\) but \(i+1<j<n\) is arbitrary (that is, when the considered pair is \((\alpha_i + \cdots +\alpha_j,\alpha_{i+1}+\cdots+ \alpha_{j} +2\alpha_{j+1}+\cdots+2\alpha_n)\) for some \(1\leqslant i<j-1<n-1).\) Arguing by induction on \(1<j-i,\) one then concludes thanks to an application of  \eqref{structure_constants_relation_3} to the positive roots  \(\alpha=\alpha_i+\alpha_{i+1},\) \(\beta=-(\alpha_i+\cdots+\alpha_{j}),\) \(\gamma=-(\alpha_{i+1}+\cdots +\alpha_j +2\alpha_{j+1}+\cdots+2\alpha_n),\) \(\delta=\alpha_{i+1}+2\alpha_{i+2}+ \cdots+2\alpha_n.\)
\item Finally, one deals with the general case by applying \eqref{structure_constants_relation_3} to   \(\alpha=\alpha_i,\) \(\beta=-(\alpha_i+\cdots+\alpha_j),\) \(\gamma=-(\alpha_{r+1}+\cdots +\alpha_j +2\alpha_{j+1}+\cdots+2\alpha_n),\) \(\delta=\alpha_{i+1}+ \cdots + \alpha_r + 2\alpha_{r+1}+ \cdots+2\alpha_n,\) together with an inductive argument. 
\end{enumerate} 
The reader can check the details, and the proof is complete. 
\end{proof}

\subsection{Rational modules}\label{Weights And Multiplities}      

In this section, we recall some elementary properties of finite-dimensional rational \(KG\)-modules, as well as elementary facts on weights and their multiplicities. Unless specified otherwise, the results recorded here can be found in \cite[Chapter XI, Section 31]{Humphreys2}.  Let \(V\)  denote a finite-dimensional, rational \(KG\)-module and for $\mu\in X(T),$ set 
\[
V_\mu=\{v\in V: t v=\mu(t) v \mbox{ for all $t\in T$}\}.
\]
A character $\mu \in X(T)$ with $V_\mu\neq 0$ is called a \emph{$T$-weight of $V$} and $V_\mu$ is said to be its corresponding \emph{weight space}. The dimension of $V_\mu$ is called the \emph{multiplicity of $\mu$ in $V$} and is denoted by $\m_V(\mu).$ Since \(V\) is finite-dimensional, it is semisimple as a \(KT\)-module, that is, can be decomposed into a direct sum of its weight spaces. Write $\Lambda(V)\subset X(T)$ to denote the set of $T$-weights of $V$ and let  
\[
\Lambda^+(V)=\Lambda(V)\cap X^+(T)
\]
denote the set of \emph{dominant} \(T\)-weights. The natural action of the Weyl group $\mathscr{W}$ of $G$ on $X(T)$ induces an action on $\Lambda(V)$ and we say that $\lambda,\mu\in X(T)$ are \emph{conjugate under the action of $\mathscr{W}$} (or $\mathscr{W}$\emph{-conjugate}) if there exists $w\in \mathscr{W}$ such that $w \lambda=\mu.$ It is well-known (see \cite[Section 13.2, Lemma A]{Humphreys1}, for example) that each weight of $V$ is $\mathscr{W}$-conjugate to a unique dominant weight in $\Lambda^+(V).$ Also, if $\lambda\in X^+(T),$ then $w\lambda \preccurlyeq \lambda$ for every $w\in \mathscr{W}.$ Finally, $\Lambda(V)$ is a union of $\mathscr{W}$-orbits and all weights in a $\mathscr{W}$-orbit have the same multiplicity.
\vspace{5mm}

A non-zero vector \(v^+\in V\) is called a \emph{maximal vector}  of weight \(\lambda\) in $V$ for $B$ if \(v^+\in V_\lambda\) and \(Bv^+= \langle v^+\rangle_K.\) Also, we say that a weight $\lambda \in \Lambda(V)$ is a \emph{highest weight} of $V$ if $\{\mu \in \Lambda(V): \lambda \prec \mu\}=\emptyset.$ In general, an arbitrary finite-dimensional, rational $KG$-module $V$ can have many distinct highest weights. However if $V$ is irreducible and $v^+\in V_{\lambda}$ is a maximal vector in $V$ for $B,$ then $V=\langle Gv^+\rangle,$ $\m_V(\lambda)=1,$ and every weight $\mu\in \Lambda(V)$ can be obtained from $\lambda$ by subtracting positive roots, so that $\lambda$ is the unique highest weight of $V.$ Reciprocally, given a dominant character $\lambda\in X^+(T),$ one can construct a finite-dimensional, irreducible, rational  $KG$-module  \(L_G(\lambda)\) with highest weight $\lambda,$ as the quotient of the \emph{Weyl module} \(V_G(\lambda)\)  by its unique maximal submodule \(\rad (\lambda),\) that is,
\[
L_G(\lambda)=\bigquotient{V_G(\lambda)}{\rad (\lambda)}.
\]
\newpage
This correspondence defines a bijection between the set \(X^+(T)\) of dominant  characters of \(T\) and the set of isomorphism classes of irreducible, finite-dimensional, rational  $KG$-modules. For an arbitrary given finite-dimensional, rational $KG$-module $V$ and $\mu\in X^+(T),$ we let $[V,L_G(\mu)]$ denote the number of times $L_G(\mu)$ appears in a composition series of $V.$
\vspace{5mm}

It is only natural to wonder whether a given irreducible $KG$-module is tensor indecomposable or not (in characteristic zero, all irreducible modules are tensor indecomposable) and a partial answer to this question is given by the following well-known result, due to Steinberg (see \cite[Theorem 1]{St2} for a proof). Here  $F:G\rightarrow G$ is a standard Frobenius morphism and for $V$ a $KG$-module, $V^{F^i}$ is the $KG$-module $V$ on which $G$ acts via $g\cdot v= F^i(g)\cdot v,$ for  $g\in G,$ $v\in V.$ Also, we say that  $\lambda$ is \emph{$p$-restricted} if $p=0$ or  if $0 \leqslant \langle \lambda, \alpha \rangle <p$ for every $\alpha\in \Pi.$

\begin{thm}[The Steinberg Tensor Product Theorem]\label{[Preliminaries] Weights and multiplicities: Steinberg tensor product theorem}
Assume $p>0$ and $G$ is simply connected. Let $\lambda\in X^+(T)$ be a dominant character. Then there exist $k\in \Z_{\geqslant 0}$ and $p$-restricted dominant characters $\mu_0,\mu_1,\ldots,\mu_k\in X^+(T)$ such that $\lambda=\sum_{r=0}^k{p^r \mu_r}$ and 
\[
L_G(\lambda) \cong L_G(\mu_0)\otimes L_G(\mu_1)^{F} \otimes \cdots \otimes L_G(\mu_k)^{F^k}.
\] 
\end{thm}

In view of Theorem \ref{[Preliminaries] Weights and multiplicities: Steinberg tensor product theorem}, the investigation of all irreducible, finite-dimensional, rational $KG$-modules is reduced to the study of the finitely many ones having $p$-restricted highest weights, on which we thus focus in the remainder of this paper, unless specified otherwise.  
\vspace{5mm}

It is well-known (see \cite[21.3]{Humphreys1}) that the set $\Lambda(\lambda)$ of weights  of $V_G(\lambda)$ is \emph{saturated} (i.e. $\mu-r\alpha\in \Lambda(\lambda)$ for every $\mu \in \Lambda(\lambda),$ $\alpha\in \Phi$ and $0\leqslant r \leqslant \langle \mu, \alpha \rangle$), containing all dominant weights under $\lambda$ together with all their $\mathscr{W}$-conjugates. Obviously $\Lambda(L_G(\lambda))\subseteq \Lambda(\lambda)$ and it turns out that the converse also holds under certain conditions on  the pair $(\Phi,p).$

\begin{thm}[Premet, \cite{Pr}]\label{[Preliminaries] Weights and multiplicities: Premet's theorem} 
Assume \((\Phi,p)\notin \{(B_n,2),(C_n,2),(F_4,2),(G_2,3)\} \) and let $\lambda\in X^+(T)$ be a $p$-restricted dominant weight. Then $\Lambda(L_G(\lambda))=\Lambda(\lambda).$
\end{thm}

Also, we refer the reader to  \cite[Proposition 1.30]{Te} for a proof of the next elementary result, used implicitly in the remainder of the paper.

\begin{lem}\label{[Preliminaries] Weights and multiplicities: multiplicity of lambda - k alpha_i} 
Let \(\lambda\in X^+(T)\) be a non-zero \(p\)-restricted dominant weight and consider \(1\leqslant i\leqslant n\) such that \(\langle \lambda, \alpha_i\rangle >0.\) Then $\m_{L_G(\lambda)}(\lambda-d\alpha_i)=\m_{V_G(\lambda)}(\lambda-d\alpha_i)=1$ for every $1\leqslant d\leqslant \langle	\lambda,\alpha_i\rangle.$  
\end{lem}

For a non-empty subset \(J\) of \(\Pi,\) set \(\Phi^{\pm}_J=\Phi^{\pm} \cap \mathbb{Z}J.\) Also define the opposite of the standard parabolic subgroup of \(G\) corresponding to \(J\) to be \(P_J=Q_J L_J,\) where \(
L_J=\langle T ,U_{\pm \alpha}:\alpha \in J\rangle\)  denotes a Levi factor of \(P_J\) with root system \(\Phi_J=\Phi^+_J\sqcup \Phi^-_J\) and \(Q_J=\langle U_{-\beta}:\beta \in \Phi^+-\Phi^+_J\rangle\) is the unipotent radical of \(P_J.\) The following result, whose proof can be found in \cite[Lemma 2.2.8]{BGT}, for example, makes it easier to compute weight multiplicities in certain situations.

\begin{lem}\label{[Preliminaries] Parabolic embeddings: restriction to Levi subgroup}
Let \(\lambda\in X^+(T)\) and let \(J\) be a non-empty subset of \(\Pi.\) Also  consider a  weight \(\mu\in \Lambda(\lambda)\) such that \(\mu=\lambda-\sum_{\alpha \in J}{c_{\alpha} \alpha},\) with $c_\alpha \in \Z_{\geqslant 0}$ for each $\alpha\in J.$ Then $\m_{L_G(\lambda)}(\mu)=\m_{L_{H}(\lambda')}(\mu'),$ where $\mu' =\mu|_{T_H},$ $\lambda'=\lambda|_{T_H}$ and $H=L_J'.$ 
\end{lem}

We conclude this section by computing the values of various weight multiplicities in certain Weyl modules for \(G\) of type \(A_n\) $(n\geqslant 2)$ or \(B_n\) $(n\geqslant 3).$ As mentioned in the introduction, the calculations shall essentially be carried out by applying the results presented in \cite{Cavallin}.

\begin{prop}\label{Various_weight_multiplicities}
Let \(G,\) \(\lambda\) and \(\mu\) be as in Table  \textnormal{\ref{Some_weight_multiplicities_in_Weyl_modules}}. Then the multiplicity \(\m_{V_G(\lambda)}(\mu)\) of \(\mu\) in \(V_G(\lambda)\) is given in the fourth column of the table.
  
\begin{table}[h]
\center
\begin{tabular}{cccc}
\hline\hline  \\[-1.5ex]
$G$			&	$\lambda$	& $\mu$	&	$\m_{V_G(\lambda)}(\mu)$	\\[0.15cm] 
\hline\hline  \\[-1.5ex]
$A_2$		& $a\lambda_1+b\lambda_2$ & $(a-2r+1)\lambda_1+(b+r-2)\lambda_2$  &$ 2$\\[0.15cm]
\hline \\[-1.5ex]
$A_n$ $(n\geqslant 3)$ & $a\lambda_1+b\lambda_n$  & $(a-2r+1)\lambda_1+ (r-1)\lambda_2+ (b-1)\lambda_n$  &$ n$ \\[0.15cm]  
\hline \\[-1.5ex]
$B_n$ $(n\geqslant 3)$ & $  a\lambda_1$ & $(a-1)\lambda_1$  &$1$\\[0.2cm] 
\\ [-1.5ex]
 & $  c\lambda_1$ & $(c-2)\lambda_1$  &$n$\\[0.15cm] 
\\ [-1.5ex]
 &  $\lambda_2$  & $0$ &$n$ \\[0.2cm]
\\ [-1.5ex]
 &  $a\lambda_1 + \lambda_i$  & $(a-1)\lambda_1+\lambda_{i-1}$ &$i(n-i+2)-1$ \\[0.1cm]
\hline\cr
\end{tabular}
\caption{Some weight multiplicities in characteristic zero. Here \(a,b\in \Z_{\geqslant 1},\) \(1\leqslant r\leqslant a,\) \(c\in \Z_{\geqslant 2},\) and \(1<i<n.\)}
\label{Some_weight_multiplicities_in_Weyl_modules}
\end{table}
\end{prop}

\begin{proof}
First assume \(G\) is of type \(A_n\) \((n\geqslant 2)\) and let \(\lambda,\mu\) be as in the first or second row of Table \ref{Some_weight_multiplicities_in_Weyl_modules}. In each case, one easily checks that \(\mu=\lambda-(r\alpha_1+\alpha_2+\cdots+\alpha_n)\) and hence an application of \cite[Proposition 1]{Cavallin} yields \(\m_{V_G(\lambda)}(\mu)= \m_{V_G(\lambda)}(\lambda-(\alpha_1+\cdots+\alpha_n)).\) The result then immediately follows from \cite[Proposition 3]{Cavallin} in this situation. In the remainder of the proof, assume \(G\) is of type \(B_n\) \((n\geqslant 3).\) If \(\lambda,\mu\in X^+(T)\) are as in the third row of Table \ref{Some_weight_multiplicities_in_Weyl_modules}, then \(\mu=\lambda-(\alpha_1+\cdots+\alpha_n)\) and applying \cite[Theorem 2]{Cavallin} with \(j=1\) (adopting the notation in the latter paper) yields 
\[
\m_{V_G(\lambda)}(\mu)=\sum_{r=2}^n{\m_{V_G(\lambda)}(\lambda-(\alpha_r+\cdots+\alpha_n))} + \m_{V_G(\lambda)}(\lambda).
\]
Since \(\lambda-(\alpha_r+\cdots+\alpha_n) \notin \Lambda(\lambda)\) for   \(2\leqslant r\leqslant n\) and  \(\m_{V_G(\lambda)}(\lambda)=1,\) the desired assertion holds. Next assume \(\lambda=c\lambda_1\) and \(\mu=(c-2)\lambda_1\) for some \(c\geqslant 2.\) Again, one easily checks that \(\mu=\lambda-2(\alpha_1+\cdots+\alpha_n)\) and thus an application of \cite[Theorem 2]{Cavallin} with \(j=1\) yields 
\begin{align*}
2\m_{V_G(\lambda)}(\mu) = \m_{V_G(\lambda)}(\lambda) &+ \sum_{r=1}^{n-1}{\m_{V_G(\lambda)}(\lambda-(\alpha_1+\cdots + \alpha_r+2\alpha_{r+1}+\cdots+2\alpha_n))}\cr
 &+ \sum_{s=1}^{n}{\m_{V_G(\lambda)}(\lambda-(\alpha_1+\cdots + \alpha_s))}.
\end{align*}
For \(1\leqslant r\leqslant n-1\) and  \(1\leqslant s\leqslant n,\) the weights  \(\lambda-(\alpha_1+\cdots + \alpha_r+2\alpha_{r+1}+\cdots+2\alpha_n)\) and   \(\lambda-(\alpha_1+\cdots + \alpha_s)\) are \(\mathscr{W}\)-conjugate to either \(\lambda\) or \((c-1)\lambda_1.\) The latter two having  multiplicity \(1\) by above, the assertion holds in this situation as well.  In the case where \(\lambda=\lambda_2\) and \(\mu=0,\) then the result immediately follows from the fact that \(V_G(\lambda)\cong \Lie(G)^*\) and hence \(\m_{V_G(\lambda)}(\mu)=\dim \Lie(T)=n.\) Finally, consider \(\lambda,\mu\) as in the last row of Table \ref{Some_weight_multiplicities_in_Weyl_modules}. As usual, one checks that \(\mu=\lambda-(\alpha_1+\cdots+\alpha_{i-1}+2\alpha_i+\cdots+2\alpha_n),\) so that applying \cite[Theorem 2]{Cavallin} with \(j=1\) yields 
\begin{align*}
\m_{V_G(\lambda)}(\mu)	=	\m_{V_G(\lambda)}(\lambda) &+ \sum_{r=i}^{n}{\m_{V_G(\lambda)}(\lambda-(\alpha_i+\cdots+\alpha_r))} \cr
&+\sum_{s=i}^{n-1}{\m_{V_G(\lambda)}(\lambda-(\alpha_i+\cdots+\alpha_s+2\alpha_{s+1}+\cdots+2\alpha_n))}\cr
&+ \m_{V_G(\lambda)}(\lambda-2(\alpha_i+\cdots+\alpha_n))\cr
&+\sum_{t=2}^{i-1}{\m_{V_G(\lambda)}(\lambda-(\alpha_t+\cdots+\alpha_{i-1}+2\alpha_i+\cdots+2\alpha_n))}.
\end{align*}   
Now observe that \(\m_{V_G(\lambda)}(\lambda-(\alpha_t+\cdots+\alpha_{i-1}+2\alpha_i+\cdots+2\alpha_n))=n-i+2\) for every \(2\leqslant t\leqslant i-1\) by above. (Consider the Levi subgroup corresponding to the simple roots $\alpha_{i-1},\ldots,\alpha_n.$) The remaining weight multiplicities being equal to \(1,\) the desired assertion follows.
\end{proof}

\subsection{Generating sets for weight spaces}     

Let \(\mathfrak{g}_\C\) denote the complex Lie algebra  having same type as \(G\) and fix  a standard Chevalley basis $\mathscr{B}=\{f_{\alpha},h_{\alpha_r},e_\alpha: \alpha\in \Phi^+, 1\leqslant r\leqslant n\}$ of $\mathfrak{g}_\C,$ as in Section \ref{structure constants and Chevalley basis}. Also let \(\mathfrak{g}_\Z=\langle \mathscr{B}\rangle_\Z\) and set \(\mathfrak{g}_K=\mathfrak{g}_\Z \otimes_\Z K,\) so that \(\mathfrak{g}_K\) is isomorphic to the Lie algebra \(\Lie(G)\) of \(G.\) For \(\lambda\in X^+(T)\) a dominant character, we denote by \(\Delta(\lambda)\) the irreducible finite-dimensional \(\mathfrak{g}_\C\)-module having highest weight \(\lambda\) and fix a maximal vector \(v^{\lambda}\in \Delta(\lambda)\) of weight \(\lambda\) for the  Borel subalgebra of \(\mathfrak{g}_\C\) corresponding to \(\Pi.\) In addition, we write \(\mathfrak{U}_\Z\) to denote the subring of the universal enveloping algebra \(\mathfrak{U}\) of \(\mathfrak{g}_\C\) defined by
\[
\mathfrak{U}_\Z = \mathbb{Z} \left [\frac{e_\alpha^r}{r !}, \frac{f_\alpha^r}{r !} : r\in \mathbb{Z}_{\geqslant 0}, \alpha\in \Phi^+\right].
\]  

It is well-known that \(V_\Z(\lambda)=\mathfrak{U}_\Z v^\lambda\) is minimal among all \(\mathfrak{U}_\Z\)-invariant lattices in \(\Delta(\lambda)\) and that the \(KG\)-module   \(V_\Z(\lambda)\otimes_\Z K\) is the Weyl module \(V_G(\lambda)\) of highest weight \(\lambda.\) The action of \(G\) on the latter is obtained by exponentiating the action of \(\mathfrak{g}_\Z\) on \(V_\Z(\lambda),\) that is, 
\[
x_\alpha(c)v=v+\sum_{r=1}^\infty{\frac{c^r}{r !}e_{\alpha}^r v},~x_{-\alpha}(c)v=v+\sum_{r=1}^\infty{\frac{c^r}{r !}f_{\alpha}^r v},
\]
for every \(c\in K,\) \(\alpha\in \Phi^+\) and  \(v\in V_G(\lambda).\) In particular, one immediately gets that \(v^{\lambda}\otimes_\Z 1\) (which we denote by \(v^{\lambda}\) again) is a maximal vector of weight \(\lambda\) in \(V_G(\lambda)\) for \(B.\) Moreover, one can deduce that for any character $\mu \in X(T),$ we have 
\begin{equation} 
V_G(\lambda)_{\mu} =\left\langle \frac{f_{\gamma_1}^{k_1}}{k_1!}\cdots \frac{f_{\gamma_r}^{k_r}}{k_r !}v^{\lambda} : \gamma_1 <  \ldots < \gamma_r \in \Phi^+, \quad \mu +\sum_{i=1}^r{k_i\gamma_i}=\lambda\right\rangle_K,
\label{Elements_generating_W(lambda)_mu}
\end{equation}
where $< $ is the ordering on $\Phi^+$ introduced in Section \ref{structure constants and Chevalley basis}.  We refer the reader to \cite[Chapter VII]{Humphreys1} or to \cite[Chapters 1, 2, 3]{St1} for a detailed account of the related theory.

\begin{lem}\label{elements spanning V(lambda)mu: a first simplification}
Adopt the notation introduced above  and let $\beta_1 < \beta_2 \leqslant  \ldots  \leqslant  \beta_r\in \Phi^+$ be such that $\lambda-\beta_1 \notin \Lambda(\lambda).$  Also let \(k \in \mathbb{Z}_{\geqslant 0}\) and set \(\mu=\lambda- k\beta_1-\sum_{i=2}^r{\beta_i}.\) Then
\[
f_{\beta_1}^{k}f_{\beta_2} \cdots f_{\beta_r}v^{\lambda}\in \left\langle f_{\gamma_1} \cdots f_{\gamma_s} v^{\lambda} : \beta_1 < \gamma_1 \leqslant  \gamma_2 \leqslant  \ldots \leqslant \gamma_s\in \Phi^+,\quad \mu +\sum_{j=1}^s{ \gamma_j}=\lambda \right\rangle_{\Z}.
\] 
\end{lem}

\begin{proof}
We first show the result for $k=1,$ proceeding by induction on $r.$ In the case where $r=2,$ we have $f_{\beta_1}f_{\beta_2}v^{\lambda}= -N_{(\beta_1,\beta_2)}f_{\beta_1+\beta_2}v^{\lambda} + f_{\beta_2}f_{\beta_1}v^{\lambda},$ where $N_{(\beta_1,\beta_2)}\in \Z.$ Also, since $\lambda-\beta_1 \notin \Lambda(\lambda),$ we have $f_{\beta_2}f_{\beta_1}v^{\lambda}=0$ and hence $f_{\beta_1}f_{\beta_2}v^{\lambda} \in \langle f_{\beta_1+\beta_2}v^{\lambda}\rangle_\Z$ as desired. So assume \(r>2\) and consider $\beta_1 < \beta_2 \leqslant   \ldots \leqslant   \beta_r \in \Phi^+$ as in the statement of the lemma. We have
\[
f_{\beta_1}\cdots f_{\beta_r}v^{\lambda}=-N_{(\beta_1,\beta_2)}f_{\beta_1+\beta_2}f_{\beta_3}\cdots f_{\beta_r}v^{\lambda} + f_{\beta_2}f_{\beta_1}f_{\beta_3}\cdots f_{\beta_r}v^{\lambda},
\]
where again \(N_{(\beta_1,\beta_2)}\in \Z,\) and since each of \(f_{\beta_1+\beta_2}f_{\beta_3}\cdots f_{\beta_r}v^{\lambda},\)   \(f_{\beta_1}f_{\beta_3}\cdots f_{\beta_r}v^{\lambda}\) satisfies the desired condition by induction, the result holds in the case where \(k=1.\) Finally, assume \(k>1,\) \(r\geqslant 2,\) and consider $\beta_1 < \beta_2 \leqslant   \ldots \leqslant \beta_r \in \Phi^+$ as in the statement of the lemma. Thanks to our induction hypothesis,  the vector \(f_{\beta_1}^{k}f_{\beta_2}\cdots f_{\beta_r}v^{\lambda}\) can be written as a  \(\Z\)-linear combination of elements of the form \(f_{\beta_1}f_{\gamma_1}\cdots f_{\gamma_s}v^\lambda\) with \(\beta_1<\gamma_1\leqslant\gamma_2\leqslant\ldots  \leqslant\gamma_s\in \Phi^+.\) Consequently  the result follows from the \(k=1\) case above.
\end{proof}

Finally, for  \(\lambda\in X^+(T)\) a dominant character, set
\[
\Phi^+_\lambda=\{\gamma \in \Phi^+: \lambda-\gamma\in \Lambda(\lambda)\},
\]
and write \(m_\lambda=|\Phi^+ |-|\Phi^+_{\lambda}|.\) The following generalization of Lemma \ref{elements spanning V(lambda)mu: a first simplification} shows that under certain conditions on $\lambda$ and $\mu,$ the set \(\Phi^+\) can be replaced by \(\Phi^+_\lambda\)  in the description of \(V_G(\lambda)_\mu\) given in \eqref{Elements_generating_W(lambda)_mu}.
\vspace{5mm}

\begin{prop}\label{A simplification for V(lambda)mu}
Adopt the notation introduced above  and consider a weight \(\mu\in \Lambda(\lambda)\)   such that \(\mu=\lambda-\sum_{r=1}^n{c_r\alpha_r}\) for some \(0\leqslant c_1,\ldots,c_n< p.\)  Then
\[
V_G(\lambda)_{\mu} =\left\langle \frac{f_{\gamma_1}^{k_1}}{k_1!}\cdots \frac{f_{\gamma_r}^{k_r}}{k_r !}v^{\lambda} : \gamma_1 <  \ldots < \gamma_r \in \Phi^+_\lambda, \quad  \mu +\sum_{i=1}^r{k_i\gamma_i}=\lambda\right\rangle_K.
\]
\end{prop}

\begin{proof}
Let \(\beta_1<\ldots<\beta_r\in \Phi^+\) and \(k_1,\ldots,k_r\in \mathbb{Z}_{>  0}\) be such that \(\mu +\sum_{i=1}^r{k_i\beta_i}=\lambda.\)
Arguing by induction on \(m_\lambda\) (the case where \(m_\lambda=1\) immediately following from  Lemma \ref{elements spanning V(lambda)mu: a first simplification}), one easily sees that \(f_{\beta_1}^{k_1} \cdots  f_{\beta_r}^{k_r} v^\lambda\) can be written as a \(\Z\)-linear combination of elements of the form \(f_{\gamma_1}^{l_1}\cdots f_{\gamma_s}^{l_s}v^\lambda,\) where \(l_1,\ldots,l_s\in \mathbb{Z}_{> 0}\) and \(\gamma_1 < \ldots < \gamma_s\in \Phi^+_\lambda.\)  Now thanks to our assumption on \(\mu,\)   we obviously have \(0< k_1,\ldots,k_r <p,\)  showing that the element
\[
\frac{f_{\beta_1}^{k_1}}{k_1!}\cdots \frac{f_{\beta_r}^{k_r}}{k_r !}v^{\lambda} \in V_G(\lambda)_\mu
\]
is a \(K\)-linear combination of elements of the form \(\tfrac{l_1 ! \cdots l_s !}{k_1! \cdots k_r !} \tfrac{f_{\gamma_1}^{l_1}}{l_1 !}\cdots \frac{f_{\gamma_s}^{l_s}}{l_s!}v^\lambda.\)  Finally, since   \(k_1,\ldots,k_r\in \mathbb{Z}_{> 0}\) and \(\beta_1<\ldots<\beta_r\in \Phi^+\)  were arbitrarily chosen (such that \(\mu +\sum_{i=1}^r{k_i\beta_i}=\lambda\)),  the proof is complete.  
\end{proof}

\section{Proof of Theorem \ref{Main result}}     
\label{[Section] proof of the main result}       

Let $K$ be an algebraically closed field of characteristic $p \geqslant 0,$ $Y$ a simply connected simple algebraic group of type $B_n$ $(n\geqslant 3)$ over $K,$ and $X$ the subgroup of $Y$ of type $D_n$ embedded in the usual way. Fix a Borel subgroup $B_Y=U_Y T_Y$ of $Y,$ where $T_Y$ is a maximal torus of $Y$ and $U_Y$ the unipotent radical of $B_Y,$ let $\Pi(Y)=\{\alpha_1,\ldots,\alpha_n\}$ denote a corresponding base of the root system $\Phi(Y)=\Phi^+(Y)\sqcup \Phi^-(Y)$ of $Y,$ and let $\{\lambda_1,\ldots,\lambda_n\}$ be the associated set of  fundamental dominant weights for $T_Y.$ Here we have 
\begin{equation*}
X=\langle T_Y, U_{\alpha}: \alpha \in \Phi(Y)\mbox{ is a long root}\rangle.
\end{equation*}
Let $B_X=U_X T_X$ be a  Borel subgroup of $X,$ where $T_X=T_Y$ is a maximal torus of $X$ and $U_X=U_Y\cap X$ the unipotent radical of $B_X,$ and denote by $\Pi(X)=\{\beta_1,\ldots,\beta_n\}$ the corresponding base of the root system $\Phi(X)=\Phi^+(X)\sqcup \Phi^-(X)$ of $X.$ Here $\beta_i=\alpha_i$ for every $1\leqslant i<n,$ $\beta_n=\alpha_{n-1}+2\alpha_n,$ while the corresponding fundamental dominant $T_X$-weights $\omega_1,\ldots,\omega_n$ satisfy 
\begin{equation}
\lambda_i|_{T_X}=\omega _i, \mbox{ for $1\leqslant i<n-1,$ } \lambda_{n-1}|_{T_X}=\omega_{n-1} +\omega_n,\mbox{ and  } \lambda_n|_{T_X}=\omega_n.
\label{[Dn<Bn] Weight restrictions}
\end{equation}

In addition,  fix  a Chevalley basis \(\mathscr{B}=\{f_{\alpha},h_{\alpha_r},e_{\alpha}: \alpha \in \Phi^+(Y), 1\leqslant r\leqslant n\}\) of $\Lie(Y),$  as in Section \ref{structure constants and Chevalley basis}, and let $V=L_Y(\lambda)$ be a non-trivial irreducible \(KY\)-module having $p$-restricted highest weight $\lambda\in X^+(T_Y).$  Finally, set  $\omega=\lambda|_{T_X}$ and let $v^+$ denote a maximal vector of weight $\lambda$ in $V$ for $B_Y.$ Since $B_X\subset B_Y,$ the latter is a maximal vector for $B_X$ as well, showing that $\omega$ affords the highest weight of a $KX$-composition factor of $V.$ In this final section, we prove the main results of this paper, namely Theorem \ref{Main result} and Corollary \ref{main_corol}.
\vspace{5mm}

First suppose that \(\langle \lambda,\alpha_n\rangle \neq 0\) and observe that in this situation, the element $f_{\alpha_n}v^+$ is non-zero and satisfies $x_{\beta}(c)f_{\alpha_n}v^+=f_{\alpha_n}v^+$ for every $\beta\in \Pi(X)$ and $c\in K.$ In other words, $f_{\alpha_n}v^+$ is a maximal vector in $V$ for $B_X $ and the $T_X$-weight $\omega'=(\lambda-\alpha_n)|_{T_X}$ affords the highest weight of a second $KX$-composition factor of $V.$ One then easily sees that $\omega$ and $\omega'$ are  interchanged by the graph automorphism \(\theta \) of order \(2\) of $X,$   so that 
\begin{equation}
\Lambda(\omega)\cap \Lambda(\omega') =\emptyset.
\label{[Dn<Bn] Proof of the main result: the weights of V(omega) and V(omega') are all disctinct}
\end{equation}

Also observe that if \(\langle \lambda, \alpha_n\rangle>1,\) then $(\lambda-2\alpha_n)|_{T_X} \in \Lambda(V|_X)$ is neither in $\Lambda(\omega)$ nor $\Lambda(\omega'),$ giving the existence of  a third $KX$-composition factor of \(V.\) Therefore we may and shall assume \(\langle \lambda, \alpha_n\rangle=1,\) and since $L_X(\omega)$ and $L_X(\omega')$ are interchanged by $\theta,$ \cite[Theorem 3.3]{Fo1} applies, yielding the following result. (The assertion on the complete reducibility of $V|_X$ immediately follows from \eqref{[Dn<Bn] Proof of the main result: the weights of V(omega) and V(omega') are all disctinct}.)

\begin{thm}[The case $\langle \lambda, \alpha_n \rangle \neq 0$]\label{le_cas_an>0}
Let \(V\) be an irreducible \(KY\)-module having \(p\)-restricted highest weight $\lambda\in X^+(T_Y),$ with $\langle \lambda, \alpha_n \rangle \neq 0.$ In addition, if \(\lambda\notin \Z\lambda_n,\) let $1\leqslant k<n$ be maximal such that $\langle \lambda, \alpha_k\rangle \neq 0.$ Then  $X$ has exactly two composition factors on $V$ if and only if one of the following holds.
\begin{enumerate}
\item \label{first_condition_in_theorem_ford} \(\lambda=\lambda_n.\)
 \item \(\lambda\) is not as in \textnormal{\ref{first_condition_in_theorem_ford}},  \(\langle \lambda, \alpha_n \rangle=1,\) and the following divisibility conditions are satisfied.
					\begin{enumerate}
					\item $p\mid a_i+a_j + j - i$ for every $1\leqslant i<j<n$ such that $a_i a_j \neq 0$ and $a_r=0$ for $i<r<j.$ 
					\item $p\mid 2(a_k +n- k )+1.$
					\end{enumerate}
\end{enumerate}
Moreover, if $X$ has exactly two composition factors on $V,$ then $V|_X$ is completely reducible.
\end{thm}

As Theorem \ref{le_cas_an>0} yields a proof of Theorem \ref{Main result} in the case where \(\langle \lambda,\alpha_n\rangle \neq 0,\) we may and shall assume $\langle \lambda, \alpha_n \rangle =0$ in the remainder of this section. In addition, we might as well assume $p\neq 2$ (so we can apply the results in Sections \ref{Weight spaces for G of type Al} and \ref{Weight spaces for G of type Bn}), since otherwise \(X\) acts irreducibly on $V=L_Y(\lambda)$ by \cite[Theorem 1, Table 1 $(\mbox{MR}_4)$]{Se}. Here again, we can define
\[
k=\max \{1\leqslant r<n : \langle \lambda, \alpha_r\rangle \neq 0\}
\] 
(as \(\lambda \neq 0\)),  and since $p\neq 2,$  observe that $w^+=f_{\alpha_k+\cdots+\alpha_n}v^+\in V$ is non-zero. Moreover,  $x_{\beta}(c)w^+=w^+$ for every $\beta\in \Pi(X)$ and $c\in K.$ Therefore $w^+$ is a maximal vector in $V$ for $B_X$ and thus the $T_X$-weight $\omega'=(\lambda-(\alpha_k+\cdots+\alpha_n))|_{T_X}$ affords the highest weight of a second $KX$-composition factor of $V.$ Finally, notice that \(\omega'=\omega-(\beta_1+\cdots+\beta_{n-2} +\tfrac{1}{2}\beta_{n-1}+\tfrac{1}{2}\beta_n)\) and hence  
\begin{equation}
\Lambda(\omega)\cap \Lambda(\omega') =\emptyset.
\label{[Dn<Bn] Proof of the main result: L(w) and L(w') in direct sum}
\end{equation}

The next result shows that a necessary condition for \(X\) to have exactly two composition factors on \(V\) is for \(\omega'\) to be \(p\)-restricted. In addition, it also provides a proof of Theorem \ref{Main result} in the case where \(V\) has highest weight \(\lambda=\lambda_k.\)

\begin{lem}\label{omega_and_omega'_have_to_be_restricted}
Adopt the notation introduced above. If \(\omega'\) is not  \(p\)-restricted, then \(X\) acts with more than two composition factors on \(V.\) Also if \(\lambda=\lambda_k,\) then \(V|_X\) has exactly two composition factors and is completely reducible.
\end{lem}

\begin{proof}
First suppose that  $\omega'$ is not $p$-restricted. Since \(\langle \lambda,  \alpha_n\rangle =0,\) we then have  \(p\mid a_{k-1}+1\) by \eqref{[Dn<Bn] Weight restrictions}, so that $(\lambda-(\alpha_{k-1}+\cdots+\alpha_n))|_{T_X} \notin \Lambda(\omega)\cup \Lambda(\omega')$ by Theorem \ref{[Preliminaries] Weights and multiplicities: Steinberg tensor product theorem}. Consequently the latter weight occurs in a third $KX$-composition factor of \(V\) and the desired assertion holds. Next assume \(\lambda=\lambda_k,\) in which case an application of \cite[Proposition 4.2.2 Parts (e),(f)]{McNinch} yields \(V\cong V_Y(\lambda),\) as well as  \(L_X(\omega)\cong V_X(\omega)\) and \(L_X(\omega')\cong V_X(\omega').\) A straightforward computation (using Weyl's degree formula \cite[Corollary 24.3]{Humphreys1}, for example) then yields \(\dim V=\dim L_X(\omega) + \dim L_X(\omega'),\) showing that \(X\) acts with exactly two composition factors on \(V.\) Finally, the assertion on the complete reducibility follows directly from \eqref{[Dn<Bn] Proof of the main result: L(w) and L(w') in direct sum}. 
\end{proof}

In view of Lemma \ref{omega_and_omega'_have_to_be_restricted}, we assume  \(\lambda\neq \lambda_k\) and suppose that each of the weights \(\omega,\omega'\in X^+(T_X)\) is \(p\)-restricted for the remainder of this section, which in particular implies \(0\leqslant a_{k-1}<p-1.\) Now observe that \(X\) acts with exactly two composition factors on \(V\) if and only if \(\langle Xv^+\rangle\) and \(\langle Xw^+\rangle\) are irreducible and \(V=\langle Xv^+\rangle \oplus \langle Xw^+\rangle.\)  The next result shows that the latter equality can be translated into an equality of modules for the Lie algebra \(\Lie(X)\) of $X.$

\begin{lem}\label{X_has_two_c.f._on_V_iff_V=Lie(X)v^+_+_Lie(X)w^+}
Adopt the notation introduced above. Then \(X\) acts with exactly two composition factors on \(V\) if and only if 
\[
V=\Lie(X)v^+ \oplus \Lie(X)w^+.
\]
\end{lem}

\begin{proof}
First assume \(X\) has exactly two composition factors on \(V.\) Then  \(V=\langle Xv^+\rangle \oplus \langle Xw^+\rangle\) (as seen above), with both \(\langle Xv^+\rangle\) and \(\langle Xw^+\rangle\) irreducible. Now since  \(\omega,\omega'\in X^+(T_X)\) are \(p\)-restricted,  we get that \(\langle Xv^+\rangle\) and \(\langle Xw^+\rangle\) are irreducible as modules for \(\Lie(X)\) by \cite[Theorem 1]{Curtis}. Hence \(\langle Xv^+\rangle = \Lie(X)v^+\) and \(\langle Xw^+\rangle=\Lie(X)w^+,\) so that \(V=\Lie(X)v^+ \oplus \Lie(X)w^+ \) as desired. Conversely, suppose that \(V=\Lie(X)v^+ \oplus \Lie(X)w^+.\) Then \(V=\langle Xv^+\rangle + \langle X w^+\rangle,\) and since \(\langle Xv^+\rangle,\) \(\langle X w^+\rangle\) are homomorphic images of \(V_X(\omega),\) \(V_X(\omega'),\) respectively, we deduce from \eqref{[Dn<Bn] Proof of the main result: L(w) and L(w') in direct sum} that  
\[
V=\langle Xv^+\rangle \oplus \langle X w^+\rangle.
\]
In particular \(V|_X\) contains a quotient isomorphic to \(L_X(\omega),\) and since \(V\) is self-dual as a \(KY\)-module, its restriction to \(X\) is self-dual as well. Therefore there exists a submodule \(U\) of \(V|_X\) such that \(U\cong L_X(\omega).\) As \((V|_X)_\omega=\langle v^+\rangle_K,\) we then get that \(\langle X v^+\rangle \subseteq U \cong L_X(\omega).\) A similar argument shows that \(\langle X w^+\rangle \cong L_X(\omega'),\) thus completing the proof.
\end{proof}

The following result, whose proof is similar to that of \cite[Lemma 3.4]{Fo2}, gives a necessary and sufficient condition for \(V|_{\Lie(X)}\) to be equal to  the direct sum of \(\Lie(X)v^+\) and \(\Lie(X)w^+,\) which by Lemma \ref{X_has_two_c.f._on_V_iff_V=Lie(X)v^+_+_Lie(X)w^+} translates to a necessary and sufficient condition for \(X\) to act with exactly two composition factors on \(V.\)

\begin{prop}\label{[Dn<Bn] Proof of the main result: Key Proposition}
Adopt the notation introduced above. Then \(V=\Lie(X)v^+ \oplus \Lie(X)w^+\) if and only if  for every $1\leqslant i\leqslant k,$ we have $f_{i,n}v^+\in \Lie(X) w^+$ and $f_{i,n} w^+\in \Lie(X) v^+.$
\end{prop}

\begin{proof}
First assume \(V=\Lie(X)v^+ \oplus \Lie(X)w^+\) and consider $1\leqslant i\leqslant k.$ Also let $v\in \Lie(X)v^+$ and $w\in\Lie(X) w^+$ be the unique elements such that $f_{i,n}v^+ = v+w.$ As $f_{i,n}v^+$ lies in the weight space $V_{\lambda-(\alpha_i+\cdots+\alpha_n)},$ so do $v$ and $w.$ Observe however that $\Lie(X)v^+ \cap V_{\lambda-(\alpha_i+\cdots+\alpha_n)}= 0,$ forcing $v=0$ and thus $f_{i,n}v^+=w \in \Lie(X) w^+.$ A similar argument shows that $f_{i,n}w^+ \in \Lie(X) v^+,$ thus the desired result. 

Conversely, suppose that $f_{i,n}v^+\in \Lie(X)w^+$ and $f_{i,n} w^+\in \Lie(X)v^+$ for every $1\leqslant i\leqslant k,$ and write $U=\Lie(X)v^+ \oplus \Lie(X)w^+\subseteq V.$ We first show that 
\[
f_{\gamma_1}\cdots f_{\gamma_s}v^+ \in  U
\]
for every $\gamma_1\in \Phi^+(Y)$ short and $\gamma_2,\ldots,\gamma_s\in \Phi^+(Y)$ long. We proceed by induction on \(s\geqslant 2,\) first considering the situation where \(s=2.\) Here we have \(f_{\gamma_1}f_{\gamma_2}v^+=N_{(\gamma_1,\gamma_2)}f_{\gamma_1+\gamma_2}v^+ + f_{\gamma_2}f_{\gamma_1}v^+,\) where \(N_{(\gamma_1,\gamma_2)}=0\) if \(\gamma_1+\gamma_2\notin \Phi^+(Y).\) Since \(f_{\gamma_1}v^+\) belongs to \(\Lie(X)w^+\) by assumption and  \(\gamma_2\) is a long root in \(\Phi^+(Y),\) we deduce that \(f_{\gamma_2}f_{\gamma_1}v^+\in \Lie(X)w^+.\) So it only remains to check that \(N_{(\gamma_1,\gamma_2)}f_{\gamma_1+\gamma_2}v^+\in \Lie(X)w^+.\) If \(\gamma_1+\gamma_2\notin \Phi^+(Y),\) then \(N_{(\gamma_1,\gamma_2)}=0\) and there is nothing to do, while if on the other hand \(\gamma_1+\gamma_2\) is a root, then one easily checks that the latter has to be short, allowing us to conclude thanks to our initial assumption. Next consider \(s>2\) and observe that we have
\begin{equation*}
f_{\gamma_1}\cdots f_{\gamma_s} v^+		= f_{\gamma_2}f_{\gamma_1}f_{\gamma_3}\cdots f_{\gamma_s}v^+ + N_{(\gamma_1,\gamma_2)}f_{\gamma_1+\gamma_2}f_{\gamma_3}\cdots f_{\gamma_s} v^+.
\end{equation*}
Thanks to our induction assumption and the above discussion, each of the vectors   $f_{\gamma_1}f_{\gamma_3}\cdots f_{\gamma_s}v^+$ and $f_{\gamma_1+\gamma_2}f_{\gamma_3}\cdots f_{\gamma_s}v^+$ lies inside $U,$ and as $\gamma_2$ is long, we get that  $f_{\gamma_1}\cdots f_{\gamma_s}v^+ \in U $ as desired. Arguing in a similar fashion, one shows that \(f_{\gamma_1}\cdots f_{\gamma_s}w^+ \in  U\) as well, where  $\gamma_1,\gamma_2,\ldots,\gamma_s$ are as above. 

To complete the proof, notice that since  \(\lambda\in X^+(T_Y)\) is \(p\)-restricted, the \(KY\)-module \(V\) is irreducible as a \(\Lie(Y)\)-module by \cite[Theorem 1]{Curtis} and hence \(V=\Lie(Y)v^+.\) Therefore in order to show that $U=V,$  it suffices to prove that $f_{\gamma_1}\cdots f_{\gamma_s}v^+ \in U$  for any positive roots $\gamma_1,\ldots,\gamma_s\in \Phi^+(Y).$ Again, we proceed by induction on \(s\geqslant 1,\) first observing that  the desired assertion holds in the case where \(s=1\) by our initial assumption. Now assume $s\geqslant 2$ and consider \(\gamma_1,\ldots,\gamma_s\in \Phi^+(Y).\) Thanks to our induction hypothesis, we  get that \(f_{\gamma_2}\cdots f_{\gamma_s}v^+\in U,\) and hence  \(f_{\gamma_1}\cdots f_{\gamma_s}v^+\) can be rewritten as a linear combination of elements of the form \(f_{\gamma_1}f_{\delta_1}\cdots f_{\delta_t}v^+,\) where \(\delta_1,\ldots,\delta_t\) are long roots in \(\Phi^+(Y).\) The result then obviously holds if \(\gamma_1\) is long, while if on the other hand \(\gamma_1\) is short, then the assertion follows from the above discussion. Consequently $U=V,$ thus completing the proof.
\end{proof}

In view of Lemma  \ref{X_has_two_c.f._on_V_iff_V=Lie(X)v^+_+_Lie(X)w^+} and Proposition \ref{[Dn<Bn] Proof of the main result: Key Proposition}, one observes that the key to the proof of Theorem \ref{Main result} resides in determining all pairs $(\lambda,p)$ such that $f_{i,n}v^+\in \Lie(X)w^+ $ and  $f_{i,n}w^+\in \Lie(X)v^+$ for every $1\leqslant i\leqslant k.$ The rest of this section is  devoted to the proof of the three following technical lemmas, which  provide a partial answer to this question. We refer the reader to Appendices \ref{Weight spaces for G of type Al} and  \ref{Weight spaces for G of type Bn} for notation such as $V_{r,k},$ where $1\leqslant r <k.$

\begin{lem}\label{Lemma3.5.1,Ford}
Suppose that $f_{r,k}v^+ \in V_{r,k}$ for every $1\leqslant r<k.$ Then  $f_{i,n}v^+ \in \Lie(X) w^+$  for every $1\leqslant i<k.$ Conversely, if \(1\leqslant i<k\)  is such that \(f_{i,n}v^+\in \Lie(X)w^+,\) then \(f_{i,k}\in V_{i,k}.\)
\end{lem}

\begin{proof}
We start by showing the second assertion, considering \(1\leqslant i<k\) as in the statement of the latter. Since $\Lie(X) w^+\cap V_{\lambda-(\alpha_i+\cdots + \alpha_n)} \subset V_{i,n},$ we immediately get that $f_{i,n}v^+ \in V_{i,n}.$ Moreover, as \(f_{t,n}v^+=0\) for \(k<t\leqslant n,\) there exists $\{\xi_j,\xi^j_{(m)}: i\leqslant j\leqslant k-1, (m)\in P(i,j)\}\subset K$ such that
\[
f_{i,n}v^+= \sum_{j=i}^{k-1}{\left( \xi_j f_{i,j}f_{j+1,n}v^+ + \sum_{(m)\in P(i,j)}{\xi^j_{(m)}f_{(m)}f_{j+1,n}v^+}\right)}.
\]
Applying successively $e_{\alpha_n},e_{\alpha_{n-1}},\ldots,e_{\alpha_{k+1}}$ gives a non-zero multiple of $f_{i,k}v^+$ on the left-hand side and elements lying inside $V_{i,k}$ on the right-hand side, so that \(f_{i,k}v^+\in V_{i,k}\) as desired. 

Conversely, assume $f_{r,k}v^+ \in V_{r,k}$ for every $1\leqslant r<k$ and consider \(1\leqslant i<k.\) We proceed by induction on \(1\leqslant k-i<k.\) In the case where \(i=k-1,\) observe that \(f_{k+1,n}v^+=0\) and \(V_{k-1,k}=\langle f_{\alpha_{k-1}}f_{\alpha_k}v^+\rangle_K,\) so that 
\[
f_{k-1,n}v^+	 =	 f_{k+1,n}f_{k-1,k}v^+						 
			 =	\xi f_{\alpha_{k-1}}f_{k+1,n}f_{\alpha_k}v^+		 
			 =	   \xi f_{\alpha_{k-1}}f_{k,n}v^+
			 = 	  \xi f_{\alpha_{k-1}}w^+
\]
for some \(\xi\in K.\) Therefore the result holds in this situation since \(\alpha_{k-1}\in \Phi^+(Y)\) is long. Next consider \(1\leqslant i<k-1\) and observe that since  \(f_{k+1,n}v^+=0,\) we have \(f_{i,n}v^+= f_{k+1,n}f_{i,k}v^+.\) Furthermore,  as \(f_{i,k}v^+\in V_{i,k}\) by assumption, there exists \(\{\xi_{(m)} :(m)\in P(i,k)\} \subset K\) such that 
\[
f_{i,k}v^+= \sum_{(m)\in P(i,k)}{\xi_{(m)}f_{(m)}v^+}.
\]
Now notice that for any fixed \( (m)=(m_r)_{r=1} ^s \in P(i,k),\) the element \(f_{k+1,n}\) commutes with \(f_{i,m_1}\) as well as \(f_{m_r+1,m_{r+1}}\) for \(1\leqslant r<s.\)  Consequently \(f_{k+1,n}f_{(m)}v^+= f_{i,m_1}\cdots f_{m_{s-1}+1,m_s}f_{m_s+1,n}v^+\) and our induction assumption applies, showing that \(f_{k+1,n}f_{(m)}v^+\in \Lie(X)w^+.\) Since \((m)\) was arbitrarily chosen in \(P(i,k),\) the result follows.
\end{proof}

Next consider the \(T_Y\)-weight $\mu=\lambda-2(\alpha_k+\cdots+\alpha_n) \in \Lambda^+(\lambda).$ By  \eqref{V_k,n}, together with Lemma \ref{[Preliminaries] Parabolic embeddings: restriction to Levi subgroup} (applied to the $B_{n-k+1}$-parabolic corresponding to the simple roots $\alpha_k,\ldots,\alpha_n$), the weight space $V_{\mu}$ is spanned by  the vectors
\begin{equation}
\left\{f_{k,n}w^+\right\} \cup \left\{f_{k,j}F_{k,j+1}v^{+}\right\}_{k\leqslant j<n}.
\label{V_k,n (second use)}
\end{equation}

As in Appendix \ref{Weight spaces for G of type Bn}, we write $V^2_{k,n}$ to designate the span of all the generators in \eqref{V_k,n (second use)} except for $f_{k,n}w^+.$ Clearly, we have $V^2_{k,n}\subset \Lie(X)v^+,$ since the elements of $V^2_{k,n}$ are of the form $$f_{\gamma_1}f_{\gamma_2}v^+,$$ where both $\gamma_1$ and $\gamma_2$ are long roots in $\Phi^+(Y).$ Conversely, one easily sees that $\Lie(X)v^+\cap V_{\mu}\subset V^2_{k,n},$ so that the next result holds. (The case where $a_k=1$ immediately follows from Proposition \ref{[Dn<Bn] Preliminaries for L of type Bn: technical proposition concerning the weight lambda-2...2 for lambda=lambda_1}.)

\begin{lem}\label{Lemma3.5.1,Ford(part2)}
Adopt the notation introduced above. If $a_k=1,$ then $ f_{k,n}w^+ \in \Lie(X)v^+,$ while if on the other hand $a_k>1,$ then  $ f_{k,n} w^+ \in V^2_{k,n}$ if and only if  $ f_{k,n}w^+ \in \Lie(X)v^+.$
\end{lem}

Finally, assume $\langle \lambda, \alpha_k\rangle =1$ and let $1\leqslant l<k$ be such that $\langle \lambda, \alpha_l\rangle \neq 0,$ but $\langle \lambda, \alpha_r \rangle =0$ for every $l<r<k.$ (Such an integer \(l\) exists, since \(\lambda\neq \lambda_k.\)) Also suppose that $f_{l,k}v^+\in V_{l,k}$ and write $\mu=\lambda-(\alpha_l+\cdots+\alpha_{k-1}+2\alpha_k+\cdots+2\alpha_n).$ Using Proposition \ref{Une premiere simplification de l'ensemble generateur de V(nu)} together with Lemma \ref{[Preliminaries] Parabolic embeddings: restriction to Levi subgroup} (applied to the Levi subgroup of type $B_{n-l+1}$  corresponding to the simple roots $\alpha_l,\ldots,\alpha_n$) shows that the weight space $V_{\mu}$ is spanned by the vectors
\begin{align}
\{F_{l,k}v^+\} 			&\cup  \{f_{l,i}F_{i+1,k}v^+\}_{l\leqslant i\leqslant k-2} 					 			\cr		
					 				&\cup  \{f_{l,i}f_{i+1,k-1}f_{\alpha_k}F_{k,k+1}v^+\}_{l\leqslant i\leqslant k-2}		\cr
					 				&\cup  \{f_{l,i}f_{k,j}F_{i+1,j+1}v^+\}_{l\leqslant i\leqslant k-2,k\leqslant j <n} 	 		\cr
					 				&\cup  \{f_{k,j}F_{l,j+1}v^+\}_{k\leqslant j<n} 									\cr
							 		&\cup  \{f_{l,n}w^+\}.
\label{generators_for_V_nu (second use)}
\end{align}

As in Appendix \ref{Weight spaces for G of type Bn}, we write $V_{l,k,n}$ to designate the span of all the generators in \eqref{generators_for_V_nu (second use)} except for $f_{l,n}w^+.$ Clearly, we have $V_{l,k,n}\subset \Lie(X)v^+,$ since the elements of $V_{l,k,n}$ are of the form $$f_{\gamma_1} \cdots f_{\gamma_r}v^+,$$ with $\gamma_1,\ldots,\gamma_r$ long roots in $\Phi^+(Y).$ Conversely, one easily sees that $\Lie(X)v^+\cap V_{\mu}\subset V_{l,k,n},$ so that the following result holds.

\begin{lem}\label{Lemma3.5.1,Ford(part3)}
Assume \(a_k=1\) and suppose that \(f_{l,k}v^+\in V_{l,k}.\) Then adopting the notation introduced above, we have that $f_{l,n}w^+ \in V_{l,k,n}$ if and only if $  f_{l,n}w^+ \in \Lie(X)v^+.$
\end{lem}

Finally, we are now in a position to prove Theorem \ref{Main result}.
\begin{proof}[Proof of Theorem \ref{Main result}]
Let us first assume that \(X\) acts with exactly two composition factors on \(V.\) By Theorem \ref{le_cas_an>0}, we may and shall assume $\langle \lambda, \alpha_n \rangle =0$ in the remainder of the proof, as well as $p\neq 2,$ (since otherwise \(X\) acts irreducibly on $V $ by \cite[Theorem 1, Table 1 $(\mbox{MR}_4)$]{Se}). Furthermore, an application of Lemma \ref{omega_and_omega'_have_to_be_restricted} shows that \(\omega'\) is restricted  and also allows us to assume the existence of \(1\leqslant l<k<n\) maximal such that \(a_la_k\neq 0.\) Thanks to Lemma \ref{X_has_two_c.f._on_V_iff_V=Lie(X)v^+_+_Lie(X)w^+} and Proposition \ref{[Dn<Bn] Proof of the main result: Key Proposition}, we then  have \(f_{r,n}v^+\in \Lie(X)w^+\) and \(f_{r,n}w^+\in \Lie(X)v^+\) for every \(1\leqslant r \leqslant k,\) which in particular yields \(f_{r,k}v^+\in V_{r,k}\) for every \(1\leqslant r\leqslant k\) by Lemma \ref{Lemma3.5.1,Ford}. Consequently, applying Theorem  \ref{How_to_relate_divisibility_conditions_and_generators_for_V_i,j} (to the Levi subgroup of type \(A_k\) corresponding to the simple roots \(\alpha_1,\ldots,\alpha_k\)) yields  the divisibility conditions  in \ref{[Dn<Bn] Theorem: Remaining cases (1)}	 of Theorem \ref{Main result}. We next show that the divisibility condition in \ref{[Dn<Bn] Theorem: Remaining cases (2)}	 of Theorem \ref{Main result} is satisfied as well, first considering the case where \(a_k>1.\) Here  \(f_{k,n}w^+\in V_{k,n}^2\) by Lemma  \ref{Lemma3.5.1,Ford(part2)} and hence an   application of Proposition \ref{Conditions_for_f(k,n)_to_belong_to_V(k,n)} (to the Levi subgroup of type \(B_{n-k+1}\) corresponding to the simple roots \(\alpha_k,\ldots,\alpha_n\)) yields the desired result. Finally , suppose that \(a_k=1,\) in which case \(f_{l,n}w^+\in V_{l,k,n}\) by Lemma \ref{Lemma3.5.1,Ford(part3)}.  If \(l=k-1,\) then  the result holds by Proposition  \ref{condition_if_l=k-1}, while if on the other hand \(1\leqslant l<k-1,\) then applying Proposition \ref{condition_if_l<k-1} yields the desired divisibility condition.
\vspace{5mm}

To complete the proof, let us consider \((\lambda,p)\) as in \ref{[Dn<Bn] Theorem: the case lambda=lambda_i for 0<i<n}, \ref{[Dn<Bn] Theorem: the case lambda=lambda_n} or \ref{[Dn<Bn] Theorem: Remaining cases} of Theorem \ref{Main result}. By Theorem \ref{le_cas_an>0} and Lemma \ref{omega_and_omega'_have_to_be_restricted}, it suffices to show that \(V|_X\) has exactly two composition factors in the case where \((\lambda,p)\) is as in \ref{[Dn<Bn] Theorem: Remaining cases} of Theorem \ref{Main result}, with  \(\langle \lambda, \alpha_n\rangle =0.\) Fix \(1\leqslant r<k\) and consider \(f_{r,k}v^+.\) If \(a_r=0,\) then \(f_{r,k}v^+= - f_{\alpha_r}f_{r,k}v^+ \in V_{r,k},\) while if \(a_r\neq 0,\) then an application of Theorem \ref{How_to_relate_divisibility_conditions_and_generators_for_V_i,j} (to the Levi subgroup corresponding to the simple roots \(\alpha_r,\ldots,\alpha_k\)) yields \(f_{r,k}v^+\in V_{r,k}.\) Since \(r\) was arbitrarily chosen, we get  that \(f_{i,k}v^+\in \Lie(X)w^+\) for every \(1\leqslant i< k\) by Lemma \ref{Lemma3.5.1,Ford}. Therefore by Lemma \ref{X_has_two_c.f._on_V_iff_V=Lie(X)v^+_+_Lie(X)w^+} and Proposition \ref{[Dn<Bn] Proof of the main result: Key Proposition}, it only remains to show that \(f_{i,n}w^+\in \Lie(X)v^+\) for \(1\leqslant i\leqslant k\) as well. 

First notice  that if \(i=k,\) then the assertion immediately follows from Proposition \ref{[Dn<Bn] Preliminaries for L of type Bn: technical proposition concerning the weight lambda-2...2 for lambda=lambda_1} or \ref{Conditions_for_f(k,n)_to_belong_to_V(k,n)} (depending on whether \(a_k=1\) or not)  together with Lemma \ref{Lemma3.5.1,Ford(part2)}, so we assume \(1\leqslant i<k\) in the remainder of the proof. Now if \(l<i<k,\) then \(f_{i,k-1}v^+=0\) and using Lemma \ref{structure_constants_for_B}, we successively get
\begin{align*}
f_{i,n}w^+	&=	-f_{i,k-1}f_{k,n}w^+ + f_{k,n}f_{i,k-1}f_{k,n}v^+		\cr
			&=  -f_{i,k-1}f_{k,n}w^+ - f_{k,n}f_{i,n}v^+					\cr
			&=	-f_{i,k-1}f_{k,n}w^+ +2F_{i,k}v^+ - f_{i,n}w^+ ,
\end{align*}
so that \(f_{i,n}w^+	= F_{i,k}v^+ - \frac{1}{2}f_{i,k-1} f_{k,n}  w^+.\) Clearly \(F_{i,k}v^+\in \Lie(X)v^+,\) while \(f_{i,k-1} f_{k,n}  w^+\in \Lie(X)v^+\) thanks to the fact that \(\alpha_i+\cdots+\alpha_{k-1}\in \Phi^+(Y)\) is long, hence the desired result in this situation. 

We next proceed by induction on \(0\leqslant l-i\leqslant l-1,\) first considering the case where \(i=l.\) If  $a_k=1,$ we have $f_{l,n}w^+\in V_{l,k,n}$ by Proposition \ref{condition_if_l=k-1} or \ref{condition_if_l<k-1} (depending on whether $l=k-1$ or not), and hence an application of Lemma \ref{Lemma3.5.1,Ford(part3)} yields \(f_{l,n}w^+\in \Lie(X)w^+.\) If on the other hand  \(a_k>1,\)  applying Remark \ref{A_maximal_vector_in_V_(A_n)(a,0,....,0,b)} to the Levi subgroup corresponding to the simple roots \(\alpha_l,\ldots,\alpha_k\) yields
\[ 
a_k f_{l,n}v^+	=	a_k f_{k+1,n}f_{l,k}v^+									
				=	 \sum_{s=l}^{k-1} f_{l,s} f_{k+1,n}f_{s+1,k}v^+ 		  
				=   \sum_{s=l}^{k-1} f_{l,s} f_{s+1,n}v^+,
\]
so that  
\begin{align*}
a_kf_{l,n}w^+ 	&= 2a_kF_{l,k}v^+ + a_kf_{k,n}f_{l,n}v^+			\cr
				&= 2a_kF_{l,k}v^+ + (1-\delta_{l,k-1})\sum_{s=l}^{k-2} f_{l,s} f_{k,n} f_{s+1,n}v^+ + f_{k,n}f_{l,k-1}w^+ 											\cr
				&= 2a_kF_{l,k}v^+ + (1-\delta_{l,k-1}) \sum_{s=l}^{k-2} f_{l,s} f_{k,n} f_{s+1,n}v^+ + f_{l,k-1}f_{k,n}w^+	+ f_{l,n}w^+	,
\intertext{and since $a_k>1,$ we finally obtain }
f_{l,n}w^+ 		&= (a_k-1)^{-1} \left( 2a_k F_{l,k}v^+ + (1-\delta_{l,k-1}) \sum_{s=l}^{k-2} f_{l,s} f_{k,n} f_{s+1,n}v^+ + f_{l,k-1}f_{k,n}w^+\right).
\end{align*}
Clearly  \(F_{l,k}v^+\in \Lie(X)v^+,\) while each of the remaining terms on the right-hand side belongs to \(\Lie(X)v^+\) thanks to our induction hypothesis,  from which we deduce that \(f_{l,n}w^+\in \Lie(X)v^+\) as desired. Finally, assume \(1\leqslant i<l.\) Again, using Lemma \ref{structure_constants_for_B}, we successively get
\begin{align*}
f_{i,n}w^+	&=	-f_{i,k-1}f_{k,n}w^+ + f_{k,n}f_{i,k-1}f_{k,n}v^+		\cr
			&=  -f_{i,k-1}f_{k,n}w^+ - f_{k,n}f_{i,n}v^+	+ (f_{k,n})^2 f_{i,k-1} v^+	\cr
			&=	-f_{i,k-1}f_{k,n}w^+  +2F_{i,k}v^+	- f_{i,n}w^+ + (f_{k,n})^2  f_{i,k-1} v^+,
\end{align*}
so it only remains to show that \((f_{k,n})^2f_{i,k-1}v^+\in \Lie(X)v^+.\) By Theorem \ref{How_to_relate_divisibility_conditions_and_generators_for_V_i,j} and Lemma \ref{If_divisibility_conditions_holds,_then_f(i,j)_in_V(i,j)} (applied to suitable Levi subgroups of $Y$), we get that \(f_{i,k-1}\in V_{i,k-1},\) and hence the existence of a subset \(\{\xi_j,\xi^j_{(m)} : i\leqslant j\leqslant k-2,~(m)\in P(i,j)\}\) of $K$  such that 
\[
f_{i,k-1}v^+=\sum_{j=i}^{k-2}{\left(\xi_j f_{i,j}f_{j+1,k-1}v^+ + \sum_{(m)\in P(i,j)}{\xi^j_{(m)} f_{(m)}f_{j+1,k-1}v^+}\right)}.
\]
Now for any \(i\leqslant j\leqslant k-2\) and any \((m)\in P(i,j),\) we have that \(f_{i,j}\) and \(f_{(m)}\) commute with \((f_{k,n})^2.\) Moreover, we successively get (using Lemma \ref{structure_constants_for_B})
\begin{align*}
(f_{k,n})^2 f_{j+1,k-1}v^+	&=	f_{k,n}f_{j+1,n}v^+ + f_{k,n}f_{j+1,k-1}f_{k,n}v^+ 	\cr
							&=	f_{k,n}f_{j+1,n}v^+ + f_{j+1,n}f_{k,n}v^+  +f_{j+1,k-1}f_{k,n}w^+.			\cr
							&=	2f_{k,n}f_{j+1,n}v^+ + 2 F_{j+1,k}v^+  +f_{j+1,k-1}f_{k,n}w^+.			\cr
\end{align*}
Clearly \(F_{j+1,k}v^+\in \Lie(X)v^+\) and \(f_{j+1,k-1}f_{k,n}w^+\in \Lie(X)v^+\) thanks to the \(i=k\) case and the fact that \(\alpha_{j+1}+\cdots+\alpha_{k-1}\) is long. Finally, our induction assumption implies that \(f_{k,n}f_{j+1,n}v^+\in \Lie(X)v^+\) as well, thus completing the proof. 
\end{proof}

Finally, we present the proof of Corollary \ref{main_corol}.

\begin{proof}[Proof of Corollary \ref{main_corol}]     
Let \(\delta\in X^+(T_Y)\) be as in the statement of the corollary, and write \(\delta=\delta_0 + p\delta_1 +\cdots +p^r\delta_r,\) where \(\delta_i\) is \(p\)-restricted for   \(0\leqslant i\leqslant r\) and \(\delta_r\neq 0.\) Setting \(V=L_Y(\delta),\) an application of Theorem  \ref{[Preliminaries] Weights and multiplicities: Steinberg tensor product theorem} yields
\[
V \cong L_Y(\delta_0)\otimes L_Y(\delta_1)^F \otimes \cdots \otimes L_Y(\delta_r)^{F^r},
\] 
from which one deduces that if \(V|_X\) has exactly two composition factors,  then there exists a unique \(0\leqslant j\leqslant r\) such that \(L_Y(\delta_j)|_X\) has exactly two composition factors, while \(L_Y(\delta_i)|_X\) is irreducible for every \(0\leqslant i\leqslant r\) different from \(j.\) If $p\neq 2,$ then \cite[Theorem 1, Table 1 (MR$_4$)]{Se} forces $\delta=p^r\delta_r,$ that is $\delta$ is as in \ref{main_corol1}  of Corollary \ref{main_corol}. If on the other hand $p=2,$ then $\delta_j=\lambda_n$ by Theorem \ref{Main result}, while \cite[Theorem 1, Table 1 (MR$_4$)]{Se} yields $\langle \delta_i,\alpha_n\rangle =0$  for every \(1\leqslant i <r\) different from $j.$ Hence $\delta$ is as in \ref{main_corol2}  of Corollary \ref{main_corol} in this case, yielding the desired assertion.
\vspace{5mm}

Reciprocally, assume \((\delta,p)\) is as in \ref{main_corol1} or \ref{main_corol2} of Corollary \ref{main_corol}. In the former case, we have $\delta=p^r\lambda$ for some $r\in \Z_{\geqslant 0}$ and some $\lambda \in X^+(T_Y)$ such that $L_Y(\lambda)|_X\cong L_X(\omega)\oplus L_X(\omega'),$ where $\omega,\omega'\in X^+(T_X)$ are $p$-restricted. Consequently Theorem \ref{[Preliminaries] Weights and multiplicities: Steinberg tensor product theorem} yields
$$L_Y(\delta)|_X \cong L_Y(\lambda)^{F^r}\cong L_X(\omega)^{F^r}\oplus L_X(\omega')^{F^r},$$
so that $X$ acts with exactly two composition factors on $L_Y(\delta)$ as desired. Finally, assumme \((\delta,p)\) is as in \ref{main_corol2} of Corollary \ref{main_corol}, i.e. $p=2$ and there exists $r\in \Z_{\geqslant 0} $ such that $\delta=\delta_0+2\delta_1+\cdots+2^r\delta_r,$ where $\delta_j=\lambda_n$ for some $0\leqslant j\leqslant r$ and $\langle \delta_i,\alpha_n\rangle=0$ for $1\leqslant i\leqslant r$ different from $j.$ An application of Theorem \ref{[Preliminaries] Weights and multiplicities: Steinberg tensor product theorem} then yields
\begin{align*}
L_Y(\delta)|_X &\cong L_Y(\delta_0)|_X \otimes \left(L_Y(\delta_1)|_X\right)^{F} \otimes \cdots \otimes \left(L_Y(\delta_r)|_X\right)^{F^r},
\end{align*}
from which one easily concludes thanks to \cite[Theorem 1, Table 1 (MR$_4$)]{Se} and Theorem \ref{Main result}, thus completing the proof.
\end{proof}

\section*{Acknowledgements}     

I would like to express my deepest thanks to my Ph.D. advisor Professor Donna M. Testerman, for her constant support and her precious guidance during my doctoral  studies, from which this paper is drawn. I would also like to extend my appreciation to Professors Timothy C. Burness and Frank L\"ubeck, for their very helpful comments and suggestions on the earlier versions of this paper. 

\newpage
\bibliography{References}   

\providecommand{\bysame}{\leavevmode\hbox to3em{\hrulefill}\thinspace}
\providecommand{\MR}{\relax\ifhmode\unskip\space\fi MR }
\providecommand{\MRhref}[2]{%
  \href{http://www.ams.org/mathscinet-getitem?mr=#1}{#2}
}
\providecommand{\href}[2]{#2}
\begin{thebibliography}{BGMT15}

\bibitem[BGMT15]{BGMT}
Timothy~C. Burness, Souma{\"{\i}}a Ghandour, Claude Marion, and Donna~M.
  Testerman, \emph{Irreducible almost simple subgroups of classical algebraic
  groups}, Mem. Amer. Math. Soc. \textbf{236} (2015), no.~1114, vi+110.
  \MR{3364837}

\bibitem[BGT16]{BGT}
Timothy~C. Burness, Souma{\"{\i}}a Ghandour, and Donna~M. Testerman,
  \emph{Irreducible geometric subgroups of classical algebraic groups}, Mem.
  Amer. Math. Soc. \textbf{239} (2016), no.~1130, v+88. \MR{3431945}

\bibitem[BMT16]{BMT}
Timothy~C. Burness, Claude Marion, and Donna~M. Testerman, \emph{On irreducible
  subgroups of simple algebraic groups}, Mathematische Annalen (2016), 1--51.

\bibitem[Bou68]{Bourbaki}
N.~Bourbaki, \emph{\'{E}l\'ements de math\'ematique. {F}asc. {XXXIV}. {G}roupes
  et alg\`ebres de {L}ie.}, Actualit\'es Scientifiques et Industrielles, No.
  1337, Hermann, Paris, 1968. \MR{0240238 (39 \#1590)}

\bibitem[Car89]{Ca}
Roger~W. Carter, \emph{Simple groups of {L}ie type}, Wiley Classics Library,
  John Wiley \& Sons Inc., New York, 1989, Reprint of the 1972 original, A
  Wiley-Interscience Publication. \MR{1013112 (90g:20001)}

\bibitem[Cav15]{Thesis}
M.~Cavallin, \emph{Restricting representations of classical algebraic groups to
  maximal subgroups}, Ph.D. thesis, SB, Lausanne, 2015.

\bibitem[{Cav}16]{Cavallin}
M.~{Cavallin}, \emph{{An algorithm for computing weight multiplicities in
  irreducible modules for complex semisimple Lie algebras}}, ArXiv e-prints
  (2016).

\bibitem[Cur60]{Curtis}
Charles~W. Curtis, \emph{Representations of {L}ie algebras of classical type
  with applications to linear groups}, J. Math. Mech. \textbf{9} (1960),
  307--326. \MR{0110766 (22 \#1634)}

\bibitem[Dyn52]{Dynkin}
E.~B. Dynkin, \emph{Maximal subgroups of the classical groups}, Trudy Moskov.
  Mat. Ob\v s\v c. \textbf{1} (1952), 39--166. \MR{0049903 (14,244d)}

\bibitem[FK97]{Mullineux}
Ben Ford and Alexander~S. Kleshchev, \emph{A proof of the {M}ullineux
  conjecture}, Math. Z. \textbf{226} (1997), no.~2, 267--308. \MR{1477629}

\bibitem[For95]{Fo3}
Ben Ford, \emph{Irreducible restrictions of representations of the symmetric
  groups}, Bull. London Math. Soc. \textbf{27} (1995), no.~5, 453--459.
  \MR{1338688 (96g:20015)}

\bibitem[For96]{Fo1}
\bysame, \emph{Overgroups of irreducible linear groups. {I}}, J. Algebra
  \textbf{181} (1996), no.~1, 26--69. \MR{1382025 (97b:20067)}

\bibitem[For99]{Fo2}
\bysame, \emph{Overgroups of irreducible linear groups. {II}}, Trans. Amer.
  Math. Soc. \textbf{351} (1999), no.~10, 3869--3913. \MR{1467464 (99m:20102)}

\bibitem[Gha10]{Gh}
Souma{\"{\i}}a Ghandour, \emph{Irreducible disconnected subgroups of
  exceptional algebraic groups}, J. Algebra \textbf{323} (2010), no.~10,
  2671--2709. \MR{2609170 (2011c:20098)}

\bibitem[Hum75]{Humphreys2}
James~E. Humphreys, \emph{Linear algebraic groups}, Springer-Verlag, New
  York-Heidelberg, 1975, Graduate Texts in Mathematics, No. 21. \MR{0396773 (53
  \#633)}

\bibitem[Hum78]{Humphreys1}
\bysame, \emph{Introduction to {L}ie algebras and representation theory},
  Graduate Texts in Mathematics, vol.~9, Springer-Verlag, New York-Berlin,
  1978, Second printing, revised. \MR{499562 (81b:17007)}

\bibitem[Jan03]{Jantzen}
Jens~Carsten Jantzen, \emph{Representations of algebraic groups}, second ed.,
  Mathematical Surveys and Monographs, vol. 107, American Mathematical Society,
  Providence, RI, 2003. \MR{2015057 (2004h:20061)}

\bibitem[JS92]{JaSe}
Jens~C. Jantzen and Gary~M. Seitz, \emph{On the representation theory of the
  symmetric groups}, Proc. London Math. Soc. (3) \textbf{65} (1992), no.~3,
  475--504. \MR{1182100}

\bibitem[L{\"u}b01]{Luebeck}
Frank L{\"u}beck, \emph{Small degree representations of finite {C}hevalley
  groups in defining characteristic}, LMS J. Comput. Math. \textbf{4} (2001),
  135--169 (electronic). \MR{1901354 (2003e:20013)}

\bibitem[McN98]{McNinch}
George~J. McNinch, \emph{Dimensional criteria for semisimplicity of
  representations}, Proc. London Math. Soc. (3) \textbf{76} (1998), no.~1,
  95--149. \MR{1476899 (99b:20076)}

\bibitem[Pre87]{Pr}
A.~A. Premet, \emph{Weights of infinitesimally irreducible representations of
  {C}hevalley groups over a field of prime characteristic}, Mat. Sb. (N.S.)
  \textbf{133(175)} (1987), no.~2, 167--183, 271. \MR{905003 (88h:20051)}

\bibitem[Sei87]{Se}
Gary~M. Seitz, \emph{The maximal subgroups of classical algebraic groups}, Mem.
  Amer. Math. Soc. \textbf{67} (1987), no.~365, iv+286. \MR{888704 (88g:20092)}

\bibitem[Ste63]{St2}
Robert Steinberg, \emph{Representations of algebraic groups}, Nagoya Math. J.
  \textbf{22} (1963), 33--56. \MR{0155937 (27 \#5870)}

\bibitem[Ste68]{St1}
\bysame, \emph{Lectures on {C}hevalley groups}, Yale University, New Haven,
  Conn., 1968, Notes prepared by John Faulkner and Robert Wilson. \MR{0466335
  (57 \#6215)}

\bibitem[Tes88]{Te}
Donna~M. Testerman, \emph{Irreducible subgroups of exceptional algebraic
  groups}, Mem. Amer. Math. Soc. \textbf{75} (1988), no.~390, iv+190.
  \MR{961210 (90a:20082)}

\end{thebibliography}
\bibliographystyle{amsalpha}     

\appendix
\section{Weight spaces for \texorpdfstring{$G$}{LG} of type \texorpdfstring{$A_l$}{LG}}                
\label{Weight spaces for G of type Al}     

Let $G$ be a simple algebraic group of type $A_l$ $(l\geqslant 2)$ over $K,$ fix a Borel subgroup $B=UT$ of $G$ as usual, and let $\Pi=\{\gamma_1,\ldots,\gamma_l\}$ be a corresponding base of the root system $\Phi=\Phi^+\sqcup \Phi^-$ of $G,$ where  $\Phi^+$ and $\Phi^-$ denote the sets of positive and negative roots of $G,$ respectively. Let $\{\sigma_1,\ldots,\sigma_l\}$ be the set of fundamental weights corresponding to our choice of base $\Pi$ and consider a standard Chevalley basis 
\[
\mathscr{B}=\{f_{\gamma},h_{\gamma_r},e_\gamma: \gamma\in \Phi^+, 1\leqslant r\leqslant l\}
\]
for the Lie algebra $\Lie(G)$ of \(G,\) as in Section \ref{structure constants and Chevalley basis}. For a dominant character $\sigma \in X^+(T),$ simply write $V(\sigma)$ (respectively, $L(\sigma)$) to denote the Weyl module for $G$ corresponding to $\sigma$ (respectively, the irreducible $KG$-module having highest weight $\sigma$). 

Throughout the first half of this section, we consider  $\sigma=a\sigma_1+b\sigma_l,$ where $a,b\in \Z_{\geqslant 1},$ and set $\mu=\sigma-(\gamma_1+\cdots+\gamma_l).$ Also for $1\leqslant r \leqslant s \leqslant l,$ we adopt the notation
\[
f_{r,s}=f_{\gamma_r+\cdots +\gamma_s},
\]
where \(f_{r,r}=f_{\gamma_r}\) for \(1\leqslant r\leqslant l\) by convention. By \eqref{Elements_generating_W(lambda)_mu} and our choice of ordering $\leqslant$ on $\Phi^+,$ the weight space $V(\sigma)_{\mu}$ is spanned by $f_{1,l}v^{\sigma}$ together with elements of the form $f_{1,r_1}f_{r_1+1,r_2}\cdots f_{r_m+1,l}v^{\sigma},$ where $v^{\sigma}\in V(\sigma)_{\sigma}$ denotes a maximal vector in $V(\sigma)$ for $B$ and $1\leqslant r_1<r_2<\ldots <r_m<l.$ Now observe that $\sigma-(\gamma_r+\cdots+\gamma_s)$ is a $T$-weight of $V(\sigma)$ if and only if  $r=1$ or $s=l.$ Therefore the list
\begin{equation}
\left\{f_{1,r}f_{r+1,l}v^{\sigma}\right\}_{1\leqslant r\leqslant l-1} \cup \left\{f_{1,l}v^{\sigma}\right\}
\label{[Dn<Bn] Preliminaries for L of type An: Generating elements for V(lambda)_mu}
\end{equation}
forms a generating set for \(V(\sigma)_\mu\) by Proposition  \ref{A simplification for V(lambda)mu}. Furthermore, an application of Proposition  \ref{Various_weight_multiplicities} yields $\dim V(\sigma)_{\mu}=l,$ forcing the generating elements of \eqref{[Dn<Bn] Preliminaries for L of type An: Generating elements for V(lambda)_mu} to be linearly independent, so that the following holds.

\begin{prop}\label{[Dn<Bn] Preliminaries for L of type An: A basis for V(lambda)_mu}
Let $\sigma=a  \sigma_1 + b\sigma_l \in X^+(T),$ where $a,b\in \Z_{\geqslant 1},$ and set $\mu=\sigma-(\gamma_1+\cdots+\gamma_l).$ Then $\mu$ is dominant and the set \eqref{[Dn<Bn] Preliminaries for L of type An: Generating elements for V(lambda)_mu} forms a basis of the weight space $V(\sigma)_{\mu}.$ 
\end{prop}

We now study the relation between the quadruple $(a,b,l,p)$ and the existence of a maximal vector of weight $\mu$ in $V(\sigma)$ for $B.$ For  $A=(A_r)_{1\leqslant r\leqslant l} \in K^l,$ set 
\begin{equation}
u(A) =\sum_{r=1}^{l-1}{A_r f_{1,r} f_{r+1,l}v^\sigma} + A_lf_{1,l}v^\sigma \in V(\sigma)_{\mu}.
\label{[Dn<Bn] Definition of u(A) for A=(Ai,...,Aj) and G of type An}
\end{equation}

\begin{lem}\label{[Dn<Bn] Prelimiaries for L of type An: existence of a maximal vector of weight lambda-1...1}
Let $\sigma,$ $\mu$ be as above, and adopt the notation of \eqref{[Dn<Bn] Definition of u(A) for A=(Ai,...,Aj) and G of type An}. Then the following assertions are equivalent. 
\begin{enumerate}
\item \label{ai+aj=i-j_part_1} There exists $0\neq A\in K^l$ such that $x_{\gamma}(c)u(A)=u(A)$ for every $\gamma \in \Pi$ and  \(c\in K.\)
\item \label{ai+aj=i-j_part_2} There exists $ A\in K^{l-1} \times K^{*}$ such that $x_{\gamma}(c)u(A)=u(A)$ for every $\gamma \in \Pi$ and  \(c\in K.\)
\item \label{ai+aj=i-j_part_3} The divisibility condition $p \mid a + b + l - 1$ is satisfied.
\end{enumerate}
\end{lem}

\begin{proof}
Let $A=(A_r)_{1\leqslant r\leqslant l} \in K^l$ and set $u=u(A).$ First observe that \(\tfrac{e_\gamma^i}{i!}u \in V(\sigma)_{\mu+i\gamma}\)  for \(i\in \mathbb{Z}_{\geqslant 0}\) and \(\gamma \in \Phi^+,\) and hence \(x_{\gamma}(c)u=u+ce_\gamma u\) (as \(V(\sigma)_{\mu+i\gamma}=0\) for \(i>1\)). In particular, this shows that  \(x_\gamma (c) u=u\) for  every \(c\in K\) if and only if \(e_{\gamma}u=0\) and so \ref{ai+aj=i-j_part_1} (respectively, \ref{ai+aj=i-j_part_2})  is equivalent to the existence of $0\neq A\in K^l$ (respectively, $ A\in K^{l-1} \times K^{*}$) such that \(e_\gamma u(A)=0\) for every \(\gamma\in \Pi.\)  Now applying Lemma \ref{structure_constants_for_A} successively yields
\begin{align*}
e_{\gamma_1}u				&= 	\sum_{r=1}^{l-1}{A_r e_{\gamma_1}f_{1,r} f_{r+1,l}v^\sigma} + A_ne_{\gamma_1}f_{1,l}v^\sigma 	\cr
							&=	(a+1) A_1 f_{2,l}v^{\sigma} -\sum_{r=2}^{l-1}{A_r f_{2,r}f_{r+1,l}v^{\sigma}} -A_l f_{2,l}v^{\sigma}\cr
										&= \left( (a +1)A_1 +\sum_{r=2}^{l-1}{A_r} -A_l \right)f_{2,l}v^{\sigma},\cr
e_{\gamma_r}		u			&= (A_r-A_{r-1})f_{1,r-1}f_{r+1,l}v^{\sigma}, 	\cr
e_{\gamma_l}u 				&= (A_l + b A_{l-1})f_{1,l-1}v^{\sigma},
\end{align*}
where $1<r<l.$ Observe that $e_{\gamma_{l-1}} \cdots e_{\gamma_2} f_{2,l}v^{\sigma}=\pm f_{\gamma_l}v^{\sigma} \neq 0,$ showing that $f_{2,l}v^{\sigma}\neq 0.$ Similarly, one checks that each of the vectors $f_{\gamma_1}f_{3,l}v^{\sigma},\ldots,f_{1,l-2}f_{\gamma_l}v^{\sigma},$ $f_{1,l-1}v^{\sigma}$ is non-zero. Therefore  	$e_{\gamma }u(A)=0$ for every $\gamma \in \Pi$ if and only if $A$ is a solution to the system of equations   
\begin{equation}
\left\{\begin{array}{rll}
A_l  &= (a+1)A_1 + \sum_{r=2}^{l-1}{A_r}								\cr
A_{r-1}  		 	&= A_r  \mbox{ for $  1 < r < l$}				\cr
A_l 					&= -b  A_{l-1}.

\end{array}
\right.
\label{first_system_of_equations}
\end{equation}

Finally, one easily sees that \eqref{first_system_of_equations} admits a non-trivial solution $A$ if and only if $p \mid a + b + l - 1$ (showing that \ref{ai+aj=i-j_part_1} and \ref{ai+aj=i-j_part_3} are equivalent), in which case $A\in \langle (1,\ldots,1,-b) \rangle_K$ (so that \ref{ai+aj=i-j_part_1} and \ref{ai+aj=i-j_part_2} are equivalent), completing the proof.
\end{proof}

Let $\sigma,$ $\mu$ be as above and consider an irreducible $KG$-module  $V=L(\sigma)$ having highest weight $\sigma.$ Take $V=V(\sigma)/\rad(\sigma)$ and write $v^+$ to denote the image of $v^{\sigma}$ in $V,$ that is, $v^+$ is a maximal vector of weight \(\sigma\) in $V$ for $B.$ By Proposition \ref{[Dn<Bn] Preliminaries for L of type An: A basis for V(lambda)_mu}, the weight space $V_{\mu}$ is spanned by the vectors
\begin{equation}
\left\{f_{1,r}f_{r+1,l}v^+\right\}_{1\leqslant r\leqslant l-1} \cup \left\{f_{1,l}v^+\right\}.
\label{generators_for_V_lambda-alpha_i-...-alphaj}
\end{equation}

We write $V_{1,l}$ to denote the span of all the generators in \eqref{generators_for_V_lambda-alpha_i-...-alphaj} except for $f_{1,l}v^+.$ The following result gives a precise description of the weight space $V_{\mu},$ as well as a characterization for \(\mu\) to afford the highest weight of a composition factor of \(V(\sigma).\)

\begin{prop}\label{Conditions_for_f(i,j)_to_belong_to_V(i,j)}
Let  $G$ be a simple algebraic group  of type $A_l$ over \(K\) and fix $a,b\in \Z_{\geqslant 1}.$ Consider an irreducible $KG$-module $V=L(\sigma)$ having $p$-restricted highest weight $\sigma=a  \sigma_1 + b \sigma_l$ and let $\mu=\sigma-(\gamma_1+\cdots+\gamma_l)\in \Lambda^+(\sigma).$ Then  the following assertions are equivalent.
\begin{enumerate}
\item \label{conditions_for_f(i,i)_to_belong_to_V(i,j)_part_1} The weight $\mu$ affords the highest weight of a composition factor of $V(\sigma).$
\item \label{conditions_for_f(i,i)_to_belong_to_V(i,j)_part_2} The generators in \eqref{generators_for_V_lambda-alpha_i-...-alphaj} are linearly dependent.
\item \label{conditions_for_f(i,i)_to_belong_to_V(i,j)_part_3} The element $f_{1,l}v^+$ lies inside $V_{1,l}.$
\item \label{conditions_for_f(i,i)_to_belong_to_V(i,j)_part_4} The divisibility condition $p \mid a + b + l - 1$ is satisfied.
\end{enumerate}
\end{prop}

\begin{proof}
Clearly both \ref{conditions_for_f(i,i)_to_belong_to_V(i,j)_part_1} and \ref{conditions_for_f(i,i)_to_belong_to_V(i,j)_part_3} imply \ref{conditions_for_f(i,i)_to_belong_to_V(i,j)_part_2}. Also if \ref{conditions_for_f(i,i)_to_belong_to_V(i,j)_part_2} holds, then $\rad (\sigma) \cap V(\sigma)_{\mu}\neq 0,$ so $L(\nu)$ occurs as a composition factor of $V(\sigma)$ for some $\nu\in \Lambda^+(\sigma)$ such that $\mu \preccurlyeq \nu \prec \sigma.$ Now one easily sees that $\m_{V(\sigma)}(\nu)=1$ for every $\mu \prec \nu \prec \sigma,$ which by Theorem \ref{[Preliminaries] Weights and multiplicities: Premet's theorem} forces \(\nu=\mu,\) so that  \ref{conditions_for_f(i,i)_to_belong_to_V(i,j)_part_1} holds. Still assuming  \ref{conditions_for_f(i,i)_to_belong_to_V(i,j)_part_2}, this also shows that there exists $0\neq A\in K^l$ such that $u(A) \in V(\sigma)_{\mu}$ is a maximal vector in $V(\sigma)$ for $B,$ where we adopt the notation of \eqref{[Dn<Bn] Definition of u(A) for A=(Ai,...,Aj) and G of type An}. Therefore \ref{conditions_for_f(i,i)_to_belong_to_V(i,j)_part_2} implies \ref{conditions_for_f(i,i)_to_belong_to_V(i,j)_part_4} as well by Lemma \ref{[Dn<Bn] Prelimiaries for L of type An: existence of a maximal vector of weight lambda-1...1}. Finally, suppose that \ref{conditions_for_f(i,i)_to_belong_to_V(i,j)_part_4} holds. By Lemma \ref{[Dn<Bn] Prelimiaries for L of type An: existence of a maximal vector of weight lambda-1...1}, there exists $A\in K^{l-1}\times K^*$ such that $x_{\gamma}(c)u(A)=u(A)$ for every $\gamma\in \Pi$ and \(c\in K.\)  Consequently, we also get $x_{\gamma}(c)(u(A) + \rad(\sigma))=u(A) + \rad(\sigma)$ for every $\gamma\in \Pi$ and \(c\in K,\) that is, $u(A)+\rad(\sigma)\in \langle v^+ \rangle_K \cap V_{\mu}=0.$ Hence   \ref{conditions_for_f(i,i)_to_belong_to_V(i,j)_part_3} holds and the proof is complete.
\end{proof}

\begin{remark}\label{A_maximal_vector_in_V_(A_n)(a,0,....,0,b)}
Let \(\sigma,\) \(\mu\) be as above and assume $p \mid a + b  +l -1.$ By Proposition \ref{Conditions_for_f(i,j)_to_belong_to_V(i,j)}, $\mu$ affords the highest weight of a composition factor of $V(\sigma)$ and $f_{1,l}v^+\in V_{1,l}.$ Moreover, the proof of Lemma \ref{[Dn<Bn] Prelimiaries for L of type An: existence of a maximal vector of weight lambda-1...1} showed that
\[
u^+ = f_{1,l}v^{\sigma} -b^{-1}\sum_{r=1}^{l-1}{f_{1,r}f_{r+1,l}v^{\sigma}}
\]
is a maximal vector of weight $\mu$ in $V(\sigma)$ for $B,$ leading to a precise description of $f_{1,l}v^+$ in terms of a basis of $V_{1,l}.$ This information shall prove of capital importance later in the paper.
\end{remark}

The remainder of this section is devoted to the proof of Theorem  \ref{How_to_relate_divisibility_conditions_and_generators_for_V_i,j}, which consists of a generalization of Proposition \ref{Conditions_for_f(i,j)_to_belong_to_V(i,j)}  to irreducible \(KG\)-modules having arbitrary \(p\)-restricted highest weights. For $1\leqslant i<j\leqslant l,$ set  
\[
P(i,j)=\left\{(m^{i,j}_{r})_{r=1} ^s: 1\leqslant s\leqslant j-i,~ i\leqslant m^{i,j}_1<\ldots < m^{i,j}_s < j\right\}
\]
and for any sequence $(m)=(m^{i,j}_r)_{r=1}^s\in P(i,j),$ write \(f_{(m)}=f_{i,m_1}f_{m_1+1,m_2}\cdots f_{m_s+1,j}.\) (Also adopt the convention $P(i,i)=\emptyset$ for $1\leqslant i\leqslant l.$) 

Consider an irreducible \(KG\)-module having  \(p\)-restricted highest weight \(\sigma=\sum_{r=1}^n{a_r\sigma_r}\in X^+(T).\) By  \eqref{Elements_generating_W(lambda)_mu}, we get that for every \(1\leqslant i \leqslant j\leqslant l,\) the weight space $V_{\sigma-(\gamma_i+\cdots+\gamma_j)}$ is  spanned by the vectors
\begin{equation*}
\{f_{(m)}v^+ : (m)\in P(i,j)\} \cup \{f_{i,j}v^+\},
\end{equation*} 
where \(v^+\) is a maximal vector of weight \(\sigma\) in \(V\) for \(B.\) As before, we let $V_{i,j}$ denote the span of all the above generators except for $f_{i,j}v^+.$

\begin{lem}\label{If_divisibility_conditions_holds,_then_f(i,j)_in_V(i,j)}
Let \(V\) be as above and assume \(f_{r,s}v^+\in V_{r,s}\)  for every \(1\leqslant r<s \leqslant l\) such that \(a_r a_s\neq 0\) and \(a_t=0\) for \(r<t<s.\) Then $f_{i,j}v^+\in V_{i,j}$ for every \(1\leqslant i<j\leqslant l\) such that \(\{i<k\leqslant j: a_k\neq 0\}\neq \emptyset.\)
\end{lem}

\begin{proof}
Let $1\leqslant i<  j\leqslant l$ be as in the statement of the lemma. If \(a_i=0,\) then \(f_{i,j}v^+=-f_{\gamma_i}f_{i+1,j}v^+\) by Lemma \ref{structure_constants_for_A} and hence the assertion obviously holds. Therefore assume \(a_i\neq 0\) in the remainder of the proof and let \(i<k\leqslant j\) be minimal such that \(a_k\neq 0.\)  Since \(f_{i,j}v^+= f_{k+1,j}f_{i,k}v^+   - f_{i,k}f_{k+1,j}v^+\) (thanks to Lemma \ref{structure_constants_for_A} again), we get  that   $f_{i,j}v^+ \in V_{i,j}$ if and only if $ f_{k+1,j}f_{i,k}v^+\in V_{i,j}.$ Now $f_{i,k}v^+ \in V_{i,k}$ by assumption and thus there exists \(\{\xi_{(m)} :(m)\in P(i,k)\}\subset K\) such that
\[
f_{i,k}v^+ = \sum_{(m)\in P(i,k)}{ \xi_{(m)}f_{(m)}v^+}.
\]
Observe that for any fixed \((m)=(m_r)_{r=1}^s \in P(i,k)\) the element \(f_{k+1,j}\) commutes with \(f_{i,m_1}\) as well as with \(f_{m_r+1,m_{r+1}}\) for \(1\leqslant r < s.\) Therefore we have  
\[
f_{k+1,j}f_{(m)}v^+ = f_{i,m_1}f_{m_1+1,m_2}\cdots f_{m_{s-1}+1,m_s}f_{k+1,j} f_{m_s+1,k}v^+.
\]
Finally, \(f_{k+1,j}f_{m_s+1,k}v^+ = f_{m_s+1,j}v^+ + f_{m_s+1,k}f_{k+1,j}v^+\) by Lemma \ref{structure_constants_for_A} and hence \(f_{k+1,j}f_{(m)}v^+ \in V_{i,j}.\) Since \((m)\) was arbitrarily chosen in \( P(i,k),\)  the result follows. 
\end{proof}

The next result is a special case of \cite[Proposition 3.1]{Fo1} (where \(i\) and \(m\) are replaced by \(1\) and \(l,\) respectively), to which we refer the reader for a proof.

\begin{prop}\label{Ford's_Proposition}
Let \(V\) be as above, with \(a_1a_l \neq 0.\)  Also let \(1 < j \leqslant l\) be minimal such that \(a_j\neq 0\) and suppose that \(f_{r,l}v^+\in V_{r,l}\) for every \(1\leqslant r<l.\) Then \(f_{1,j}v^+\in V_{1,j}.\) 
\end{prop}

We conclude this section by the following theorem, which shall   play a key role in the proof of Theorem \ref{Main result}. Here again, one could easily generalize the result (considering   \(1\leqslant r<s\leqslant l\) instead of \(1\leqslant r<l\) in the statement of the latter, for example).

\begin{thm}\label{How_to_relate_divisibility_conditions_and_generators_for_V_i,j}
Let \(G\) be a simple algebraic group of type \(A_l\) over \(K,\) and consider an irreducible \(KG\)-module \(V=L(\sigma)\) having \(p\)-restricted highest weight \(\sigma=\sum_{r=1}^l{a_r \sigma_r},\) with \(a_l\neq 0.\) Then \(f_{r,l}v^+\in V_{r,l}\) for every \(1\leqslant r<l\) if and only if \(p\mid a_i+a_j+j-i\) for every \(1\leqslant i<j\leqslant l\) such that \(a_ia_j\neq 0\) and \(a_s=0\) for \(i<s<j.\)
\end{thm}

\begin{proof}
First assume the divisibility conditions stated in the theorem hold and let \(1\leqslant r<l.\) If \(a_r=0,\) then \(f_{\gamma_r}v^+ =0 \)  and thus  \(f_{r,l}v^+ = -  f_{\gamma_r}f_{r+1,l}v^+ \in V_{r,l}\) as desired. If  on the other hand \(a_r\neq 0,\) then successively applying Proposition \ref{Conditions_for_f(i,j)_to_belong_to_V(i,j)} to  suitable Levi subgroups of \(G\) shows that \(f_{i,j}v^+\in V_{i,j}\) for every \(r\leqslant i<j\leqslant l\) such that \(a_ia_j\neq 0\) and \(a_k=0\) for \(i<k<j.\) The result then follows from Lemma \ref{If_divisibility_conditions_holds,_then_f(i,j)_in_V(i,j)}. Conversely, assume \(f_{r,l}v^+\in V_{r,l}\) for every \(1\leqslant r< l\) and let \(1\leqslant i<j\leqslant l\) be as in the statement of the theorem. Applying  Proposition \ref{Ford's_Proposition} to the Levi subgroup corresponding to the simple roots \(\gamma_i,\ldots,\gamma_l\) then yields \(f_{i,j}v^+\in V_{i,j}.\) Finally, an application of  Proposition \ref{Conditions_for_f(i,j)_to_belong_to_V(i,j)} yields \(p\mid a_i+a_j+j-i,\) thus completing the proof.
\end{proof}
\newpage

\section{Weight spaces for \texorpdfstring{$G$}{LG} of type \texorpdfstring{$B_n$}{LG}}										
\label{Weight spaces for G of type Bn}     

Let $G$ be a simple algebraic group of type $B_n$ $(n\geqslant 3)$ over $K,$ fix a Borel subgroup $B=UT$ of $G$ as usual, and let $\Pi=\{\alpha_1,\ldots,\alpha_n\}$ be a corresponding base of the root system $\Phi=\Phi^+\sqcup \Phi^-$ of $G,$ where  $\Phi^+$ and $\Phi^-$ denote the sets of positive and negative roots of $G,$ respectively. Let $\{\lambda_1,\ldots,\lambda_n\}$ be the set of fundamental weights corresponding to our choice of base $\Pi$ and consider a standard Chevalley basis 
\[
\mathscr{B}=\{f_{\alpha},h_{\alpha_r},e_\alpha: \alpha\in \Phi^+, 1\leqslant r\leqslant n\}
\]
for the Lie algebra $\Lie(G)$ of \(G,\) as in Section \ref{structure constants and Chevalley basis}. For $\lambda \in X^+(T),$ simply write $V(\lambda)$ (respectively, $L(\lambda)$) to denote the Weyl module for $G$ corresponding to $\lambda$ (respectively, the irreducible $KG$-module having highest weight $\lambda$). Although most of the results presented here hold for $K$ having arbitrary characteristic, we shall assume $p\neq 2$ throughout this section for simplicity. Finally, we adopt the notation $f_{i,j}=f_{\alpha_i+\cdots +\alpha_j},$ for  $1\leqslant i\leqslant j\leqslant n$ (where \(f_{i,i}=f_{\alpha_i}\) for \(1\leqslant i\leqslant n\) by convention), as well as
\[
F_{r,s}=f_{\alpha_r+\cdots+\alpha_{s-1}+2\alpha_s+\cdots+2\alpha_n},
\]
for  $1\leqslant r < s \leqslant n.$ (Here we set \(F_{r,r+1}=f_{\alpha_r+2\alpha_{r+1}+\cdots+2\alpha_n}\) and \(F_{r,n}=f_{\alpha_r+\cdots+\alpha_{n-1}+2\alpha_n}\) for \(1\leqslant r < n\) by convention.)

\subsection{Study of \texorpdfstring{$L(a\lambda_1)$}{LG} \texorpdfstring{$(a\in \Z_{>0})$}{LG}}     

Let $a\in \Z_{>0}$ and consider the dominant character $\lambda=a\lambda_1\in X^+(T).$ Also write $\mu=\lambda-2(\alpha_1+\cdots+\alpha_n).$ By Proposition \ref{A simplification for V(lambda)mu} (recall that \(p\neq 2\) here) together with our choice of ordering $\leqslant$ on $\Phi^+,$ one sees that
\[
V(\lambda)_{\mu}=\left\langle \tfrac{1}{2}(f_{1,n})^2v^{\lambda}, f_{1,j}F_{1,j+1}v^{\lambda} :1\leqslant j < n\right\rangle_K,
\]
where $v^{\lambda}\in V(\lambda)_{\lambda}$ denotes a maximal vector of weight $\lambda$ in $V(\lambda)$ for $B.$ Again, since we are assuming $p\neq 2,$ we get that  $\tfrac{1}{2}(f_{1,n})^2v^{\lambda} \in V(\lambda)_{\mu}$ if and only if $ (f_{1,n})^2v^{\lambda} \in V(\lambda)_{\mu}$, so that the weight space $V(\lambda)_{\mu}$ is spanned by the vectors   
 \begin{equation}
\{f_{1,j}F_{1,j+1}v^{\lambda}\}_{1\leqslant j < n} \cup \{(f_{1,n})^2v^{\lambda}\}.
\label{W_k,n}
\end{equation}

Now if $a=1,$ then $\mu$ is $\mathscr{W}$-conjugate to $\lambda,$  which has multiplicity $1$ in $V(\lambda).$ Furthermore, successively applying $e_{\alpha_1},\ldots,e_{\alpha_n}$ to  the element $f_{\alpha_1}F_{1,2}v^{\lambda}$ shows that it is non-zero, hence  $V(\lambda)_{\mu}=\langle f_{\alpha_1}F_{1,2}v^{\lambda}\rangle_K.$ Finally, we leave to the reader to check (using Lemma \ref{structure_constants_for_B} together with the fact that $V(\lambda)_{\lambda-(2\alpha_1+\alpha_2+\cdots+\alpha_n)}=0$)  that the following result holds.

\begin{prop}\label{[Dn<Bn] Preliminaries for L of type Bn: technical proposition concerning the weight lambda-2...2 for lambda=lambda_1}
Let $\lambda=\lambda_1$ and consider $\mu=\lambda-2(\alpha_1+\cdots+\alpha_n)\in \Lambda(\lambda).$ Then $V(\lambda)_{\mu}=\langle f_{\alpha_1}F_{1,2}v^{\lambda}\rangle_K$ and the following assertions hold.
\begin{enumerate}
\item \label{[Dn<Bn] Preliminaries for L of type Bn: technical proposition concerning the weight lambda-2...2 for lambda=lambda_1 (part 1)} $f_{1,j}F_{1,j+1}v^{\lambda}= f_{\alpha_1}F_{1,2}v^{\lambda}$ for every $1\leqslant j<n.$
\item \label{[Dn<Bn] Preliminaries for L of type Bn: technical proposition concerning the weight lambda-2...2 for lambda=lambda_1 (part 2)} $(f_{1,n})^2 v^{\lambda} = 2f_{\alpha_1}F_{1,2}v^{\lambda}.$
\end{enumerate}
\end{prop}

For the remainder of this section, we assume $a>1,$ in which case the weight $\mu$ is dominant. An application of Proposition \ref{Various_weight_multiplicities} gives $\dim V(\lambda)_{\mu}=n,$ so that the generating elements of \eqref{W_k,n} are linearly independent, leading to the following result.

\begin{prop}\label{A_basis_of_W(lambda)_(lambda-2alphak-...-2alphan)}
Let $\lambda=a\lambda_1\in X^+(T),$ where $a\in \Z_{>1},$ and consider $\mu=\lambda-2(\alpha_1+\cdots+\alpha_n).$ Then $\mu$ is dominant and the set \eqref{W_k,n} forms a basis of the weight space $V(\lambda)_{\mu}.$ 
\end{prop}

We now study the relation between the triple $(a,n,p)$ and the existence of a maximal vector of weight $\mu$ in $V(\lambda)$ for $B.$ For  $A=(A_r)_{1\leqslant r\leqslant n} \in K^n,$ set
\begin{equation}
w(A) = \sum_{j=1}^{n-1}{A_j f_{1,j}F_{1,j+1}v^{\lambda}} + A_n(f_{1,n})^2v^{\lambda}.
\label{[Dn<Bn] Definition of w(A) for A=(A1,...,An), mu=lambda-2....2 and G of type Bn}
\end{equation}

\begin{lem}\label{Conditions_pour_que_la_combinaison_lineaire_de_vecteurs_generateurs_de_V(k,n)_soit_tuee_par_tous_les_e_alpha}
Let $\lambda,$ $\mu$ be as above and adopt the notation of \eqref{[Dn<Bn] Definition of w(A) for A=(A1,...,An), mu=lambda-2....2 and G of type Bn}. Then the following assertions are equivalent. 
\begin{enumerate}
\item \label{V(k,n)_part_1} There exists $0 \neq A \in K^n$ such that $x_{\alpha}(c)w(A)=0$ for every $\alpha\in \Pi$ and   \(c\in K.\)
\item \label{V(k,n)_part_2} There exists $A\in K^{n-1} \times K^{*}$ such that $x_{\alpha}(c)w(A)=0$ for every $\alpha\in \Pi$ and   \(c\in K.\)
\item \label{V(k,n)_part_3} The divisibility condition $p \mid 2(a+n)-3$ is satisfied.
\end{enumerate}
\end{lem}

\begin{proof}
Let $A=(A_r)_{1\leqslant r\leqslant n}\in K^n$ and set $w=w(A).$ Arguing as in the proof of Lemma \ref{[Dn<Bn] Prelimiaries for L of type An: existence of a maximal vector of weight lambda-1...1}, one sees that \(x_\alpha (c) w=w\) for every \(\alpha \in \Pi\) and \(c\in K\) if and only if \(e_{\alpha}w=0\) for every \(\alpha\in \Pi.\) Also applying Lemma \ref{structure_constants_for_B} successively yields
\begin{align*}
e_{\alpha_1}w		&= aA_1F_{1,2}v^{\lambda} - \sum_{j=2}^{n-1}{A_j f_{2,j}F_{1,j+1}v^{\lambda}} -A_nf_{2,n}f_{1,n}v^{\lambda}  \cr
								&=\bigg(aA_1 +\sum_{j=2}^{n-1}{A_j} + 2A_n\bigg)F_{1,2}v^{\lambda}, \cr 
e_{\alpha_r}w 		&= (A_r-A_{r-1})f_{1,r-1}F_{1,r+1}v^{\lambda}, 	\cr
e_{\alpha_n}w	 	&=(4A_n-A_{n-1})f_{1,n-1}f_{1,n}v^{\lambda},
\end{align*}
where $1<r<n.$ As in the proof of Lemma \ref{[Dn<Bn] Prelimiaries for L of type An: existence of a maximal vector of weight lambda-1...1}, one checks that each of the vectors $F_{1,2}v^{\lambda},$ $f_{\alpha_1}F_{1,3}v^{\lambda},\ldots,f_{1,n-2}F_{1,n}v^{\lambda},$ $f_{1,n-1}f_{1,n}v^{\lambda}$ is non-zero, so that $e_{\alpha}w(A)=0$ for every $\alpha\in \Pi$ if and only if $A\in K^n$ is a solution to the system of equations 
\begin{equation}
\left\{\begin{array}{rll}
2A_n+aA_1 		&= -\sum_{r=2}^{n-1}{A_r}								\cr
A_{r-1}   		&= A_r \mbox{ for  $1<r <n $}			\cr
A_{n-1} 			&= 4A_n.
\end{array}
\right.
\label{second_system_of_equations}
\end{equation}

One easily sees that \eqref{second_system_of_equations} admits a non-trivial solution $A\in K^n$ if and only if $p\mid 2(a+n)-3$ (showing that \ref{V(k,n)_part_1} and \ref{V(k,n)_part_3} are equivalent), in which case $A\in \langle(4,\ldots,4,1)\rangle_K$ (so that \ref{V(k,n)_part_1} and \ref{V(k,n)_part_2} are equivalent), completing the proof.
\end{proof}

Let $\lambda,$ $\mu$ be as above and consider an irreducible $KG$-module $V=L(\lambda)$ having  highest weight $\lambda.$ Take $V=V(\lambda)/\rad(\lambda)$ and write $v^+$ to denote the image of $v^{\lambda}$ in $V=L(\lambda),$ that is, $v^+$ is a maximal vector of weight \(\lambda\) in $V$ for $B.$  By \eqref{W_k,n} and our choice of ordering $\leqslant$ on $\Phi^+,$ the weight space $V_{\mu}$ is spanned by the vectors   
\begin{equation}
\left\{f_{1,j}F_{1,j+1}v^{+}\right\}_{1\leqslant j<n} \cup \left\{(f_{1,n})^2v^{+}\right\}.
\label{V_k,n}
\end{equation}

We write $V_{1,n}^2$ to denote the span of all the generators in \eqref{V_k,n} except for $(f_{1,n})^2v^+.$ The following result gives a precise description of the weight space $V_{\mu},$ as well as a characterization for \(\mu\) to afford the highest weight of a composition factor of \(V(\lambda).\)

\begin{prop}\label{Conditions_for_f(k,n)_to_belong_to_V(k,n)}

Let  $G$ be a simple algebraic group  of type $B_n$ over $K$ and fix $a\in \Z_{>1}.$  Also consider an irreducible $KG$-module $V=L(\lambda)$ having $p$-restricted highest weight $\lambda=a\lambda_1\in X^+(T)$ and let $\mu=\lambda-2(\alpha_1+\cdots+\alpha_n).$ Then  the following assertions are equivalent.
\begin{enumerate}
\item \label{conditions_for_f(i,i)_to_belong_to_V(k,n)_part_1} The weight $\mu$ affords the highest weight of a composition factor of $V(\lambda).$
\item \label{conditions_for_f(i,i)_to_belong_to_V(k,n)_part_2} The generators in \eqref{V_k,n} are linearly dependent.
\item \label{conditions_for_f(i,i)_to_belong_to_V(k,n)_part_3} The element $(f_{1,n})^2v^+$ lies inside $V^2_{1,n}.$
\item \label{conditions_for_f(i,i)_to_belong_to_V(k,n)_part_4} The divisibility condition $p\mid 2(a+n)-3$ is satisfied.
\end{enumerate}
\end{prop}

\begin{proof}
First observe that the weights $\nu \in \Lambda^+(\lambda)$ such that $\mu \prec \nu \prec \lambda$ are $\lambda-\alpha_1,$ $\lambda-2\alpha_1-\alpha_2$ (if $a>2$), $\lambda-(\alpha_1+\cdots+\alpha_n)$ and $\lambda-(2\alpha_1+\alpha_2+\cdots+\alpha_n)$ (if $a>2$), which all satisfy $\m _{V(\lambda)}(\nu)=1.$ Therefore none of the latter can afford the highest weight of a composition factor of \(V(\lambda)\) by Theorem \ref{[Preliminaries] Weights and multiplicities: Premet's theorem}. Proceeding exactly as in the proof of Proposition \ref{Conditions_for_f(i,j)_to_belong_to_V(i,j)}, using Lemma \ref{Conditions_pour_que_la_combinaison_lineaire_de_vecteurs_generateurs_de_V(k,n)_soit_tuee_par_tous_les_e_alpha} instead of Lemma \ref{[Dn<Bn] Prelimiaries for L of type An: existence of a maximal vector of weight lambda-1...1} then yields the desired result. We leave the details to the reader.
\end{proof}

\subsection{Study of \texorpdfstring{$L(\lambda_i)$}{LG} \texorpdfstring{$(1 < i < n)$}{LG}}     

Let $\lambda=\lambda_2$ and consider $\mu=\lambda-(\alpha_{1}+2\alpha_2+\cdots+2\alpha_n).$ (Observe that $\mu$ is the zero weight.) By (\ref{Elements_generating_W(lambda)_mu}), our choice of ordering $\leqslant$ on $\Phi^+,$ and Proposition \ref{[Dn<Bn] Preliminaries for L of type Bn: technical proposition concerning the weight lambda-2...2 for lambda=lambda_1} (applied to the   Levi subgroup corresponding to the simple roots \(\alpha_2,\ldots,\alpha_n\)), one checks that the weight space $V(\lambda)_{\mu}$ is spanned by the vectors
\begin{align*}
\{F_{1,2}v^{\lambda}\}			&\cup 				\{f_{\alpha_1}f_{\alpha_2}F_{2,3}v^{\lambda}\} 					\cr
								&\cup 				\{f_{1,j}F_{2,j+1}v^{\lambda}\}_{2\leqslant j < n}						\cr
								&\cup 				\{f_{2,j}F_{1,j+1} v^{\lambda}\}_{2\leqslant j < n} 					\cr
								&\cup					\{f_{2,n}f_{1,n}v^{\lambda}\},													\end{align*} 
where \(v^\lambda\) is a maximal vector of weight \(\lambda\) in \(V(\lambda)\) for \(B.\)

\begin{prop}\label{basis_of_W_lambda-alpha_k-1-2alpha_k-...-2alpha_n}
Let $\lambda=\lambda_2$ and set $\mu=\lambda-(\alpha_1+2\alpha_2+\cdots+2\alpha_n)\in \Lambda^+(\lambda).$ Then $\m_{V(\lambda)}(\mu)=n$ and a basis of $V(\lambda)_{\mu}$ is given by 
\begin{align}
\{F_{1,2}v^{\lambda}\} 			&\cup \{f_{\alpha_1}f_{\alpha_2}F_{2,3}v^{\lambda}\}				\cr
																				&\cup \{f_{2,j}F_{1,j+1} v^{\lambda}\}_{2\leqslant j < n}.
\label{W(lambda)_lambda-alpha_k-1-2alpha_k-...-2alpha_n}
\end{align}
\end{prop}

\begin{proof}
The assertion on \(\dim V(\lambda)_\mu\) holds by Proposition \ref{Various_weight_multiplicities}, hence it remains to show that $f_{1,j}F_{2,j+1}v^{\lambda}$ $(2\leqslant j<n)$ and $f_{2,n}f_{1,n}v^{\lambda}$ can be expressed as linear combinations of elements of \eqref{W(lambda)_lambda-alpha_k-1-2alpha_k-...-2alpha_n}. Let then $2\leqslant j<n$ be fixed. By Lemma \ref{structure_constants_for_B} and Proposition \ref{[Dn<Bn] Preliminaries for L of type Bn: technical proposition concerning the weight lambda-2...2 for lambda=lambda_1} \ref{[Dn<Bn] Preliminaries for L of type Bn: technical proposition concerning the weight lambda-2...2 for lambda=lambda_1 (part 1)}, applied to the $B_{n-1}$-Levi subgroup corresponding to the simple roots $\alpha_2,\ldots,\alpha_n$ (noticing that the structure constants  were chosen in a compatible way in Section \ref{structure constants and Chevalley basis}), we have 
\begin{align*}
f_{1,j}F_{2,j+1}v^{\lambda} 	&= f_{2,j}f_{\alpha_1}F_{2,j+1}v^{\lambda} -f_{\alpha_1}f_{2,j}F_{2,j+1}v^{\lambda} 		\cr
															&= - f_{2,j}F_{1,j+1}v^{\lambda} - f_{\alpha_1}f_{\alpha_2}F_{2,3}v^{\lambda},
\end{align*}
that is, $f_{1,j}F_{2,j+1}v^{\lambda} \in \left\langle f_{2,j}F_{1,j+1}v^{\lambda}, f_{\alpha_1}f_{\alpha_2}F_{2,3}v^{\lambda} \right\rangle_K.$ In a similar way, Lemma   \ref{structure_constants_for_B} and case \ref{[Dn<Bn] Preliminaries for L of type Bn: technical proposition concerning the weight lambda-2...2 for lambda=lambda_1 (part 2)} of Proposition \ref{[Dn<Bn] Preliminaries for L of type Bn: technical proposition concerning the weight lambda-2...2 for lambda=lambda_1}  (applied to the $B_{n-1}$-Levi subgroup corresponding to the simple roots $\alpha_2,\ldots,\alpha_n$) yield 
\begin{align*}
f_{2,n}f_{1,n}v^{\lambda} 			&=  -2F_{1,2}v^{\lambda} + f_{1,n}f_{2,n}v^{\lambda} \cr
																&=  -2F_{1,2}v^{\lambda} -f_{\alpha_1}(f_{2,n})^2 v^{\lambda} + f_{2,n}f_{\alpha_1}f_{2,n}v^{\lambda} \cr
																&=  -2F_{1,2}v^{\lambda} -2f_{\alpha_1}f_{\alpha_2}F_{2,3} v^{\lambda} - f_{2,n}f_{1,n}v^{\lambda},
\end{align*}
so that $f_{2,n}f_{1,n}v^{\lambda} = - F_{1,2}v^{\lambda} - f_{\alpha_1}f_{\alpha_2}F_{2,3}v^{\lambda}.$ Therefore $f_{2,n}f_{1,n}v^{\lambda}$ lies in the subspace of $V(\lambda)_\mu$ generated by $ F_{1,2}v^{\lambda}$ and $f_{\alpha_1}f_{\alpha_2}F_{2,3}v^{\lambda}$ as desired, thus completing the proof.
\end{proof}

Let $\lambda,$ $\mu$ be as above and consider an irreducible $KG$-module $V=L(\lambda)$ having highest weight $\lambda.$ As usual, take $V=V(\lambda)/\rad(\lambda)$ and write $v^+$ to denote the image of $v^{\lambda}$ in $V,$ that is, $v^+$ is a maximal vector of weight \(\lambda\) in $V$ for $B.$ By Proposition \ref{basis_of_W_lambda-alpha_k-1-2alpha_k-...-2alpha_n}, the weight space $V_{\mu}$ is spanned by the vectors
\begin{align}
\left\{F_{1,2}v^+\right\}	&\cup	\left\{f_{\alpha_1}f_{\alpha_2}F_{2,3}v^+\right\}				\cr
						 	&\cup \left\{f_{2,j}F_{1,j+1} v^+\right\}_{2\leqslant j < n}.
\label{[Dn<Bn] Preliminaries for L of type Bn: a basis of L(lambda2)_mu, where mu=lambda-12...2}
\end{align}
Now by \cite[Theorems 4.4 and 5.1]{Luebeck}, the $KG$-module $V(\lambda)$ is irreducible (recall that $p\neq 2$),  which in particular yields the following result.

\begin{prop}\label{[Dn<Bn] Preliminaries for L of type Bn: irreducibility of V(lambda2) and basis of L(lambda2)_mu, where mu=lambda-12...2}
Consider an irreducible $KG$-module $V=L(\lambda)$ having highest weight $\lambda=\lambda_2.$ Then $V=V(\lambda)$ and the $T$-weight $\mu=\lambda-(\alpha_1+2\alpha_2+\cdots+2\alpha_n)$ is dominant. Also $\m_V(\mu)=n$ and the set \eqref{[Dn<Bn] Preliminaries for L of type Bn: a basis of L(lambda2)_mu, where mu=lambda-12...2} forms a basis of $V_{\mu}.$
\end{prop}

Finally,  consider an irreducible $KG$-module $V=L(\lambda)$ having highest weight $\lambda=\lambda_i$ for some $1<i<n$ and set  $\mu=\lambda-(\alpha_1+\cdots+\alpha_{i-1} +2\alpha_i+\cdots+2\alpha_n).$ Proceeding exactly as in the proof of Proposition \ref{basis_of_W_lambda-alpha_k-1-2alpha_k-...-2alpha_n}, one easily deduces that the weight space $V_{\mu}$ is spanned by the vectors
\begin{align}
\left\{F_{1,i}v^+\right\} 							&\cup \left\{f_{1,i-1}f_{\alpha_i}F_{i,i+1}v^+\right\}		\cr																			&\cup \left\{f_{i,j}F_{1,j+1} v^+\right\}_{i\leqslant j < n},
\label{[Dn<Bn] Preliminaries for L of type Bn: a basis of L(lambda_i)_mu, where mu=lambda-1...12...2}
\end{align}
where $v^+$ is a maximal vector of weight $\lambda$ in $V$ for $B.$ Hence applying Lemma \ref{[Preliminaries] Parabolic embeddings: restriction to Levi subgroup} to the $B_{n-i+2}$-Levi subgroup of $G$ corresponding to the simple roots $\alpha_{i-1},\ldots,\alpha_n,$ together with Proposition \ref{[Dn<Bn] Preliminaries for L of type Bn: irreducibility of V(lambda2) and basis of L(lambda2)_mu, where mu=lambda-12...2}, yields the following result. The details are left to the reader.

\begin{prop}\label{basis_of_W_lambda-alpha_i-2alpha_k-...-2alpha_n}
Let \(1<i<n\) and consider an irreducible $KG$-module $V=L(\lambda)$ having highest weight $\lambda=\lambda_i.$ Then the $T$-weight $\mu=\lambda-(\alpha_1+\cdots+\alpha_{i-1} +2\alpha_i+\cdots+2\alpha_n)$ is dominant, $\m_V(\mu)=n-i+2,$ and the set \eqref{[Dn<Bn] Preliminaries for L of type Bn: a basis of L(lambda_i)_mu, where mu=lambda-1...12...2} forms a basis of $V_{\mu}.$ 
\end{prop}

\subsection{Study of \texorpdfstring{$L(a\lambda_1+\lambda_2)$}{LG} \texorpdfstring{$(a \in \Z_{>0})$}{LG}}    

Let $a\in \Z_{>0}$ and set  $\lambda=a\lambda_1+\lambda_2.$ Also write $\mu_{1,2}=\lambda-\alpha_1-\alpha_2$ and  consider $\mu=\lambda-(\alpha_1+2\alpha_2+\cdots+2\alpha_n).$  By Proposition \ref{A simplification for V(lambda)mu}, our choice of ordering $\leqslant$ on $\Phi^+$ and Proposition \ref{[Dn<Bn] Preliminaries for L of type Bn: technical proposition concerning the weight lambda-2...2 for lambda=lambda_1} (applied to the Levi subgroup of type $B_{n-1}$ corresponding to the simple roots $\alpha_2,\ldots,\alpha_n$), one sees that the weight space $V(\lambda)_{\mu}$ is spanned by the vectors  
\begin{align}
\{F_{1,2}v^{\lambda}\} 			&\cup \{f_{\alpha_1}f_{\alpha_2}F_{2,3}v^{\lambda}\} 						\cr 
																				&\cup \{f_{1,j}F_{2,j+1} v^{\lambda}\}_{1 < j < n} 							\cr
																				&\cup \{f_{2,j}F_{1,j+1}v^{\lambda}\}_{1 < j < n} 								\cr
																				&\cup \{f_{2,n}f_{1,n}v^{\lambda}\},
\label{case_l=k-1}
\end{align}
where $v^{\lambda}\in V(\lambda)_{\lambda}$ denotes a maximal vector of weight \(\lambda\)  in $V(\lambda)$ for $B.$ As usual, an application of Proposition \ref{Various_weight_multiplicities} gives $\dim V(\lambda)_{\mu}= 2n-1,$ so that the generating elements of \eqref{case_l=k-1} are linearly independent. The following assertion thus holds.

\begin{prop}\label{A_basis_for_W(lambda)_(lambda-alphak-1-2alphak-...-2alphan)}
Let $\lambda=a\lambda_1+\lambda_2\in X^+(T),$ where $a\in \Z_{>0},$ and set  $\mu=\lambda-(\alpha_1+2\alpha_2+\cdots+2\alpha_n).$ Then $\mu$ is dominant and the set \eqref{case_l=k-1} forms a basis of the weight space $V(\lambda)_{\mu}.$ 
\end{prop}

Suppose for the remainder of this section that $p \mid a+2,$ so that $\mu_{1,2}$ affords the highest weight of a composition factor of $V(\lambda)$ by Proposition \ref{Conditions_for_f(i,j)_to_belong_to_V(i,j)} (applied to a suitable Levi subgroup). Also denote by $u^+$ the corresponding maximal vector in $V(\lambda)_{\mu_{1,2}}$  given in Remark \ref{A_maximal_vector_in_V_(A_n)(a,0,....,0,b)}, and set $$\overline{V(\lambda)}= V(\lambda)/\langle G u^+\rangle_K.$$

\begin{lem}\label{[Dn<Bn] Preliminaries for L of type Bn: reduction to the study of V(lambda)/<Lu> for lambda = a lambda1 + lambda2}
Assume $p\mid a+2$ and adopt the notation introduced above. Then $[\langle G u^+\rangle_K,L(\mu)]=0.$ In particular $$[V(\lambda),L(\mu)]=[\overline{V(\lambda)},L(\mu)].$$
\end{lem}

\begin{proof}
By \cite[Lemma 2.13 b)]{Jantzen}, the \(KG\)-module $\langle G u^+\rangle_K$ is an image of \(V(\mu_{1,2}),\) in which $L(\mu)$ cannot occur as a composition factor by Proposition \ref{[Dn<Bn] Preliminaries for L of type Bn: irreducibility of V(lambda2) and basis of L(lambda2)_mu, where mu=lambda-12...2}. The result then follows.
\end{proof}

In view of Lemma \ref{[Dn<Bn] Preliminaries for L of type Bn: reduction to the study of V(lambda)/<Lu> for lambda = a lambda1 + lambda2}, we are led to investigate the structure of the quotient $\overline{V(\lambda)}.$ Write $\bar{v}^{\lambda}$ for the image of $v^{\lambda}$ in $\overline{V(\lambda)}.$ By Lemma \ref{structure_constants_for_B} and Remark \ref{A_maximal_vector_in_V_(A_n)(a,0,....,0,b)}, we  successively get 
\begin{equation}
f_{1,r}\bar{v}^{\lambda} = f_{3,r}f_{1,2}\bar{v}^{\lambda} = f_{3,r}f_{\alpha_1}f_{\alpha_2}\bar{v}^{\lambda}= f_{\alpha_1} f_{3,r} f_{\alpha_2}\bar{v}^{\lambda}= f_{\alpha_1}f_{2,r}\bar{v}^{\lambda} 
\label{[Dn<Bn] Preliminaries: f_(1,r) = f_(1) f_(2,r) for 2<r<n}
\end{equation}
for $2 < r \leqslant n.$ Also, since $\langle G u^+\rangle_K$ is an image of $V(\mu_{1,2})$ and $\m_{L(\mu_{1,2})}(\mu)=n-1$ by Proposition \ref{[Dn<Bn] Preliminaries for L of type Bn: irreducibility of V(lambda2) and basis of L(lambda2)_mu, where mu=lambda-12...2}, we have $\dim \overline{V(\lambda)}_\mu = n.$ Those observations can be used to determine a basis for the weight space $\overline{V(\lambda)}_{\mu},$ as the following result shows.

\newpage
\begin{prop}\label{Une premiere simplification de l'ensemble generateur de V(mu)}
Let $a\in \Z_{>0}$ be such that $p\mid a+2$ and set  $\lambda=a\lambda_1+\lambda_2.$ Also consider $\mu=\lambda-(\alpha_1+2\alpha_2+\cdots+2\alpha_n) $ and let $u^+$ be the maximal vector in $V(\lambda)_{\mu_{1,2}}$ for $B$ given in Remark \textnormal{\ref{A_maximal_vector_in_V_(A_n)(a,0,....,0,b)}}. Finally,  write $\bar{v}^{\lambda}$ for the image of $v^{\lambda}$ in $\overline{V(\lambda)}= V(\lambda)/\langle G u^+\rangle_K.$ Then \(\dim \overline{V(\lambda)}_\mu=n\) and a basis of the weight space $\overline{V(\lambda)}_{\mu}$ is given by
\begin{equation}
\{F_{1,2}\bar{v}^{\lambda}\}	\cup \{f_{2,j}F_{1,j+1}\bar{v}^{\lambda}\}_{1 < j < n} \cup \{f_{2,n}f_{1,n}\bar{v}^{\lambda}\}.
\label{elements_generating_W_{lambda-alpha_k-1-2alpha_k-...-2alpha_n}}
\end{equation} 
\end{prop}

\begin{proof}
We start by showing that each of  $f_{\alpha_1}f_{\alpha_2}F_{2,3}\bar{v}^{\lambda},$ $f_{1,j}F_{2,j+1}\bar{v}^{\lambda}$ $(1 < j < n)$ can be written as a linear combination of elements of \eqref{elements_generating_W_{lambda-alpha_k-1-2alpha_k-...-2alpha_n}}. Let $1< j<n $  be fixed. By Lemma \ref{structure_constants_for_B}, case \ref{[Dn<Bn] Preliminaries for L of type Bn: technical proposition concerning the weight lambda-2...2 for lambda=lambda_1 (part 1)} of Proposition \ref{[Dn<Bn] Preliminaries for L of type Bn: technical proposition concerning the weight lambda-2...2 for lambda=lambda_1}, and  \eqref{[Dn<Bn] Preliminaries: f_(1,r) = f_(1) f_(2,r) for 2<r<n}, we successively get
\begin{align*}
f_{1,j}F_{2,j+1}\bar{v}^{\lambda} &=  		 F_{1,2}\bar{v}^{\lambda} + F_{2,j+1}f_{\alpha_{1}}f_{2,j} \bar{v}^{\lambda} 																																											\cr
																	&= F_{1,2}\bar{v}^{\lambda} + F_{1,j+1}f_{2,j}\bar{v}^{\lambda} + f_{\alpha_1}f_{2,j}F_{2,j+1}\bar{v}^{\lambda}																											\cr
																	&= 2F_{1,2}\bar{v}^{\lambda} + f_{2,j}F_{1,j+1}\bar{v}^{\lambda} + f_{\alpha_1}f_{\alpha_2}F_{2,3}\bar{v}^{\lambda},													
\end{align*}
 so that $f_{1,j}F_{2,j+1}\bar{v}^{\lambda}  \in \left\langle F_{1,2}\bar{v}^{\lambda}, f_{\alpha_1}f_{\alpha_2}F_{2,3}\bar{v}^{\lambda}, f_{2,j}F_{1,j+1}\bar{v}^{\lambda}\right\rangle_K.$ It then remains to show that $f_{\alpha_1}f_{\alpha_2}F_{2,3}\bar{v}^{\lambda}$ is a linear combination of elements of \eqref{elements_generating_W_{lambda-alpha_k-1-2alpha_k-...-2alpha_n}}. By Lemma \ref{structure_constants_for_B}, case \ref{[Dn<Bn] Preliminaries for L of type Bn: technical proposition concerning the weight lambda-2...2 for lambda=lambda_1 (part 2)} of Proposition  \ref{[Dn<Bn] Preliminaries for L of type Bn: technical proposition concerning the weight lambda-2...2 for lambda=lambda_1}, and  \eqref{[Dn<Bn] Preliminaries: f_(1,r) = f_(1) f_(2,r) for 2<r<n}, we have
\begin{align*}
f_{\alpha_1}f_{\alpha_2}F_{2,3}\bar{v}^{\lambda}	&=			\tfrac{1}{2}f_{\alpha_1}(f_{2,n})^2\bar{v}^{\lambda}																						\cr
																									&=			\tfrac{1}{2}(f_{2,n}f_{\alpha_1}f_{2,n}\bar{v}^{\lambda}-f_{1,n}f_{2,n}\bar{v}^{\lambda})				\cr
																									&=			\tfrac{1}{2}(f_{2,n}f_{1,n}\bar{v}^{\lambda}-f_{1,n}f_{2,n}\bar{v}^{\lambda})										\cr
																									&=			-F_{1,2}\bar{v}^{\lambda},
\end{align*}
hence $f_{\alpha_1}f_{\alpha_2}F_{2,3}\bar{v}^{\lambda} \in \langle F_{1,2}\bar{v}^{\lambda} \rangle_K,$ showing that $\overline{V(\lambda)}_{\mu}$ is spanned by the vectors in \eqref{elements_generating_W_{lambda-alpha_k-1-2alpha_k-...-2alpha_n}}. Finally, the assertion on $\dim \overline{V(\lambda)}_{\mu}$ given above allows us to conclude.
\end{proof}

We now study the relation between the pair $(n,p)$ and the existence of a maximal vector of weight $\mu$ in  $\overline{V(\lambda)}$ for $B.$ For  $A=(A_r)_{1\leqslant r\leqslant n}\in K^n,$ set 
\begin{equation}
\bar{w}(A) = A_{1} F_{1,2}\bar{v}^{\lambda} +\sum_{j=2}^{n-1}{A_jf_{2,j}F_{1,j+1}\bar{v}^{\lambda}} +A_nf_{2,n}f_{1,n}\bar{v}^{\lambda}.
\label{[Dn<Bn] Definition of w(A) for A=(A1,...,An), mu=lambda-12...2 and G of type Bn}
\end{equation}

\begin{lem}\label{Conditions_pour_que_la_combinaison_lineaire_de_vecteurs_generateurs_de_V(k-1,k,n)_soit_tuee_par_tous_les_e_alpha}
Let $\lambda,$ $\mu$ be as above, with $p\mid a+2,$ and adopt the notation of \eqref{[Dn<Bn] Definition of w(A) for A=(A1,...,An), mu=lambda-12...2 and G of type Bn}.  Then the following assertions are equivalent. 
\begin{enumerate}
\item \label{V(k-1,n)_part_1} There exists $0 \neq A \in K^n $ such that $x_{\alpha}(c)\bar{w}(A)=0$ for every $\alpha \in \Pi$ and  \(c\in K.\)
\item \label{V(k-1,n)_part_2} There exist $A \in K^{n-1}\times K^*$  such that $x_{\alpha}(c)\bar{w}(A)=0$ for every $\alpha \in \Pi$ and  \(c\in K.\)
\item \label{V(k-1,n)_part_3} The divisibility condition $p \mid 2n-3$ is satisfied.
\end{enumerate}
\end{lem}

\begin{proof} 
Let $A=(A_r)_{1\leqslant r\leqslant n}\in K^n$ and set $\bar{w}=\bar{w}(A).$ Arguing as in the proofs of Lemmas \ref{[Dn<Bn] Prelimiaries for L of type An: existence of a maximal vector of weight lambda-1...1} 
 and \ref{Conditions_pour_que_la_combinaison_lineaire_de_vecteurs_generateurs_de_V(k,n)_soit_tuee_par_tous_les_e_alpha}, one sees that \(x_\alpha (c) \bar{w} =\bar{w}\) for every \(\alpha \in \Pi\) and  \(c\in K\) if and only if \(e_{\alpha}\bar{w}=0\) for every \(\alpha\in \Pi.\) Then using Lemma \ref{structure_constants_for_B} and Proposition \ref{[Dn<Bn] Preliminaries for L of type Bn: technical proposition concerning the weight lambda-2...2 for lambda=lambda_1}, we get 
\begin{align*}
e_{\alpha_1}\bar{w} &= -\sum_{i=2}^{n-1}{A_i f_{2,i}F_{2,i+1}\bar{v}^{\lambda}} - A_n(f_{2,n})^2\bar{v}^{\lambda} = -\left(\sum_{i=2}^{n-1}{A_i} + 2A_n\right)f_{\alpha_2}F_{2,3}\bar{v}^{\lambda}, \cr
\intertext{while Lemma \ref{structure_constants_for_B} yields}
e_{\alpha_2}\bar{w}	&= -A_1 F_{1,3}\bar{v}^{\lambda} + A_2 h_{\alpha_2}F_{1,3}\bar{v}^{\lambda} -\sum_{i=3}^{n-1}{A_if_{3,i}F_{1,i+1}\bar{v}^{\lambda}} - A_nf_{3,n}f_{1,n}\bar{v}^{\lambda}\cr 
										&= \left(-A_1 + 2A_2 + \sum_{r=3}^{n-1}{A_r} + 2A_n\right)F_{1,3}\bar{v}^{\lambda}. \cr
\intertext{Similarly, one easily checks that $e_{\alpha_r}\bar{w}	=  (A_r - A_{r-1}) f_{2,r-1}F_{1,r+1}\bar{v}^{\lambda},$ for every $2<r<n,$ and finally, we have}
e_{\alpha_n}\bar{w} &= 2A_n(f_{2,n-1}f_{1,n}\bar{v}^{\lambda} + f_{2,n}f_{1,n-1}\bar{v}^{\lambda}) - A_{n-1}f_{2,n-1}f_{1,n}\bar{v}^{\lambda}\cr
										&= ( 4A_n-A_{n-1} ) f_{2,n-1}f_{1,n}\bar{v}^{\lambda},
\end{align*}
where the last equality comes from the fact that $V(\lambda)_{\lambda-(2\alpha_2+\cdots+2\alpha_{n-1}+\alpha_n)}=0$ and \eqref{[Dn<Bn] Preliminaries: f_(1,r) = f_(1) f_(2,r) for 2<r<n} applied to $f_{2,n}f_{1,n-1}\bar{v}^{\lambda}.$ One checks that each of the vectors $f_{\alpha _2}F_{2,3}\bar{v}^{\lambda},$ $F_{1,3}\bar{v}^{\lambda},$ $f_{\alpha_2}F_{1,4}\bar{v}^{\lambda},\ldots,f_{2,n-2}F_{1,n}\bar{v}^{\lambda},$ and $f_{2,n-1}f_{1,n}\bar{v}^{\lambda}$ is non-zero. Consequently, $e_{\alpha}\bar{w}(A)=0$ for every $\alpha \in \Pi$ if and only if $A\in K^n$ is a solution to the system of equations 
\begin{equation}
\left\{\begin{array}{llll}
2A_n			&= -\sum_{j=2}^{n-1}{A_j} 										\\
A_1   		&= 2(A_2+A_n) +\sum_{j=3}^{n-1}{A_j} 					\\
A_{r-1}		&= A_r \mbox{ for $2< r \leqslant n-1$}			\\
A_{n-1} 	&= 4 A_n.
\end{array}
\right.
\label{third_system_of_equations}
\end{equation}

Now one easily sees that \eqref{third_system_of_equations} admits a non-trivial solution $A$ if and only if $p\mid 2n-3$ (showing that \ref{V(k-1,n)_part_1} and \ref{V(k-1,n)_part_3} are equivalent), in which case $A\in \langle(4,\ldots,4,1)\rangle_K$ (so that \ref{V(k-1,n)_part_1} and \ref{V(k-1,n)_part_2} are equivalent), completing the proof.
\end{proof}

Let $\lambda$ and $\mu$ be as above, with $p \mid a + 2,$  and consider an irreducible $KG$-module $V=L(\lambda)$ having highest weight $\lambda.$ As usual, take $V=V(\lambda)/\rad(\lambda),$ so that 
\[
V \cong \bigquotient{\overline{V(\lambda)}}{\overline{\rad(\lambda)}},
\]
where $\overline{\rad(\lambda)}=\rad(\lambda)/\langle G u^+\rangle_K.$ Also  write $v^+$ to denote the image of $\bar{v}^{\lambda}$ in $V,$ that is, $v^+$ is a maximal vector of weight \(\lambda\) in $V$ for $B.$ By Proposition \ref{Une premiere simplification de l'ensemble generateur de V(mu)}, the weight space $V_\mu$ is spanned by the vectors
\begin{equation} 
\left\{F_{1,2}v^+\right\}  \cup \left\{f_{2,j}F_{1,j+1}v^+\right\}_{1 < j < n} \cup \left\{f_{2,n}f_{1,n}v^+\right\}.
 \label{elements_generating_V_{lambda-alpha_k-1-2alpha_k-...-2alpha_n}}
\end{equation}

We write $V_{1,2,n}$ to denote the span of all the generators in \eqref{elements_generating_V_{lambda-alpha_k-1-2alpha_k-...-2alpha_n}} except for $f_{2,n}f_{1,n}v^+$. As usual, the following result gives a precise description of the weight space $V_{\mu},$ as well as a characterization for \(\mu\) to afford the highest weight of a composition factor of \(V(\lambda).\)

\begin{prop}\label{condition_if_l=k-1}
Let $G$ be a simple algebraic group of type $B_n$ over $K$ and consider an irreducible $KG$-module  $V=L(\lambda)$  having $p$-restricted highest weight $\lambda=a\lambda_1+\lambda_2,$ where  $a\in \Z_{>0}$ is such that  $p\mid a+2.$ Also consider $\mu=\lambda-(\alpha_1 + 2\alpha_2+\cdots+2\alpha_n)\in \Lambda^+(\lambda).$ Then the following assertions are equivalent.
\begin{enumerate}
\item \label{conditions_for_f(i,i)_to_belong_to_V(k-1,k,n)_part_1} The weight $\mu$ affords the highest weight of a composition factor of $V(\lambda).$
\item \label{conditions_for_f(i,i)_to_belong_to_V(k-1,k,n)_part_2} The generators in \eqref{elements_generating_V_{lambda-alpha_k-1-2alpha_k-...-2alpha_n}} are linearly dependent.
\item \label{conditions_for_f(i,i)_to_belong_to_V(k-1,k,n)_part_3} The element $f_{2,n}f_{1,n}v^+$ lies inside $V_{1,2,n}.$
\item \label{conditions_for_f(i,i)_to_belong_to_V(k-1,k,n)_part_4} The divisibility condition $p \mid 2n-3$ is satisfied.
\end{enumerate}
\end{prop} 

\begin{proof}
Clearly \ref{conditions_for_f(i,i)_to_belong_to_V(k-1,k,n)_part_3} implies \ref{conditions_for_f(i,i)_to_belong_to_V(k-1,k,n)_part_2}, while if \ref{conditions_for_f(i,i)_to_belong_to_V(k-1,k,n)_part_1} holds, then Lemma \ref{[Dn<Bn] Preliminaries for L of type Bn: reduction to the study of V(lambda)/<Lu> for lambda = a lambda1 + lambda2} yields $[\overline{V(\lambda)},L(\mu)]\neq 0,$ so that \ref{conditions_for_f(i,i)_to_belong_to_V(k-1,k,n)_part_2} holds (by comparing \eqref{elements_generating_W_{lambda-alpha_k-1-2alpha_k-...-2alpha_n}} with \eqref{elements_generating_V_{lambda-alpha_k-1-2alpha_k-...-2alpha_n}}). Now if \ref{conditions_for_f(i,i)_to_belong_to_V(k-1,k,n)_part_2} is satisfied, then $L(\nu)$ occurs as a composition factor of $\overline{V(\lambda)}$ for some $\nu\in \Lambda^+(\lambda)$ such that $\mu \preccurlyeq \nu\prec \lambda$ by Proposition \ref{Une premiere simplification de l'ensemble generateur de V(mu)}. Since only $\mu$ can afford the highest weight of such a composition factor, \ref{conditions_for_f(i,i)_to_belong_to_V(k-1,k,n)_part_1} holds by Lemma \ref{[Dn<Bn] Preliminaries for L of type Bn: reduction to the study of V(lambda)/<Lu> for lambda = a lambda1 + lambda2} again. This also shows the existence of $0\neq A \in K^n$ such that $\bar{w}(A)$ is a maximal vector of weight \(\mu\)  in $\overline{V(\lambda)}$ for $B,$ where we adopt the notation of \eqref{[Dn<Bn] Definition of w(A) for A=(A1,...,An), mu=lambda-12...2 and G of type Bn}. Therefore \ref{conditions_for_f(i,i)_to_belong_to_V(k-1,k,n)_part_2} implies \ref{conditions_for_f(i,i)_to_belong_to_V(k-1,k,n)_part_4} by Lemma \ref{Conditions_pour_que_la_combinaison_lineaire_de_vecteurs_generateurs_de_V(k-1,k,n)_soit_tuee_par_tous_les_e_alpha}. Finally suppose that \ref{conditions_for_f(i,i)_to_belong_to_V(k-1,k,n)_part_4} holds. By Lemma \ref{Conditions_pour_que_la_combinaison_lineaire_de_vecteurs_generateurs_de_V(k-1,k,n)_soit_tuee_par_tous_les_e_alpha}, there exists $A\in K^{n-1} \times K^*$ such that $x_{\alpha}(c)\bar{w}(A)=0$ for every $\alpha\in \Pi$ and  \(c\in K.\) Consequently, we also get $x_{\alpha}(c)(\bar{w}(A)+\rad(\lambda))=0$ for every $\alpha \in \Pi$ and  \(c\in K,\) that is, $\bar{w}(A) + \rad(\lambda) \in \langle v^+ \rangle_K \cap V(\lambda)_{\mu}=0.$ Therefore \ref{conditions_for_f(i,i)_to_belong_to_V(k-1,k,n)_part_3} holds and the proof is complete.
\end{proof}

\subsection{Study of \texorpdfstring{$L(a\lambda_1+\lambda_k)$}{LG}  \texorpdfstring{$(2<k<n,\mbox{ } a\in \mathbb{Z}_{\geqslant 0})$}{LG}}     

Let \(a\in \Z_{>0},\) \(2<k<n,\) and set $\lambda=a\lambda_1+\lambda_k.$ Also write $\mu_{1,k}=\lambda-(\alpha_1+\cdots+\alpha_k)$ and consider  $$\mu=\lambda-(\alpha_1+2\alpha_2+\cdots+2\alpha_n).$$  By Proposition \ref{A simplification for V(lambda)mu}, our choice of ordering $\leqslant$ on $\Phi^+$ and Proposition \ref{[Dn<Bn] Preliminaries for L of type Bn: technical proposition concerning the weight lambda-2...2 for lambda=lambda_1} (applied to a suitable Levi subgroup), one sees that the weight space $V(\lambda)_{\mu}$ is spanned by  the vectors
\begin{align}
\{F_{1,k}v^{\lambda}\} 	&\cup \{f_{1,k-1}f_{\alpha_k}F_{k,k+1}v^{\lambda}\} 																\cr
																		&\cup \{f_{1,i}F_{i+1,k}v^{\lambda}\}_{1\leqslant i\leqslant k-2} 														\cr
																		&\cup \{f_{1,j}F_{k,j+1}v^{\lambda}\}_{k\leqslant j<n} 																	\cr
																		&\cup \{f_{1,i}f_{i+1,k-1}f_{\alpha_k}F_{k,k+1}v^{\lambda}\}_{1\leqslant i\leqslant k-2} 		\cr
																		&\cup \{f_{1,i}f_{k,j}F_{i+1,j+1}v^{\lambda}\}_{1\leqslant i\leqslant k-2,k\leqslant j <n} 				\cr
																		&\cup \{f_{k,j}F_{1,j+1}v^{\lambda}\}_{k\leqslant j<n} 																	\cr
																		&\cup \{f_{k,n}f_{1,n}v^{\lambda}\},
\label{generators_for_W(lambda)_nu}
\end{align}
where $v^{\lambda}$ is a maximal vector of weight $\lambda$ in $V(\lambda)$ for $B.$ As usual, an application of Proposition \ref{Various_weight_multiplicities} yields $\m_{V(\lambda)}(\mu)=k(n-k+2)-1,$ forcing the generating elements of \eqref{generators_for_W(lambda)_nu} to be linearly independent. The following result thus holds.

\begin{prop}\label{Elements_generating_W(lambda)_nu}
Let $a\in \Z_{>0}$ and $2<k<n.$ Also set $\lambda=a\lambda_1+\lambda_k\in X^+(T)$ and consider the dominant $T$-weight  $\mu=\lambda-(\alpha_1+\cdots+\alpha_{k-1}+2\alpha_k+\cdots+2\alpha_n)\in \Lambda^+(\lambda).$ Then  $\m_{V(\lambda)}(\mu)=k(n-k+2)-1$ and the set  given in \eqref{generators_for_W(lambda)_nu} forms a basis of $V(\lambda)_{\mu}.$
\end{prop}

Suppose for the remainder of this section that $p \mid a+k,$ so that $\mu_{1,k}$ affords the highest weight of a composition factor of $V(\lambda)$ by Proposition \ref{Conditions_for_f(i,j)_to_belong_to_V(i,j)} (applied to the Levi subgroup corresponding to the simple roots \(\alpha_1,\ldots,\alpha_k\)). Also denote by $u^+$ the corresponding maximal vector in $V(\lambda)_{\mu_{1,k}}$ for  $B$ given in Remark \ref{A_maximal_vector_in_V_(A_n)(a,0,....,0,b)}, and set 
$$\overline{V(\lambda)}= V(\lambda)/\langle G u^+\rangle_K.$$ We omit the details of the proof of the following result, as it is identical to that of Lemma \ref{[Dn<Bn] Preliminaries for L of type Bn: reduction to the study of V(lambda)/<Lu> for lambda = a lambda1 + lambda2}.

\begin{lem}\label{[Dn<Bn] Preliminaries for L of type Bn: reduction to the study of V(lambda)/<Lu> for lambda = a lambda_1 + lambda_k}
Assume $p\mid a+k$ and adopt the notation introduced above. Then $[\langle G u^+\rangle_K,L(\mu)]=0 .$ In particular $[V(\lambda),L(\mu)]=[\overline{V(\lambda)},L(\mu)].$
\end{lem}

In view of Lemma \ref{[Dn<Bn] Preliminaries for L of type Bn: reduction to the study of V(lambda)/<Lu> for lambda = a lambda_1 + lambda_k}, it is  natural to investigate the structure of the quotient $\overline{V(\lambda)}.$ Write $\bar{v}^{\lambda}$ for the class of $v^{\lambda}$ in $\overline{V(\lambda)}.$ By Lemma \ref{structure_constants_for_B} and Remark \ref{A_maximal_vector_in_V_(A_n)(a,0,....,0,b)}, we successively get
\begin{equation}
f_{1,r}\bar{v}^{\lambda}	=  f_{k+1,r}f_{1,k}\bar{v}^{\lambda} =  \sum_{s=1}^{k-1}{f_{k+1,r}f_{1,s}f_{s+1,k}\bar{v}^{\lambda}}  =  \sum_{s=1}^{k-1}{f_{1,s}f_{s+1,r}\bar{v}^{\lambda}}
\label{[Dn<Bn] Preliminaries: f_(1,r) = sum f_(1,s) f_(s+1,r) for 2<r<n}
\end{equation}
for $k<r \leqslant n.$ Also, since $\langle G u^+\rangle_K$ is an image of $V(\mu_{1,k})$ and $\m_{L(\mu_{1,k})}(\mu)=n-k+1$ by Proposition \ref{basis_of_W_lambda-alpha_i-2alpha_k-...-2alpha_n} (applied to a suitable Levi subgroup), we have $\dim \overline{V(\lambda)}_\mu = (k-1)(n-k+2).$ Those observations can be used to determine a basis of the weight space $\overline{V(\lambda)}_{\mu},$ as the following result shows.

%
%
\begin{prop}\label{Une premiere simplification de l'ensemble generateur de V(nu)}
Let $a\in \Z_{>0}$ and $2<k<n$ be such that $p\mid a+k,$ and consider the dominant character $\lambda=a\lambda_1+\lambda_k.$ Also set $\mu=\lambda-(\alpha_1+\cdots+ \alpha_{k-1} +2\alpha_k+\cdots+2\alpha_n)\in \Lambda^+(\lambda)$ and let $u^+$ be the maximal vector of weight $\mu_{1,k}$ in $V(\lambda)$ for $B$ given in Remark \textnormal{\ref{A_maximal_vector_in_V_(A_n)(a,0,....,0,b)}}. Finally, set $\overline{V(\lambda)}= V(\lambda)/\langle G u^+\rangle_K,$ and write $\bar{v}^{\lambda}$ for the class of $v^{\lambda}$ in $\overline{V(\lambda)}.$ Then $\dim \overline{V(\lambda)}_\mu = (k-1)(n-k+2)$ and a basis of the weight space $\overline{V(\lambda)}_{\mu}$ is given by
\begin{align}
\{F_{1,k}\bar{v}^{\lambda}\} 		&\cup  \{f_{1,i}F_{i+1,k}\bar{v}^{\lambda}\}_{1\leqslant i\leqslant k-2} 														\cr
																						&\cup  \{f_{1,i}f_{i+1,k-1}f_{\alpha_k}F_{k,k+1}\bar{v}^{\lambda}\}_{1\leqslant i\leqslant k-2} 			\cr
																						&\cup  \{f_{1,i}f_{k,j}F_{i+1,j+1}\bar{v}^{\lambda}\}_{1\leqslant i\leqslant k-2,k\leqslant j <n} 				\cr
																						&\cup  \{f_{k,j}F_{1,j+1}\bar{v}^{\lambda}\}_{k\leqslant j<n} 																	\cr
																						&\cup  \{f_{k,n}f_{1,n}\bar{v}^{\lambda}\}.
\label{generators_for_V(lambda)_nu}
\end{align}
\end{prop}

\begin{proof}
We first show that $f_{1,k-1}f_{\alpha_k}F_{k,k+1}\bar{v}^{\lambda}$ lies inside the subspace of $\overline{V(\lambda)}$ generated by the elements $F_{1,k}\bar{v}^{\lambda},$ $f_{1,i}F_{i+1,k}\bar{v}^{\lambda},$ $f_{1,i}f_{i+1,k-1}f_{\alpha_k}F_{k,k+1}\bar{v}^{\lambda},$ and $f_{1,i}f_{k,j}F_{i+1,j+1}\bar{v}^{\lambda},$ where $1\leqslant i\leqslant k-2$ and $k\leqslant j<n.$ By Lemma \ref{structure_constants_for_B} and case  \ref{[Dn<Bn] Preliminaries for L of type Bn: technical proposition concerning the weight lambda-2...2 for lambda=lambda_1 (part 2)} of Proposition \ref{[Dn<Bn] Preliminaries for L of type Bn: technical proposition concerning the weight lambda-2...2 for lambda=lambda_1}, we have
\begin{align*}
f_{1,k-1}f_{\alpha_k}F_{k,k+1}\bar{v}^{\lambda}			&= 		\tfrac{1}{2}f_{1,k-1}(f_{k,n})^2\bar{v}^{\lambda} 																																\cr
																										&= 		\tfrac{1}{2}(f_{k,n}f_{1,k-1}f_{k,n}\bar{v}^{\lambda}-f_{1,n}f_{k,n}\bar{v}^{\lambda})														\cr
																										&=		\tfrac{1}{2}(f_{k,n}f_{1,k-1}f_{k,n}\bar{v}^{\lambda}-2F_{1,k}\bar{v}^{\lambda}- f_{k,n}f_{1,n}\bar{v}^{\lambda}),
\end{align*}
and by \eqref{[Dn<Bn] Preliminaries: f_(1,r) = sum f_(1,s) f_(s+1,r) for 2<r<n}, we get $f_{k,n}f_{1,k-1}f_{k,n}\bar{v}^{\lambda} = f_{k,n}f_{1,n}\bar{v}^{\lambda}-\sum_{r=1}^{k-2}{f_{1,r}f_{k,n}f_{r+1,n}\bar{v}^{\lambda}}.$ For each \(1\leqslant r\leqslant k-2,\) an application of Proposition \ref{basis_of_W_lambda-alpha_i-2alpha_k-...-2alpha_n} (to the Levi subgroup corresponding to the simple roots \(\alpha_{r+1},\ldots,\alpha_n\)) then shows that \(f_{1,r}f_{k,n}f_{r+1,n}\bar{v}^\lambda\) can be written as a linear combination of elements of \eqref{generators_for_V(lambda)_nu} as desired. Next let $k\leqslant j <n,$ and first observe that by Lemma \ref{structure_constants_for_B} and \eqref{[Dn<Bn] Preliminaries: f_(1,r) = sum f_(1,s) f_(s+1,r) for 2<r<n}, we have 
\begin{align*}
f_{1,j}F_{k,j+1}\bar{v}^{\lambda}			 		&= F_{1,k}\bar{v}^{\lambda} + F_{k,j+1}f_{1,j}\bar{v}^{\lambda} 																																						\cr
																		&= F_{1,k}\bar{v}^{\lambda} + \sum_{r=1}^{k-2}{f_{1,r}F_{k,j+1}f_{r+1,j}\bar{v}^{\lambda}} + F_{k,j+1}f_{1,k-1}f_{k,j}\bar{v}^{\lambda}.
\end{align*}
For each \(1\leqslant r\leqslant k-2,\) applying Proposition \ref{basis_of_W_lambda-alpha_i-2alpha_k-...-2alpha_n} (to the Levi subgroup corresponding to the simple roots \(\alpha_{r+1},\ldots,\alpha_n\)) shows that $f_{1,r}F_{k,j+1}f_{r+1,j}\bar{v}^{\lambda}$ lies inside the subspace of $\overline{V(\lambda)}$ generated by the elements of \eqref{generators_for_V(lambda)_nu} as desired, while  an application of  case \ref{[Dn<Bn] Preliminaries for L of type Bn: technical proposition concerning the weight lambda-2...2 for lambda=lambda_1 (part 1)} of Proposition \ref{[Dn<Bn] Preliminaries for L of type Bn: technical proposition concerning the weight lambda-2...2 for lambda=lambda_1} yields
\begin{align*}
F_{k,j+1}f_{1,k-1}f_{k,j}\bar{v}^{\lambda} &= F_{1,j+1}f_{k,j}\bar{v}^{\lambda} + f_{1,k-1}F_{k,j+1}f_{k,j}\bar{v}^{\lambda} \cr
														 &= F_{1,k}\bar{v}^{\lambda} + f_{k,j}F_{1,j+1}\bar{v}^{\lambda} + f_{1,k-1}f_{\alpha_k}F_{k,k+1}\bar{v}^{\lambda}.
\end{align*}

Therefore the weight space $\overline{V(\lambda)}_{\mu}$ is spanned by the elements of \eqref{generators_for_V(lambda)_nu} and the assertion on $\dim \overline{V(\lambda)}_{\mu}$ given above allows us to conclude.
\end{proof}

In order to investigate the existence of a maximal vector of weight $\mu$ in $\overline{V(\lambda)}$ for $B$ as in Lemma \ref{Conditions_pour_que_la_combinaison_lineaire_de_vecteurs_generateurs_de_V(k-1,k,n)_soit_tuee_par_tous_les_e_alpha}, we require the following technical result.

\begin{lem}\label{technical_lemma}
Let $\lambda,$ and $\mu$ be as above, with  $p \mid a+k.$ Then the following assertions hold.
\begin{enumerate} 
\item	\label{technical_lemma_part_1} $f_{k,n}f_{2,n}\bar{v}^{\lambda} = -F_{2,k}\bar{v}^{\lambda} - f_{2,k-1}f_{\alpha_k}F_{k,k+1}\bar{v}^{\lambda}.$
\item \label{technical_lemma_part_3} $f_{k,n-1}f_{1,n}\bar{v}^{\lambda} = -\sum_{r=1}^{k-2}{f_{1,r}f_{r+1,n-1}f_{k,n}\bar{v}^{\lambda}} + f_{1,n-1}f_{k,n}\bar{v}^{\lambda}.$
\item \label{technical_lemma_part_4} $f_{k,n-1}f_{r+1,n}\bar{v}^{\lambda} = -f_{r+1,n-1}f_{k,n}\bar{v}^{\lambda}$ for  $1\leqslant r<k-1.$
\end{enumerate}
\end{lem}

\begin{proof}
By Lemma \ref{structure_constants_for_B} and case \ref{[Dn<Bn] Preliminaries for L of type Bn: technical proposition concerning the weight lambda-2...2 for lambda=lambda_1 (part 2)} of Proposition \ref{[Dn<Bn] Preliminaries for L of type Bn: technical proposition concerning the weight lambda-2...2 for lambda=lambda_1} (together with the fact that \(f_{2,k-1}\bar{v}^\lambda=0\)), we get 
\begin{align*}
f_{k,n}f_{2,n}\bar{v}^{\lambda}			&= 	 -f_{2,n}f_{k,n}\bar{v}^{\lambda} - f_{2,k-1}(f_{k,n})^2\bar{v}^{\lambda}\cr
																		&= -2F_{2,k}\bar{v}^{\lambda} - f_{k,n}f_{2,n}\bar{v}^{\lambda} - f_{2,k-1}(f_{k,n})^2\bar{v}^{\lambda} \cr
																		&= -2F_{2,k}\bar{v}^{\lambda} - f_{k,n}f_{2,n}\bar{v}^{\lambda} - 2f_{2,k-1}f_{\alpha_k}F_{k,k+1}\bar{v}^{\lambda},
\end{align*}
from which \ref{technical_lemma_part_1} immediately follows. Also, by Lemma \ref{structure_constants_for_B} and \eqref{[Dn<Bn] Preliminaries: f_(1,r) = sum f_(1,s) f_(s+1,r) for 2<r<n}, we have
\begin{align*}
f_{k,n-1}f_{1,n} \bar{v}^{\lambda}		&= \sum_{r=1}^{k-2}{f_{1,r}f_{k,n-1}f_{r+1,n}\bar{v}^{\lambda}} + f_{k,n-1}f_{1,k-1}f_{k,n}\bar{v}^{\lambda} 				\cr
																		&= \sum_{r=1}^{k-2}{f_{1,r}f_{k,n-1}f_{r+1,n}\bar{v}^{\lambda}} + f_{1,n-1}f_{k,n}\bar{v}^{\lambda}.
\end{align*}
Noticing that $f_{k,n-1}f_{r+1,n}\bar{v}^{\lambda}=-f_{k,n-1}f_{r+1,k-1}f_{k,n}\bar{v}^{\lambda}=-f_{r+1,n-1}f_{k,n}\bar{v}^{\lambda}$ for $1\leqslant r\leqslant k-1$ then yields \ref{technical_lemma_part_3} and \ref{technical_lemma_part_4}, thus completing the proof.
\end{proof}

We now study the relation between the triple $(n,k,p)$ and the existence of a maximal vector in  $\overline{V(\lambda)}_{\mu}$ for $B,$ assuming $p \mid a + k.$  For $S=(A,B_i,C_i,D_{ij},E_j,F)\in K^{(k-1)(n-k+2)}$ $(1\leqslant i\leqslant k-2$ and $ k\leqslant j\leqslant n-1),$ set  
\begin{equation}
\begin{aligned}
\bar{w}(S) = AF_{1,k}\bar{v}^{\lambda} 	&+\sum_{i=1}^{k-2}{B_i f_{1,i}F_{i+1,k}\bar{v}^{\lambda}} + \sum_{i=1}^{k-2}{C_i f_{1,i}f_{i+1,k-1}f_{\alpha_k}F_{k,k+1}\bar{v}^{\lambda}} \cr
																				&+ \sum_{i=1}^{k-2}{\sum_{j=k}^{n-1}{D_{ij}f_{1,i}f_{k,j}F_{i+1,j+1}\bar{v}^{\lambda}}}+ \sum_{j=k}^{n-1}{E_j f_{k,j}F_{1,j+1}\bar{v}^{\lambda}} \cr
																				&+ Ff_{k,n}f_{1,n} \bar{v}^{\lambda}.
\end{aligned}
\label{[Dn<Bn] Definition of w(X) for mu=lambda-1...12...2 and G of type Bn}
\end{equation}

\begin{lem}\label{Conditions_pour_que_la_combinaison_lineaire_de_vecteurs_generateurs_de_V(l,k,n)_soit_tuee_par_tous_les_e_alpha}
Let $\lambda,$ $\mu$ be as above, with $p \mid a+k,$ and adopt the notation of \eqref{[Dn<Bn] Definition of w(X) for mu=lambda-1...12...2 and G of type Bn}. Then the following assertions are equivalent. 
\begin{enumerate}
\item \label{V(l,k,n)_part_1} There exists $0\neq S=(A,B_i,C_i,D_{ij},E_j,F)\in K^{(k-1)(n-k+2)}$ $(1\leqslant i\leqslant k-2$ and $k\leqslant j\leqslant n-1)$ such that $x_{\alpha}(c)\bar{w}(S)=0$ for every $\alpha \in \Pi$ and  \(c\in K.\)
\item \label{V(l,k,n)_part_2} There exists $S=(A,B_i,C_i,D_{ij},E_j,F)\in K^{(k-1)(n-k+2)-1}\times K^{*}$ $(1\leqslant i\leqslant k-2$ and $k\leqslant j\leqslant n-1)$ such that $x_{\alpha}(c)\bar{w}(S)=0$ for every $\alpha \in \Pi$ and  \(c\in K.\)
\item \label{V(l,k,n)_part_3} The divisibility condition $ p\mid 2(n-k) + 1$ is satisfied.
\end{enumerate}
\end{lem}

\begin{proof} 
We start by assuming $k=3,$ so write $B,C$ and $D_{j} $ for $B_{1},C_{1}$ and $D_{1,j}$ $(3\leqslant j\leqslant n-1),$ and let $S=(A,B,C,D_j,E_j,F)\in K^{2(n-1)}.$ With these simplifications, $\bar{w}=\bar{w}(S)$ can be rewritten as
\begin{align*}
\bar{w}=AF_{1,3}\bar{v}^{\lambda} &+Bf_{\alpha_1}F_{2,3}\bar{v}^{\lambda} + C f_{\alpha_1}f_{\alpha_2}f_{\alpha_3}F_{3,4}\bar{v}^{\lambda} \cr
							&+ \sum_{j=3}^{n-1}{D_{j}f_{\alpha_1}f_{3,j}F_{2,j+1}\bar{v}^{\lambda}}+ \sum_{j=3}^{n-1}{E_j f_{3,j}F_{1,j+1}\bar{v}^{\lambda}} \cr
							&+ Ff_{3,n}f_{1,n} \bar{v}^{\lambda}.
\end{align*}
Arguing as in the proof of Lemmas \ref{[Dn<Bn] Prelimiaries for L of type An: existence of a maximal vector of weight lambda-1...1} and \ref{Conditions_pour_que_la_combinaison_lineaire_de_vecteurs_generateurs_de_V(k,n)_soit_tuee_par_tous_les_e_alpha}, one sees that \(x_\alpha (c) \bar{w}=\bar{w}\) for every \(\alpha \in \Pi\) and  \(c\in K\) if and only if \(e_{\alpha}\bar{w}=0\) for every \(\alpha\in \Pi.\) Lemma \ref{structure_constants_for_B} then yields

\begin{align*}
e_{\alpha_1}\bar{w}		&=		(-A + (a+1)B)F_{2,3}\bar{v}^{\lambda} +(a+1)Cf_{\alpha_2}f_{\alpha_3}F_{3,4}\bar{v}^{\lambda}	\cr
						&\;\;\;\;\,\,+ \sum_{j=3}^{n-1}{((a+1)D_j-E_j)f_{3,j}F_{2,j+1}\bar{v}^{\lambda}} - Ff_{3,n}f_{2,n}\bar{v}^{\lambda}\cr
						&= (-A + (a+1)B+F)F_{2,3}\bar{v}^{\lambda} +((a+1)C+F) f_{\alpha_2}f_{\alpha_3}F_{3,4}\bar{v}^{\lambda}\cr
						&\;\;\;\;\,\,+ \sum_{j=3}^{n-1}{((a+1)D_j-E_j)f_{3,j}F_{2,j+1}\bar{v}^{\lambda}},
\end{align*}
where the last equality follows from case \ref{technical_lemma_part_1} of Lemma \ref{technical_lemma}. Similarly, we have
\begin{align*}
e_{\alpha_2}\bar{w} 			&= (2C-D_3)f_{\alpha_1}f_{\alpha_3}F_{3,4}\bar{v}^{\lambda} - \sum_{j=4}^{n-1}{D_j f_{\alpha_1}f_{3,j}F_{3,j+1}\bar{v}^{\lambda}}\cr
							&= \bigg(2C-\sum_{j=3}^{n-1}{D_j}\bigg)f_{\alpha_1}f_{\alpha_3}F_{3,4}\bar{v}^{\lambda},\cr
\intertext{where the last equality follows from Proposition \ref{[Dn<Bn] Preliminaries for L of type Bn: technical proposition concerning the weight lambda-2...2 for lambda=lambda_1} (case \ref{[Dn<Bn] Preliminaries for L of type Bn: technical proposition concerning the weight lambda-2...2 for lambda=lambda_1 (part 1)}). Again, applying Lemma \ref{structure_constants_for_B} successively gives}
e_{\alpha_3}\bar{w} 			&= (2E_3-A)F_{1,4}\bar{v}^{\lambda} - \sum_{j=4}^{n-1}{D_jf_{\alpha_1}f_{4,j}F_{2,j+1}\bar{v}^{\lambda}}\cr 
							&\;\;\;\;\,\,+ (2D_3-B-C)f_{\alpha_1}F_{2,4}\bar{v}^{\lambda} - \sum_{j=4}^{n-1}{E_jf_{4,j}F_{1,j+1}\bar{v}^{\lambda}}- F f_{4,n}f_{1,n}\bar{v}^{\lambda}\cr
							&= \bigg(2E_3 +\sum_{j=4}^{n-1}{E_j}-A+2F \bigg) F_{1,4}\bar{v}^{\lambda} + \bigg(2D_3+ \sum_{j=4}^{n-1}{D_j}-B-C \bigg) f_{\alpha_1}F_{2,4}\bar{v}^{\lambda},																																												\cr
\intertext{while for every $4\leqslant r\leqslant n-1$ one shows that }
e_{\alpha_r}\bar{w}				&= (D_r-D_{r-1})f_{\alpha_1}f_{3,r-1}F_{2,r+1}\bar{v}^{\lambda} + (E_r-E_{r-1})f_{3,r-1}F_{1,r+1}\bar{v}^{\lambda}.
\end{align*}
Finally, we leave to the reader to check (using case \ref{technical_lemma_part_3} of Lemma \ref{technical_lemma}  and \eqref{[Dn<Bn] Preliminaries: f_(1,r) = sum f_(1,s) f_(s+1,r) for 2<r<n}) that 
\begin{align*}
f_{\alpha_1}f_{3,n-1}f_{2,n}\bar{v}^{\lambda}			&=			f_{3,n-1}f_{1,n}\bar{v}^{\lambda}-f_{3,n-1}f_{1,2}f_{3,n}\bar{v}^{\lambda}			\cr
																									&=			f_{3,n-1}f_{1,n}\bar{v}^{\lambda}-f_{1,n-1} f_{3,n}\bar{v}^{\lambda}						\cr
																									&=			- f_{\alpha_1}f_{2,n-1}f_{3,n}\bar{v}^{\lambda},
\end{align*}
and hence 
\begin{align*}
e_{\alpha_n}\bar{w} 			&= -D_{n-1}f_{\alpha_1}f_{3,n-1}f_{2,n}\bar{v}^{\lambda}+ (2F-E_{n-1})f_{3,n-1}f_{1,n}\bar{v}^{\lambda}+ 2Ff_{1,n-1}f_{3,n}\bar{v}^{\lambda} 																\cr
													&= (2F-D_{n-1}-E_{n-1})f_{\alpha_1}f_{2,n-1}f_{3,n}\bar{v}^{\lambda} + (4F-E_{n-1})f_{1,n-1}f_{3,n}\bar{v}^{\lambda}.
\end{align*}

As usual, one then checks that the vector $f_{\alpha_1}f_{\alpha_3}F_{3,4}\bar{v}^{\lambda}$ is non-zero. Also by Lemma \ref{basis_of_W_lambda-alpha_i-2alpha_k-...-2alpha_n}, the set $\{F_{2,3}\bar{v}^{\lambda},f_{\alpha_2}f_{\alpha_3}F_{3,4}\bar{v}^{\lambda},f_{3,j}F_{2,j+1}\bar{v}^{\lambda} : 3\leqslant j< n\}$ is linearly independent. Similarly, one  sees that each of the sets $\{F_{1,4}\bar{v}^{\lambda} , f_{\alpha_1}F_{2,4}\bar{v}^{\lambda}\},$ $\{f_{\alpha_1}f_{3,r-1}F_{2,r+1}\bar{v}^{\lambda} , f_{3,r-1}F_{1,r+1}\bar{v}^{\lambda}\}$ (for every $4\leqslant r\leqslant n-1$),  $\{f_{\alpha_1}f_{2,n-1}f_{3,n}\bar{v}^{\lambda}, f_{1,n-1}f_{3,n}\bar{v}^{\lambda}\}$ is linearly independent as well.  One then concludes that $e_{\alpha}\bar{w}(S)=0$ for every $\alpha \in \Pi$ if and only if $p\mid 2n-5$ (showing that \ref{V(l,k,n)_part_1} and \ref{V(l,k,n)_part_3} are equivalent), in which case  
\[
S \in \langle (4,1-n,3-n,\underbrace{-2,\ldots,-2}_{n-3},\underbrace{4,\ldots,4}_{n-3},1)\rangle_K
\] 
(so that \ref{V(l,k,n)_part_1} and \ref{V(l,k,n)_part_2} are equivalent). The result  follows in this situation.

Now assume $3<k<n$  and let $S=(A,B_i,C_i,D_{ij},E_j,F)\in K^{(k-1)(n-k+2)},$ where $1\leqslant i\leqslant k-2$ and $k\leqslant j\leqslant n-1.$ Again, notice that \(x_\alpha (c) \bar{w}=\bar{w}\) for every \(\alpha \in \Pi\) and  \(c\in K\) if and only if \(e_{\alpha}\bar{w}=0\) for every \(\alpha \in \Pi.\) By Lemmas \ref{structure_constants_for_B} and \ref{technical_lemma}, we have
\begin{align*}
e_{\alpha_1}\bar{w} 	&= \bigg((a+1)B_1 + \sum_{i=2}^{k-2}{B_i}-A\bigg)F_{2,k}\bar{v}^{\lambda}	+ \bigg((a+1)C_1 + \sum_{i=2}^{k-2}{C_i}\bigg)f_{2,k-1}f_{\alpha_k}F_{k,k+1}\bar{v}^{\lambda} \cr
					&\;\;\;\;\,\,+ \sum_{j=k}^{n-1}{\bigg((a+1)D_{1j} + \sum_{i=2}^{k-2}{D_{ij}-E_j}\bigg)}f_{k,j}F_{2,j+1}\bar{v}^{\lambda}  - Ff_{k,n}f_{2,n}\bar{v}^{\lambda} \cr 
					&= \bigg( (a+1)B_1 +\sum_{i=2}^{k-2}{B_i}-A + F\bigg) F_{2,k}\bar{v}^{\lambda}  + \bigg( (a+1)C_1 +\sum_{i=2}^{k-2}{C_i} + F\bigg) f_{2,k-1}f_{\alpha_k}F_{k,k+1}\bar{v}^{\lambda} \cr
					&\;\;\;\;\,\,+ \sum_{j=k}^{n-1}{\bigg((a+1)D_{1j} + \sum_{i=2}^{k-2}{D_{ij}-E_j}\bigg)}f_{k,j}F_{2,j+1}\bar{v}^{\lambda},\cr
\end{align*}
where the last equality can be deduced from case \ref{technical_lemma_part_1} of Lemma \ref{technical_lemma}. Also, for $1<r<k-1,$ we get
\begin{align*}
e_{\alpha_r}\bar{w} = (B_r-B_{r-1})f_{1,r-1}F_{r+1,k}\bar{v}^{\lambda} 	&+ (C_r-C_{r-1})f_{1,r-1}f_{r+1,k-1}f_{\alpha_k}F_{k,k+1}\bar{v}^{\lambda} 														\cr
																																				&+ \sum_{j=k}^{n-1}{(D_{rj}-D_{r-1,j})}f_{1,r-1}f_{k,j}F_{r+1,j+1}\bar{v}^{\lambda},
\end{align*}
while case \ref{[Dn<Bn] Preliminaries for L of type Bn: technical proposition concerning the weight lambda-2...2 for lambda=lambda_1 (part 1)} of Proposition \ref{[Dn<Bn] Preliminaries for L of type Bn: technical proposition concerning the weight lambda-2...2 for lambda=lambda_1}  yields 
\begin{align*}
e_{\alpha_{k-1}}\bar{w} &= (2C_{k-2}-D_{k-2,k})f_{1,k-2}f_{\alpha_k}F_{k,k+1}\bar{v}^{\lambda} - \sum_{j=k+1}^{n-1}{D_{k-2,j}f_{1,k-2}f_{k,j}F_{k,j+1}\bar{v}^{\lambda}}			\cr
												&= \bigg(2C_{k-2}-\sum_{j=k}^{n-1}{D_{k-2,j}}\bigg)f_{1,k-2}f_{\alpha_k}F_{k,k+1}\bar{v}^{\lambda}.	
\intertext{Also}
e_{\alpha_k}\bar{w} 		&= \bigg(-A+2E_k +\sum_{j=k+1}^{n-1}{E_j}+2F\bigg)F_{1,k+1}\bar{v}^{\lambda} \cr
						&\;\;\;\;\;\,\,\,\,- \sum_{i=1}^{k-2}\bigg(B_i+C_i-2D_{i,k} - \sum_{j=k+1}^{n-1}{D_{i,j}}\bigg)f_{1,i}F_{i+1,k+1}\bar{v}^{\lambda},
\intertext{while for $k<s<n,$ we have}
e_{\alpha_s}\bar{w} 		&= \sum_{i=1}^{k-2}{(D_{is}-D_{i,s-1})f_{1,i}f_{k,s-1}F_{i+1,s+1}\bar{v}^{\lambda}} +(E_s-E_{s-1})f_{k,s-1}F_{1,s+1}\bar{v}^{\lambda}.
\intertext{Finally, thanks to Lemma \ref{technical_lemma} (case \ref{technical_lemma_part_4}), we see that}
e_{\alpha_n}\bar{w}	&= \sum_{i=1}^{k-2}{D_{i,n-1}f_{1,i}f_{i+1,n-1}f_{k,n}\bar{v}^{\lambda}} +(2F-E_{n-1})f_{k,n-1}f_{1,n}\bar{v}^{\lambda} + 2Ff_{1,n-1}f_{k,n}\bar{v}^{\lambda}\cr
					&= \sum_{i=1}^{k-2}{(D_{i,n-1}+E_{n-1}-2F)f_{1,i}f_{i+1,n-1}f_{k,n}\bar{v}^{\lambda}} + (4F-E_{n-1})f_{1,n-1}f_{k,n}\bar{v}^{\lambda},
\end{align*}
where the last equality follows from Lemma \ref{technical_lemma} (case  \ref{technical_lemma_part_3}). As usual, one  concludes that $e_{\alpha}\bar{w}(S)=0$ for every $\alpha \in \Pi$ if and only if  $p \mid 2(n-k)+1$ (showing that \ref{V(l,k,n)_part_1} and \ref{V(l,k,n)_part_3} are equivalent), in which case
\[
S\in \langle (4,\underbrace{n-k-1,\ldots,n-k-1}_{k-2},\underbrace{k-n,\ldots,k-n}_{k-2},\underbrace{-2,\ldots,-2}_{(n-k)(k-2)},\underbrace{4,\ldots,4}_{n-k},1)\rangle _K,
\]
thus completing the proof.
\end{proof}

Let $\lambda,$ and $\mu$ be as above, with $p \mid a + k,$  and consider an irreducible $KG$-module $V=L(\lambda)$ having highest weight $\lambda.$ As in the case where $k=2,$ take $V=V(\lambda)/\rad(\lambda),$ so that 
\[
V \cong \bigquotient{\overline{V(\lambda)}}{\overline{\rad(\lambda)}}, 
\]
where $\overline{\rad(\lambda)}=\rad(\lambda)/\langle G u^+\rangle_K.$ Also  write $v^+$ to denote the image of $\bar{v}^{\lambda}$ in $V,$ that is, $v^+$ is a maximal vector of weight \(\lambda\) in $V$ for $B.$  By Proposition \ref{Une premiere simplification de l'ensemble generateur de V(nu)}, the weight space $V_{\mu}$ is spanned by the vectors
\begin{equation}
\begin{aligned}
\left\{F_{1,k}v^+\right\} 									&\cup  \left\{f_{1,i}F_{i+1,k}v^+\right\}_{1\leqslant i\leqslant k-2} 														\cr
																						&\cup  \left\{f_{1,i}f_{i+1,k-1}f_{\alpha_k}F_{k,k+1}v^+\right\}_{1\leqslant i\leqslant k-2} 			\cr
																						&\cup  \left\{f_{1,i}f_{k,j}F_{i+1,j+1}v^+\right\}_{1\leqslant i\leqslant k-2,k\leqslant j <n} 				\cr
																						&\cup  \left\{f_{k,j}F_{1,j+1}v^+\right\}_{k\leqslant j<n} 																	\cr
																						&\cup  \left\{f_{k,n}f_{1,n}v^+\right\}.
\end{aligned}
\label{generators_for_V_nu}
\end{equation}

We write $V_{1,k,n}$ to denote the span of all the generators in \eqref{generators_for_V_nu} except for $f_{k,n}f_{1,n}v^+.$ As usual, the following result consists of a precise description of the weight space $V_{\mu},$ as well as a characterization for \(\mu\) to afford the highest weight of a composition factor of \(V(\lambda).\)


\begin{prop}\label{condition_if_l<k-1}
Let $G$ be a simple algebraic group of type $B_n$ over $K$ and consider an irreducible $KG$-module  $V=L(\lambda)$ having $p$-restricted highest weight $\lambda=a\lambda_1+\lambda_k,$ where  $a\in \Z_{>0},$ and $2<k<n$ are such that $p\mid a+k.$ Also set $\mu=\lambda-(\alpha_1+\cdots+\alpha_{k-1}+2\alpha_k+\cdots+2\alpha_n).$ Then  the following assertions are equivalent.
\begin{enumerate} 
\item \label{conditions_for_f(i,i)_to_belong_to_V(l,k,n)_part_1} The weight $\mu$ affords the highest weight of a composition factor of $V(\lambda).$
\item \label{conditions_for_f(i,i)_to_belong_to_V(l,k,n)_part_2} The generators in \eqref{generators_for_V_nu} are linearly dependent.
\item \label{conditions_for_f(i,i)_to_belong_to_V(l,k,n)_part_3} The element $f_{k,n}f_{1,n}v^+$ lies inside $V_{1,k,n}.$
\item \label{conditions_for_f(i,i)_to_belong_to_V(l,k,n)_part_4} The divisibility condition $p\mid 2(n-k)+1$ is satisfied.
\end{enumerate}
\end{prop}

\begin{proof}
Proceed exactly as in the proof of  Proposition \ref{condition_if_l=k-1}, replacing $\mu_{1,2}$ by $\mu_{1,k},$ Lemma \ref{[Dn<Bn] Preliminaries for L of type Bn: reduction to the study of V(lambda)/<Lu> for lambda = a lambda1 + lambda2} by Lemma \ref{[Dn<Bn] Preliminaries for L of type Bn: reduction to the study of V(lambda)/<Lu> for lambda = a lambda_1 + lambda_k}, and Lemma \ref{Conditions_pour_que_la_combinaison_lineaire_de_vecteurs_generateurs_de_V(k-1,k,n)_soit_tuee_par_tous_les_e_alpha} by Lemma \ref{Conditions_pour_que_la_combinaison_lineaire_de_vecteurs_generateurs_de_V(l,k,n)_soit_tuee_par_tous_les_e_alpha}.
\end{proof}

\end{document}